\documentclass[12pt,reqno]{amsart}
\usepackage{amsmath,amsthm,amscd,amssymb,bbm,graphicx}
\usepackage{floatflt,setspace, wrapfig}
\usepackage{xcolor}
\usepackage{mathtools}
\usepackage{graphics,multicol}
\usepackage{mathrsfs} 
\usepackage{hyperref}
\usepackage{geometry}
\usepackage{float}
\usepackage{graphicx}
\usepackage{enumitem}
\usepackage{mathabx}
\usepackage{nicefrac}
\usepackage{stackrel}
\usepackage[square,numbers]{natbib}
\usepackage{centernot}
\usepackage{stmaryrd}

\makeatletter
\newcommand{\xMapsto}[2][]{\ext@arrow 0599{\Mapstofill@}{#1}{#2}}
\def\Mapstofill@{\arrowfill@{\Mapstochar\Relbar}\Relbar\Rightarrow}
\makeatother

\DeclareMathOperator*{\esssup}{ess\,sup}
\DeclareMathOperator*{\essinf}{ess\,inf}

\newtheorem{thrm}{Theorem}

\newtheorem{cor}{Corollary}
\newtheorem{hyp}{H\hspace{-1.5mm}}
\newtheorem{subhyp}{\hspace{2.5mm}}[hyp]

\newtheorem{cond}{C\hspace{-1.5mm}}
\newtheorem{lem}{Lemma}
\newtheorem{dfn}{Definition}
\newtheorem{remark}{Remark}

\newcommand{\mybigint}{\scalebox{1.2}[3]{$\displaystyle \int$}}

\newtheorem*{notation}{Notation}

\newcommand{\id}{\mbox{id}}

\newlist{statistics}{enumerate}{1}
\setlist[statistics]{%
  label=\Roman*),
  resume
}

\def\T{\hat{T}}

\newcommand{\bea}[1]{\begin{eqnarray}\label{#1}}
\newcommand{\eea}{\end{eqnarray}}

\title[Compound Poisson distributions for random dynamical systems]{Compound Poisson distributions for random dynamical systems using  probabilistic approximations}

\begin{document}
 \author{Lucas Amorim}
 \thanks{L Amorim, University of Porto, Mathematics Department, Center for Mathematics, Porto, Portugal; Universit\'e de Toulon, CPT, Marseille, France.
 E-mail: {\tt {lucas.amorim.vb@gmail.com}}}
\author{Nicolai Haydn}
\thanks{N Haydn, Mathematics Department, USC,
Los Angeles, 90089-2532. E-mail: {\tt {nhaydn@usc.edu}}}
 \author{Sandro Vaienti}
 \thanks{S Vaienti, Aix Marseille Universit\'e, Universit\'e de Toulon, CNRS, CPT, UMR 7332, Marseille, France.
 E-mail: {\tt {vaienti@cpt.univ-mrs.fr}}}

\date{\today}



\begin{abstract}
We obtain quenched hitting distributions to be compound Poissonian for a certain class of random dynamical systems. The theory is general and designed to accommodate non-uniformly expanding behavior and targets that do not overlap much with the region where uniformity breaks. Based on annealed and quenched polynomial decay of correlations, our quenched result adopts annealed Kac-type time-normalization and finds limits to be noise-independent. The technique involves a probabilistic block-approximation where the quenched hit-counting function up to annealed Kac-normalized time is split into equally sized blocks which are mimicked by an independency of random variables distributed just like each of them. The theory is made operational due to a result that allows certain hitting quantities to be recovered from return quantities. Our application is to a class of random piecewise expanding one-dimensional systems, casting new light on the well-known deterministic dichotomy between periodic and aperiodic points, their usual extremal index formula $\operatorname{EI}= 1 - 1/ JT^p(x_0)$, and recovering the Polya-Aeppli case for general Bernoulli-driven systems, but distinct behavior otherwise. Future and on-going investigations aim to produce and accommodate examples of bonafide non-uniformly expanding random systems and targets approaching their neutral points.
\end{abstract}

\maketitle

\newpage
\tableofcontents

\newpage

\section{Introduction}

Limiting hitting distributions and hitting time statistics of dynamical systems, together with their return counterparts, and the related quantitative recurrence questions, have a long history of investigation. This investigation remains active and in the last few years has advanced in many different directions, such as more elaborate targets, non-uniformly hyperbolic behavior, random systems, and connections to extreme behavior, both in theory and real-life applications.  

In the deterministic case, the canonical picture is presented for uniformly hyperbolic or expanding systems with singleton targets and Kac-type normalization, where a dichotomy occurs: either the target consists of a non-periodic generic point and the limit behavior is pure Poisson (see e.g., \cite{doeblin1940remarques}, \cite{pitskel1991poisson}, \cite{hirata1993poisson}, 
\cite{collet1996some},
\cite{galves1997inequalities}, \cite{denker2004poisson}), or the target consists of a periodic point and the limit behavior is Polya-Aeppli (see e.g., \cite{hirata1993poisson}, \cite{haydn2009compound},
\cite{freitas2013compound},
\cite{kifer2014poisson}, \cite{carvalho2015extremal}). The so-called extremal index (EI) can summarize both cases: in the pure Poisson case $EI=1$, whereas in the Polya-Aeppli case $EI = 1 - 1 / JT^p(x_0) \in (0,1)$. 

A direction of generalization found in the literature is to consider different types of targets, not limited to singletons. In general, this situation exhibits limiting hitting distributions in the compound Poisson class, which includes, but isn't limited to, the pure Poisson and Polya-Aeppli cases. This can be seen most simply in the case of finite targets with pieces of orbits (\cite{azevedo2016clustering}, \cite{holland2012extreme} and \cite{abadi2020dynamical}), but more complicated situations were also studied, such as countable sets (\cite{azevedo2017extreme}), submanifolds (\cite{faranda2018extreme}, \cite{carney2021extremes}) and fractal sets (\cite{freitas2020rare}, \cite{mantica2016extreme}, \cite{freitas2021rare}). More abstract approaches to such general target sets were developed in \cite{freitas2020enriched} and \cite{haydn2020limiting}.

Another main direction of generalization is to handle non-uniformly expanding behavior. Many contributions have been given in the literature, such as \cite{holland2012extreme}, \cite{freitas2016rare}, \cite{freitas2013compound}, \cite{freitas2020enriched} and \cite{haydn2020limiting}. We emphasize that the relation between the target position and the position of the neutral fixed points of such maps plays a major role, because, when they intersect, strong dependence/recurrence around the target occurs, requiring special normalization as to find non-trivial limiting distributions (see, e.g, \cite{freitas2016rare}).

Finally, the theory has also been generalized to the realm of random dynamical systems, see, for example \cite{marie2009recurrence}, \cite{rousseau2014exponential}, \cite{aytacc2015laws}, \cite{haydn2016return}, \cite{freitas2020point}, \cite{crimmins2022quenched} and, very recently, \cite{atnip2023compound}. Compound Poissonian quenched hitting distributions were also shown in \cite{atnip2023compound} with the spectral method. Despite their applications being similar to ours, the main differences are that their theory needs exponential decay of correlations, and their time-normalization is quenched. Quenched time-normalization usually stands with a merely ergodic driving system, which is the case there. However, from an applied point of view, this can be an impractical restriction: quenched time-normalization says that the experimenter will not pre-determine (deterministically) how long to watch the experiment, but will get informed about the complete noise realization (at least until its remote past) and use it to determine the desired watch time.

We now discuss the contributions of this work and some of its features.

We show that quenched hitting distributions are compound Poissonian for a certain class of random dynamical systems, using a probabilistic block-approximation approach and generalizing the deterministic theory developed in \cite{haydn2020limiting} after the approach introduced in \cite{chazottes2013poisson}. This is the content of theorem \ref{thm:main}, our main result.

The probabilistic block-approximation (theorem \ref{thm:approx}) splits the quenched hit-counting function up to annealed-Kac-normalized into equally sized blocks which are mimicked by an independency of random variables distributed just like each of them. The said approximation goes for any given noise realization $\omega$ and $\omega$-dependent leading terms and errors appear. Both of them are tamed by an almost sure convergence statement (lemma \ref{lem:asconv}) based on a Borel-Cantelli argument, which allows for the quenched result to hold.

The limiting compound Poisson distribution, revealed by the asymptotics of the aforementioned leading terms, and its underlying multiplicity distribution are characterized by a set of hitting quantities ($\lambda_\ell$'s), which are transparently expressed in terms of the asymptotics of the dynamics, its invariant measure and the target. Hitting quantities are introduced in section \ref{sec:usestatquant}

The theory is made operational due to theorem \ref{thm:lambdaalpha}, which allows for the latter hitting quantities to be recovered from a set of return quantities ($\alpha_\ell$'s). Return quantities are introduced in section \ref{sec:usestatquant}. The advantage here is that the return quantities are easier to calculate in concrete examples.

Moreover, our theory is based on a mild set of hypotheses, introduced in section \ref{sec:wrksetup}, designed to accommodate non-uniformly expanding behavior (with polynomial decay) and general targets that do not overlap much with the region where uniformity breaks and that presents well-defined return quantities. 

Our assumptions on the quasi-invariant family of measures do not consider their absolute continuity with respect to the Lebesgue measure, but regularity in a dimensional sense. 

A drawback of our approach is that results are just along sufficiently fast shrinking neighborhoods of the target set. This is intimately connected with the use of a Borel-Cantelli argument well-adapted to the annealed time normalization and annealed decay of correlations.

We conclude with an application to a class of random piecewise expanding one-dimen-sional systems, casting new light on the well-known deterministic dichotomy between periodic and aperiodic points, their typical extremal index formula $\operatorname{EI}= 1 - 1/ JT^p(\zeta)$, and recovering the geometric case for general Bernoulli-driven systems, but distinct behavior otherwise. See section \ref{sec:ex}.

\section{Assumptions and main results} \label{sec:assumptionsresults}
\subsection{General setup}\label{sec:gensetup}
Consider $M$ and $\Omega$ complete separable metric spaces and $(\theta,\mathbb{P})$ a measurably invertible ergodic system on $\Omega$.

Consider maps $T_\omega: M \to M$ ($\omega \in \Omega$) which combine to make the a measurable skew product $S: \Omega \times M \rightarrow \Omega \times M$, $(\omega,x) \mapsto (\theta \omega, T_\omega x)$. As usual, for higher-order iterates we denote 
$S^n(\omega,x)=(\theta^n\omega, T_\omega^n(x))$ where $T_\omega^n=T_{\theta^{n-1}\omega}\circ \cdots \circ
T_{\theta\omega}\circ T_\omega$ $(n \ge 1)$.

For $E \in \mathscr{B}_\Omega \times \mathscr{B}_M$ and $\omega \in \Omega$, write $E(\omega) = \{ x \in M : (\omega, x) \in E\}$. Denote
\begin{equation*}
    \mathcal{P}^{\mathbb{P}}(\Omega \times M) = \{ \hat{\mu} \in \mathcal{P}(\Omega \times M) : \hat{\mu}, {\pi_{\Omega}}_* \hat{\mu} = \mathbb{P}  \},
\end{equation*}
\begin{equation*}
    \mathcal{P}_S^{\mathbb{P}}(\Omega \times M) = \{ \hat{\mu} \in \mathcal{P}(\Omega \times M) : S_*\hat{\mu} = \hat{\mu}, {\pi_{\Omega}}_* \hat{\mu} = \mathbb{P}  \},
\end{equation*}
and 
\begin{equation*}
    \mathcal{RP}^{(\mathbb{P})}(M) {=} \hspace{-0.5mm}\left\{ \hspace{-1mm}\substack{\displaystyle \mu: \omega \in \Omega \stackrel{\mathbb{P}\text{-a.s.}}{\mapsto} \mu_\omega \in \mathcal{P}(M) \text{ so that:} \\ \displaystyle \text{ }\omega \in \Omega \stackrel{\mathbb{P}\text{-a.s.}}{\mapsto} \mu_\omega(E(\omega)) \in [0,1] \text{ is }(\mathscr{B}_\Omega,\mathscr{B}_{[0,1]})\text{-measurable, }\forall E \in \mathscr{B}_{\Omega} \times \mathscr{B}_M} \hspace{-1mm}\right\},
\end{equation*}
\begin{equation*}
    \mathcal{RP}^{(\mathbb{P})}_{T}(M) {=} \hspace{-0.5mm} \left\{ \hspace{-1mm}\substack{\displaystyle \mu: \omega \in \Omega \stackrel{\mathbb{P}\text{-a.s.}}{\mapsto} \mu_\omega \in \mathcal{P}(M) \text{ so that:} \\ \displaystyle \text{ }\omega \in \Omega \stackrel{\mathbb{P}\text{-a.s.}}{\mapsto} \mu_\omega(E(\omega)) \in [0,1] \text{ is }(\mathscr{B}_\Omega,\mathscr{B}_{[0,1]})\text{-measurable, }\forall E \in \mathscr{B}_{\Omega} \times \mathscr{B}_M\\ \displaystyle \text{ }\mathbb{P}\text{-a.s., } \forall n \ge 0: {T_\omega^n}_* \mu_\omega = \mu_{\theta^n \omega} } \hspace{-1mm} \right\}.
\end{equation*}

\begin{notation}
Elements in the latter two sets will be written as $\mu = (\mu_\omega)_{\omega}$, where the outer `$\omega$' subscript (instead of `$\omega \in \Omega$') is to identify that the given family if defined $\mathbb{P}$-a.s.. The underlying full measure subset $\Omega_0$ can be assumed to be forward and backward $\theta$-invariant (otherwise we substitute it by $\bigcap_{n \in \mathbb{ Z}} \theta^{n} \Omega_0$).
\end{notation}

Any $\hat{\mu}$ in $\mathcal{P}^{\mathbb{P}}(\Omega \times M)$ (in $\mathcal{P}_S^{\mathbb{P}}(\Omega \times M)$) rewrites (disintegrates) as 
\begin{equation*}\label{eq:disintegration}
    \hat{\mu}(E) = \int_\Omega \mu_\omega(E(\omega)) d\mathbb{P}(\omega),
\end{equation*}where $(\mu_\omega)_\omega$ is in $\mathcal{RP}^{(\mathbb{P})}(M)$ (in $\mathcal{RP}^{(\mathbb{P})}_{T}(M)$). Conversely, given $(\mu_\omega)_\omega$ in $\mathcal{RP}^{(\mathbb{P})}(M)$ (in $\mathcal{RP}^{(\mathbb{P})}_{T}(M)$), equation (\ref{eq:disintegration}) defines $\hat\mu$ in $\mathcal{P}^{\mathbb{P}}(\Omega \times M)$ (in $\mathcal{P}_S^{\mathbb{P}}(\Omega \times M)$). See \cite{crauel2002random} (prop. 3.3) and \cite{arnold1995random} (sec. 1.4).

Now, consider a given $\hat\mu = d\mu_\omega d\mathbb{P}(\omega) \in \mathcal{P}_S^{\mathbb{P}}(\Omega \times M)$, with the associated $(\mu_\omega)_{\omega } \in \mathcal{RP}^{(\mathbb{P})}_{T}(M)$. Define the marginal measure  $ \check{\mu} = {\pi_M}_* \hat\mu = \int_\Omega\mu_\omega\,d\mathbb{P}(\omega)\in \mathcal{P}(M)$. 

Finally, consider $\Gamma \in \mathscr{B}_\Omega \times \mathscr{B}_M$ so that, $\mathbb{P}$-a.s, $\Gamma(\omega)$ is compact and so that  $\mu_\omega(\Gamma(\omega))=0$. The set $\Gamma$ is the so-called random target. Denote $\Gamma_\rho(\omega)=B_\rho(\Gamma(\omega))$ ($\rho > 0$) and the corresponding $\omega$-collection by $\Gamma_\rho \in \mathscr{B}_\Omega \times \mathscr{B}_M$.

The objects above comprise what we call a `system', denoted by $(\theta, \mathbb{P}, T_\omega, \mu_\omega, \Gamma)$.

\subsection{Preliminary definitions}
We now define some working objects. 

Let $U \in \mathscr{B}_\Omega \times \mathscr{B}_M$ be so that $\mu_\omega(U(\omega))>0$, $\mathbb{P}$-a.s.. We again can consider that $\mu_{\theta^n \omega}(U(\theta^n \omega))>0$, for all $n \in \mathbb{Z}$, $\mathbb{P}$-a.s.. 

\begin{dfn}The \textbf{first hitting time} of $(\theta,\mathbb{P},T_\omega,\mu_\omega,U)$ is the family of functions 
\begin{equation*}
    \begin{tabular}{c c c l}
       $r_U^{\omega,1}:$  & $M$ & $\rightarrow$ & $\mathbb{N}_{\ge 1} \cup \{\infty\}$  \\
         & $x$ & $\mapsto$ & $\inf\{ i \in \mathbb{N}_{\ge 1} : T^i_\omega (x) \in U(\theta^i \omega) \}$
    \end{tabular}.
\end{equation*}The associated \textbf{higher-order hitting times} are given, for $\ell \ge 2$, by the family of functions
\begin{equation*}
        \begin{tabular}{r c c c l}
       $r_U^{\omega,\ell}$ & $:$  & $M$ & $\rightarrow$ & $\mathbb{N}_{\ge \ell} \cup \{\infty\}$  \\
         & & $x$ & $\mapsto$ & $r_U^{\omega,\ell}(x)=r_U^{\omega,\ell-1}(x)+r_U^{\omega'}(T_{\omega}^{r_U^{\omega,\ell-1}}(x))$
    \end{tabular},
\end{equation*}where $\omega'=\theta^{r_U^{\omega,\ell-1}(x)}\omega$.
\end{dfn}

\begin{dfn}\label{def:hitcount}
    The \textbf{hit counting function} of $(\theta,\mathbb{P},T_\omega,\mu_\omega,U)$ with noise $\omega \in \Omega$ and up time $L \ge 1$ is given by
    \begin{equation*}
    \begin{tabular}{r @{} c @{} c @{} c @{} l}
       $Z_{*U}^{\omega,L}$ & $:$  & $M$ & $\rightarrow$ & $\mathbb{N}_{\ge 0}$  \\
         & & $x$ & $\mapsto$ & $\displaystyle \sum_{i=1}^L \mathbbm{1}_{U(\theta^i \omega)} \circ T_\omega^i(x)$
    \end{tabular}, \hspace{3mm}
    \begin{tabular}{r @{} c @{} c @{} c @{} l}
       $\text{ }Z_{U}^{\omega,L}$ & $:$  & $M$ & $\rightarrow$ & $\mathbb{N}_{\ge 0}$  \\
        & & $x$ & $\mapsto$ & $\displaystyle \sum_{i=0}^{L-1} \mathbbm{1}_{U(\theta^i \omega)} \circ T_\omega^i(x)$
    \end{tabular}.
\end{equation*}
\end{dfn}

These objects are related, for example, in the sense that $\{Z^{\omega,L}_{*U}\ge \ell \}=\{r_{U}^{\omega,\ell}\le L\}$, 
 $\{Z^{\omega,L}_{*U}=\ell\}=\{r_{U}^{\omega,\ell}\le L<r_{U}^{\omega,\ell+1}\}$. When $U = \Gamma_\rho$, we write $I^{\omega,\rho}_i=\mathbbm{1}_{\Gamma_\rho(\theta^i\omega)}\circ T_\omega^i$.

 \begin{dfn}\label{def:hitmark}
     The \textbf{hit marking function} of $(\theta,\mathbb{P},T_\omega,\mu_\omega,U)$ with noise $\omega \in \Omega$ and up time $L \ge 1$ is given by
     \begin{equation*}
        \begin{tabular}{r @{} c @{} c @{} c @{} l}
       $\text{ }Y_{U}^{\omega,L}$ & $:$  & $M$ & $\rightarrow$ & $\mathfrak{M}$  \\
        & & $x$ & $\mapsto$ & $\displaystyle \sum_{i=0}^{L-1} \delta_{\nicefrac{i}{L}}\mathbbm{1}_{U(\theta^i \omega)} \circ T_\omega^i(x)$
        \end{tabular},
     \end{equation*}where $\mathfrak{M}=\{ \sum_{i=1}^\kappa \delta_{x_i} : \kappa < \infty, (x_i)_{i=1}^\kappa \subset [0,1] \}$\footnote{\label{note:topologies}The set $\mathfrak{M}$ can be given the vague topology (with $C^+_K([0,1])$ test functions, see \cite{resnick2008extreme} section 3.4), making it a complete separable metric space, while $\mathcal{P}(\mathfrak{M})$ is another topological space with the weak topology (with $C_b^+(\mathfrak{M})$ test functions, see \cite{resnick2008extreme} section 3.5).}.
 \end{dfn}

\begin{notation}
    A $\mathbb{R}$-valued function defined on the product space, $f(\omega,x)$, is often rewritten as $f^\omega(x)$ or $f_\omega(x)$ and seen as a family of functions defined on $M$. And vice versa. When integrating a function, we may omit the variable of integration, even if it is a sup/subscript. We leave it for the reader to infer what variables and parameters are being integrated and were omitted. 
\end{notation}

\begin{notation}
    Consider non-negative sequences $a(n)$ and $b(n)$ ($n \ge 0$). we will write $a(n) \lesssim_n b(n)$ to mean that there exists a quantity $C>0$, independent of $n$, so that $a(n) \le C b(n) \text{ }(\forall n \ge 0)$. When $a$ and $b$ have more arguments, we indicate which of them are controlled uniformly. For example:\\i) $a(n,m) \lesssim_n b(n,m)$ when there exists $C_m>0$ so that $a(n,m) \le C_m b(n,m)$ ($\forall n, m \ge 0$),\\   
    ii) $a(n,m) \lesssim_{n,m} b(n,m)$ when there exists $C>0$ so that $a(n,m) \le C b(n,m)$ ($\forall n, m \ge 0$).
    
When some of the arguments are taken to the limit, we implicitly consider that these are the ones being controlled uniformly and we omit the associated subscripts from the $\lesssim$ symbol. 
We also employ the usual big-O and little-o notation.    
\end{notation}

\begin{dfn}\label{def:cpd}
    The compound Poisson distribution with intensity parameter $s \in \mathbb{R}_{>0}$ and cluster size distribution $(\lambda_\ell)_{\ell \in \mathbb{N}_{\ge 1}} \in \mathcal{P}(\mathbb{N}_{\ge 1})$, $\sum_{\ell =1}^\infty \ell \lambda_\ell < \infty$, denoted $\operatorname{CPD}_{s,(\lambda_\ell)_\ell} \in \mathcal{P}(\mathbb{N}_{\ge 0})$, is the distribution of a random variable $M : (\mathcal{X}, \mathscr{X}, \mathbb{Q}) \rightarrow \mathbb{N}_{\ge 0}$ given by $
M(\xi) = \sum_{j=1}^{N(\xi)} Q_j(\xi)$, where $N$ is a $\mathbb{N}_{\ge 0}$-valued random variable on $(\mathcal{X}, \mathscr{X}, \mathbb{Q})$ having Poisson distribution with intensity parameter $s$ and $(Q_j)_{j \in \mathbb{N}_{\ge 1}}$ is a sequence of $\mathbb{N}_{\ge 1}$-valued random variables on $(\mathcal{X}, \mathscr{X}, \mathbb{Q})$ which are iid, independent of $N$ and whose entries have distribution $\mathbb{Q}(Q_j =\ell) = \lambda_\ell$ ($j,\ell \in \mathbb{N}_{\ge 1}$). Denote $R_l = \sum_{j=1}^l Q_j$. Then the probability mass function of $\operatorname{CPD}_{s,(\lambda_\ell)_\ell}$ is given indirectly by \begin{equation}\label{eq:cpddensity}
\begin{tabular}{c c l}
    $\displaystyle \operatorname{CPD}_{\gamma,(\lambda_\ell)_\ell}(n)$ & $=$ & $\displaystyle \sum_{l=1}^n \mathbb{P}(N=l) \mathbb{P}(R_l=n) = \sum_{l=1}^n \frac{\displaystyle s^l e^{-s}}{\displaystyle l!} \sum_{\substack{(n_1,\ldots,n_l) \in \mathbb{N}_{\ge 1}^l \\ n_1 + \ldots + n_l =n}} \prod_{i=1}^l \lambda_{n_i}$.
\end{tabular}
\end{equation}
\end{dfn} 

\begin{dfn}\label{def:cppp}
    The compound Poisson point process with intensity parameter $s \in \mathbb{R}_{>0}$ and cluster size distribution $(\lambda_\ell)_{\ell \in \mathbb{N}_{\ge 1}} \in \mathcal{P}(\mathbb{N}_{\ge 1})$, $\sum_{\ell =1}^\infty \ell \lambda_\ell < \infty$, denoted $\operatorname{CPPP}_{s,(\lambda_\ell)_\ell} \in \mathcal{P}(\mathfrak{M})$, is the distribution of a random variable $N : (\mathcal{X}, \mathscr{X}, \mathbb{Q}) \rightarrow \mathfrak{M}$ that satisfies:

    \begin{itemize}
    \item[-] $\forall (F_1,\ldots,F_k) \subset \mathscr{B}_{[0,1]}$ mutually disjoint, $( N(\cdot)(F_i) )_{i=1}^k$ is iid,
    \item[-] $\forall F \in \mathscr{B}_{[0,1]}$, ${N(\cdot) (F)}_* \mathbb{Q} = CPD_{s \operatorname{Leb}(F), (\lambda_\ell)_{\ell } }$.
\end{itemize}
\end{dfn}

\subsection{Statistical quantities}\label{sec:usestatquant}
\begin{notation}
    Write $\lim\limits_{L \to \infty} \overline{\varliminf\limits_{\rho \to 0}} a(L,\rho)$ for the value of $\lim\limits_{L \to \infty} \varlimsup\limits_{\rho \to 0} a(L,\rho)$ and $\lim\limits_{L \to \infty} \varliminf\limits_{\rho \to 0} a(L,\rho)$, when they do exist and coincide. Denote also $\stackrel{+}{a} \hspace{-1mm}(L) := \varlimsup\limits_{\rho \to 0} a(L,\rho)$ and $\stackrel{-}{a}\hspace{-1mm}(L) := \varliminf\limits_{\rho \to 0} a(L,\rho)$.    
\end{notation}

We now introduce a few quantities that play a major role in the theory. Those denoted with a `$\lambda$' are hitting quantities, and those with an `$\alpha$' are return quantities. 
Whenever the following limits exist (and the appropriate ones coincide), denote, for $\ell \ge 1$ and $\omega \in \Omega$:

\begin{statistics}
    \item \begin{equation}\label{eq:lambdaomegadef}
\lambda^\omega_\ell = \lim\limits_{L \to \infty} \overline{\varliminf\limits_{\rho \to 0}} \lambda_\ell^\omega(L,\rho)
\end{equation}where
\begin{equation}\label{eq:deftildelambdaomegainner}
\begin{tabular}{c}
$\displaystyle \lambda^\omega_\ell(L,\rho) = \mu_\omega(Z^{\omega,L}_{\Gamma_\rho} =  \ell | Z^{\omega,L}_{\Gamma_\rho} > 0) = \frac{\mu_\omega(Z^{\omega,L}_{\Gamma_\rho} =  \ell)}{\mu_\omega( Z^{\omega,L}_{\Gamma_\rho} > 0)}.$ 
\end{tabular}
\end{equation}

\item \begin{equation}\label{eq:lambdadef}
\lambda_\ell = \lim\limits_{L \to \infty} \overline{\varliminf\limits_{\rho \to 0}} \lambda_\ell(L,\rho)
\end{equation}where
\begin{equation}\label{eq:deftildelambdainner}
\lambda_\ell(L,\rho) = \hat\mu(Z^{L}_{\Gamma_\rho} =  \ell | Z^{L}_{\Gamma_\rho} > 0) = \frac{\hat\mu(Z^{L}_{\Gamma_\rho} =  \ell)}{\hat\mu( Z^{L}_{\Gamma_\rho} > 0)} = \int_\Omega \lambda^\omega_\ell(L,\rho) \frac{\mu_\omega(Z^{\omega,L}_{\Gamma_\rho} >0) }{\int_\Omega \mu_\omega(Z^{\omega,L}_{\Gamma_\rho} >0) 
 d \mathbb{P}(\omega)} d \mathbb{P}(\omega).
\end{equation}

    \item \begin{equation}\label{eq:hatalphaomegadef}
\hat\alpha^\omega_\ell = \lim\limits_{L \to \infty} \overline{\varliminf\limits_{\rho \to 0}} \hat\alpha_\ell^\omega(L,\rho)\footnotemark
\end{equation}\footnotetext{\label{note:Lmonotone}Notice that, by $L$-monotonicity, the outer limits always exist provided that the inner ones do.}where
\begin{equation}\label{eq:deftildealphahatomegainner}
\hat\alpha_\ell^\omega(L,\rho) = \mu_\omega(Z^{\omega,L}_{\Gamma_\rho} \ge \ell  | I_0^{\omega,\rho}=1) = \frac{\mu_\omega(Z^{\omega,L}_{\Gamma_\rho} \ge \ell  , I_0^{\omega,\rho}=1)}{\mu_\omega(\Gamma_\rho(\omega))}
\end{equation}

    \item \begin{equation}\label{eq:alphaomegadef}
\alpha^\omega_\ell = \lim\limits_{L \to \infty} \overline{\varliminf\limits_{\rho \to 0}} \alpha_\ell^\omega(L,\rho)
\end{equation}where
\begin{equation}\label{eq:deftildealphaomegainner}
 \alpha_\ell^\omega(L,\rho) = \mu_\omega(Z^{\omega,L}_{\Gamma_\rho} = \ell | I_0^{\omega,\rho}=1) = \frac{\mu_\omega(Z^{\omega,L}_{\Gamma_\rho} = \ell, I_0^{\omega,\rho}=1)}{\mu_\omega(\Gamma_\rho(\omega))}.
\end{equation}
\end{statistics}

Since $\{Z^{\omega,L}_{\Gamma_\rho} \ge \ell \} \supset  \{Z^{\omega,L}_{\Gamma_\rho} \ge \ell+1 \}$ and $\{Z^{\omega,L}_{\Gamma_\rho} \ge \ell \} \setminus \{Z^{\omega,L}_{\Gamma_\rho} \ge \ell+1 \} = \{Z^{\omega,L}_{\Gamma_\rho} = \ell \}$, then 
\begin{equation}\label{eq:deftildealphaomegainner2}
\hat\alpha^\omega_{\ell}(L,\rho)-\hat\alpha^\omega_{\ell+1}(L,\rho)=\alpha^\omega_\ell(L,\rho).
\end{equation}
which entails that the existence of $\hat\alpha^\omega_\ell$'s implies that of the $\alpha^\omega_\ell$'s with $\alpha^\omega_\ell = \hat\alpha^\omega_\ell - \hat\alpha^\omega_{\ell+1}$.

\begin{statistics}
\item \begin{equation}\label{eq:hatalphadef}
\hat\alpha_\ell = \lim\limits_{L \to \infty} \overline{\varliminf\limits_{\rho \to 0}} \hat\alpha_\ell(L,\rho)\footnotemark,
\end{equation}\footnotetext{See footnote \ref{note:Lmonotone}.}where
\begin{equation}
\hat\alpha_\ell(L,\rho) \hspace{-0.5mm} = \hspace{-0.5mm} \hat{\mu}(Z^{L}_{\Gamma_\rho} \ge \ell | I_0^\rho =1 ) \hspace{-0.5mm} = \hspace{-0.5mm} \frac{\hat{\mu}(Z^{L}_{\Gamma_\rho} \ge \ell, I_0^\rho =1 )}{\hat{\mu}(\Gamma_\rho )} \hspace{-0.5mm} = \hspace{-0.5mm} \int_\Omega \hat\alpha^\omega_\ell(L,\rho)\frac{\mu_\omega(\Gamma_\rho (\omega))} {\int_\Omega \mu_\omega(\Gamma_\rho (\omega)) d \mathbb{P}(\omega) } d \mathbb{P}(\omega).
\end{equation}

\item \begin{equation}\label{eq:alphadef}
\alpha_\ell = \lim\limits_{L \to \infty} \overline{\varliminf\limits_{\rho \to 0}} \alpha_\ell(L,\rho) 
\end{equation}where
\begin{equation}\label{eq:deftildealphainner}
\alpha_\ell(L,\rho) \hspace{-0.5mm} = \hspace{-0.5mm} \hat{\mu}(Z^{L}_{\Gamma_\rho}  = \ell | I_0^\rho =1 ) \hspace{-0.5mm} = \hspace{-0.5mm} \frac{\hat{\mu}(Z^{L}_{\Gamma_\rho} = \ell, I_0^\rho =1 )}{\hat{\mu}( \Gamma_\rho )}
 \hspace{-0.5mm} = \hspace{-0.5mm} \int_\Omega \alpha^\omega_\ell(L,\rho)\frac{\mu_\omega(\Gamma_\rho (\omega))} {\int_\Omega \mu_\omega(\Gamma_\rho (\omega)) d \mathbb{P}(\omega) } d \mathbb{P}(\omega).
\end{equation}
\end{statistics}

Since $\{Z^{L}_{\Gamma_\rho} \ge \ell \} \supset  \{Z^{L}_{\Gamma_\rho} \ge \ell+1 \}$ and $\{Z^{L}_{\Gamma_\rho} \ge \ell \} \setminus \{Z^{L}_{\Gamma_\rho} \ge \ell+1 \} = \{Z^{L}_{\Gamma_\rho} = \ell \}$, then 
\begin{equation}
\hat\alpha_{\ell}(L,\rho)-\hat\alpha_{\ell+1}(L,\rho)=\alpha_\ell(L,\rho).
\end{equation}
which entails that the existence of $\hat\alpha_\ell$'s implies that of the $\alpha_\ell$'s with $\alpha_\ell = \hat\alpha_\ell - \hat\alpha_{\ell+1}$.

\subsection{Working setup}\label{sec:wrksetup}

Now we particularize the general setup of section \ref{sec:gensetup} to specify our working setup.

So we consider a system $(\theta, \mathbb{P}, T_\omega, \mu_\omega, \Gamma)$ satisfying the following hypotheses.

\begin{hyp}[Ambient]\label{hyp:amb} Let $M$ be a compact Riemannian manifold and $\Omega$ a compact metric space.   
\end{hyp}

\begin{hyp}[Invertibility features] \label{hyp:inv}
\end{hyp}
\begin{subhyp}[Degree] \label{hyp:deg}
    $\forall \omega \in \Omega, \forall n \ge 1$, $\forall x \in M: \# (T_\omega^n)^{-1}(\{x\}) < \infty$ with
    \begin{equation*}
    \sup_{ n \ge 0} \# (T_\omega^n)^{-1}(\{x\}) {\le} \infty \hspace{0.5mm}(\forall \omega, x),  \sup_{ \omega \in \Omega} \# (T_\omega^n)^{-1}(\{x\}) {\le} \infty \hspace{0.5mm}(\forall n, x),\sup_{ x \in M} \# (T_\omega^n)^{-1}(\{x\}) {<} \infty \hspace{0.5mm} (\forall \omega, n).
    \end{equation*}
\end{subhyp}
\begin{subhyp}[Covering] \label{hyp:cov}
    $\exists R >0, \mathcal{N} \ge 1, \forall \omega \in \Omega, \forall n \ge 1, \exists (y^{\omega,n}_k)_{k \in K_{\omega,n} } \subset M$ with $\# K_{\omega,n}<\infty$ so that $(B_R( y^{\omega,n}_k))_{k \in K_{\omega,n}}$ has at most $\mathcal{N}$ overlaps.
\end{subhyp}

Terminology suggests that $(B_R( y^{\omega,n}_k) )_{k \in K_{\omega,n}}$ covers $M$ entirely, but a small defect is allowed, in the sense of (H\ref{hyp:largecov}) below. 
\begin{subhyp}[Inverse branches] \label{hyp:ib}
    $\forall \omega \in \Omega, \forall n \ge 1, \forall k \in K_{\omega,n}$, $$\operatorname{IB}^{\omega,n}_k = \{ \varphi: B_R( y^{\omega,n}_k) \rightarrow M \text{ diffeomorphic onto its image with } T_\omega^n \circ \varphi = \operatorname{id} \}$$is non-empty, finite\footnote{Cardinalities behave as in (H\ref{hyp:deg}).} and so that $\varphi, \psi \in \operatorname{IB}^{\omega,n}_k, \varphi \neq \psi \Rightarrow \varphi(\operatorname{dom}(\varphi)) \cap \psi(\operatorname{dom}(\psi)) = \emptyset$. In particular, the set $\operatorname{IB}(T_\omega^n) = \bigcup_{k \in K_{\omega,n}} \operatorname{IB}^{\omega,n}_k$ is finite and so that $\varphi, \psi \in \operatorname{IB}(T_\omega^n), \operatorname{dom}(\varphi) \cap \operatorname{dom}(\psi) = \emptyset \Rightarrow \varphi(\operatorname{dom}(\varphi)) \cap \psi(\operatorname{dom}(\psi)) = \emptyset$.
\end{subhyp}

The following item is a consequence of the previous ones, but we list it here for convenience.

\begin{subhyp}[Cylinders] \label{hyp:cyl}
    $\forall \omega \in \Omega, \forall n \ge 1$,  $C^\omega_n = \{ \xi = \varphi(\operatorname{dom}(\varphi)) : \varphi \in \operatorname{IB}(T_\omega^n) \}$ is finite and has at most $\mathcal{N}$ overlaps.
\end{subhyp}

\begin{subhyp}[Large covering] \label{hyp:largecov}
For $\mathbb{P}$-a.e. $\omega \in \Omega$, $\forall n \ge 1$, $\mu_\omega\left(M \setminus \bigcup_{\xi \in C^\omega_n} \xi\right) = 0$.
\end{subhyp}

\begin{subhyp}[Big images] \label{hyp:nondegen}
 $\exists \iota >0$ so that
\begin{equation*}
    \essinf_{\omega \in \Omega} \inf_{n \ge 1} \inf_{k \in K_{\omega,n}} \mu_{\theta^n \omega} ( B_R( y^{\omega,n}_k )) > \iota.
\end{equation*}
\end{subhyp}

Next, we consider that the aforementioned (plain) cylinders are refined enough as to split and distinguish regions with different hyperbolic behavior. 

\begin{hyp}[Hyperbolicity and cylinders] \label{hyp:hyper} 
Plain cylinders split into acceptable (and unacceptable) cylinders, whereas acceptable cylinders subsplit into good (and bad) cylinders. \\Namely: $\forall \omega \in \Omega, \forall n \ge 1:$ $C_\omega^n = \stackrel[]{+}{C^\omega_n} \sqcup \stackrel[]{-}{C^\omega_n}$, $\stackrel[]{+}{C^\omega_n} = \stackrel[]{++}{C^\omega_n} \sqcup \stackrel[]{+-}{C^\omega_n}$, making measurable
\begin{equation*}
    \stackrel[]{*}{\mathcal{C}_n}(\omega,x) = \begin{cases}
        1, x \in \bigcup_{ \xi \in \stackrel[]{*}{C^\omega_n} } \xi, \\ 0, \text{ otherwise} 
    \end{cases} \textbf{ }(* \in \{ +, -, ++, +- \}).
\end{equation*}

\begin{notation}
For $* \in \{+,-,++,+-\}$, write $\stackrel{*}{\operatorname{IB}}(T_\omega^n) {=} \{ \varphi \in \operatorname{IB}(T_\omega^n) : \xi = \varphi(\operatorname{dom}(\varphi)) \in \stackrel[]{*}{C^\omega_n} \}$.    
\end{notation}

This splitting distinguishes hyperbolic behavior in the sense of satisfying: 
\end{hyp}

\begin{subhyp}[Weak hyperbolicity on plain cylinders] \label{hyp:weakhypplaincyl} $\forall n \ge 1:$
\begin{equation*}
    1 \le \inf_{\omega \in \Omega} \inf_{\xi \in C^\omega_n} \inf_{\substack{v \in T_xM \\ \|v\| =1}} |DT_\omega^n(x) v|  \le \sup_{\omega \in \Omega} \sup_{\xi \in C^\omega_n} \sup_{x \in \xi} \sup_{\substack{v \in T_xM \\ \|v\| =1}} |DT_\omega^n(x) v|  \le \infty.
\end{equation*}
\end{subhyp}

\begin{subhyp}[Bounded derivatives on acceptable cylinders]\label{hyp:bddderiv}
    $\forall n \ge 1$:
    \begin{equation*}
         \sup_{\omega \in \Omega} \sup_{\xi \in C^\omega_n} \sup_{x \in \xi} \sup_{\substack{v \in T_xM \\ \|v\| =1}} |DT_\omega^n(x) v|  =: a_n < \infty.
    \end{equation*}
\end{subhyp}
\begin{subhyp}[Distortion on good cylinders] \label{hyp:dist}$\exists \mathfrak{d} \ge 0, \exists C >1$, $\forall n \ge 1$: (denoting $\xi = \varphi(\operatorname{dom}(\varphi))$) $$\esssup_{\omega \in \Omega} \sup_{\varphi \in \stackrel{++}{\operatorname{IB}} \hspace{-0.5mm}(T_\omega^n)} \sup_{x,y \in \xi} \frac{J_\varphi(x)}{J_\varphi(y)} \le C n^\mathfrak{d},$$where$$J_\varphi(x) = \frac{d \varphi_* \left[ \mu_{\theta^n \omega}|_{\operatorname{dom}(\varphi)} \right] }{d\mu_\omega|_{\varphi(\operatorname{dom}(\varphi))}}(x) = \frac{d \varphi_* \left[ \mu_{\theta^n \omega}|_{T_\omega^n \xi} \right ] }{d\mu_\omega|_{\xi}}(x).$$ 
\end{subhyp}
\begin{subhyp}[Backward contraction on good cylinders] \label{hyp:backcontr}
$\exists \kappa > 1, \exists D >1, \forall n \ge 1$: (denoting $\xi = \varphi(\operatorname{dom}(\varphi))$) $$\esssup_{\omega \in \Omega} \hspace{-1mm}\sup_{\varphi \in \stackrel{++}{\operatorname{IB}}(T_\omega^n)} \sup_{z \in \operatorname{dom}(\varphi)} \sup_{ \substack{ v \in T_z M \\ \|v\|=1} }  |D \varphi(z) v| {\le} D n^{- \kappa} \hspace{1.5mm}\text{i.e.}\hspace{1mm} D n^{\kappa} {\le} \essinf_{\omega \in \Omega} \hspace{-1mm} \inf_{\varphi \in \stackrel{++}{\operatorname{IB}}(T_\omega^n)} \inf_{x \in \xi} \inf_{ \substack{ v \in T_x M \\ \|v\|=1} } | D T_\omega^n(x) v|,$$and, in particular, $$\esssup_{\omega \in \Omega} \sup_{\varphi \in \stackrel{++}{\operatorname{IB}} \hspace{-0.5mm}(T_\omega^n)} \operatorname{diam}(\xi) \le D n^{- \kappa}.$$
\end{subhyp}

\begin{hyp}[Target position] \label{hyp:targetposit} 
\end{hyp}
\begin{subhyp}[Uniform inclusion in adequate set]\label{hyp:quenchedsep}
$
\forall L \ge 1, \exists \rho_{\operatorname{sep}}(L) >0, \forall \rho \le \rho_{\operatorname{sep}}(L), \forall \omega \in \Omega$:
$$\forall 1 \le L' \le L,\forall 0 \le j \le L'-1 : (T_\omega^j)^{-1} \Gamma_{\nicefrac{3}{2}\rho}(\theta^j \omega) \subset \hspace{1mm} \stackrel{+}{\mathcal{C}} \hspace{-3.5mm}\phantom{\mathcal{C}}^\omega_{L'-1} .$$
\end{subhyp}

\begin{subhyp}[Quenched separation from non-good set]\label{hyp:quenchedsep2} It holds that
\begin{equation*}
    \lim_{L \rightarrow \infty} \varliminf_{\rho \to 0} \hspace{-7mm}\overline{\phantom{\varliminf}}  \hspace{2mm}
  \sum_{n=L}^\infty \sum_{i=0}^{  \lfloor 1/ \hat{\mu}(\Gamma_\rho)  \rfloor  } \mu_{\theta^i \omega }\left(\Gamma(\theta^i \omega) \cap \left[\stackrel[]{+-}{\mathcal{C} } \hspace{-4mm}\phantom{\mathcal{C}}^{\theta^i \omega }_n \cup \stackrel[]{-}{\mathcal{C} } \hspace{-3mm}\phantom{\mathcal{C}}^{\theta^i \omega }_n  \right] \right)   = 0 \text{, }\mathbb{P}\text{-a.s.}.
\end{equation*}
\end{subhyp}


\begin{hyp}[Lipschitz regularities]\label{hyp:lip}
\begin{subhyp}[Map]
    $\sup_{x \in M} \operatorname{Lip}(T_\cdot(x): \Omega \to M) < \infty$.
\end{subhyp}
\begin{subhyp}[Driving]
$\operatorname{Lip}(\theta) < \infty$.    
\end{subhyp}
\begin{subhyp}[Target]
    $\operatorname{Lip}(\Gamma: \Omega \to \mathscr{P}(M)) < \infty$,where $\mathscr{P}(M) = \{A \subset M, A \text{ compact, }A \neq \emptyset\}$ is equipped with the Hausdorff distance $d_H(A,B) = \sup\limits_{x \in A} \inf\limits_{y \in B} d(x,y) \vee \sup\limits_{y \in B} \inf\limits_{X \in A} d(x,y)$, which makes it a compact metric space.
\end{subhyp}
\end{hyp}

\begin{hyp}[Measure regularity] \label{hyp:meas}
\end{hyp}

\begin{subhyp}[Ball regular] \label{hyp:ball} $\exists 0 < d_0 \le d_1 < \infty, \exists C_0,C_1 > 0, \exists \rho_{\operatorname{dim}} \le 1, \forall \rho \le \rho_{\operatorname{dim}}$, for $\mathbb{P}$-a.e. $\omega \in \Omega$: $$C_1 \rho^{d_1} \le \mu_\omega(\Gamma_\rho(\omega)) \le C_0 \rho^{d_0}.$$
\end{subhyp}

\begin{subhyp}[Annulus regular] \label{hyp:annulus} $\exists \eta \ge \beta >0, \exists E > 0, \exists \rho_{\operatorname{dim}} \le 1, \forall \rho \le \rho_{\operatorname{dim}}, \forall r \in (0,\rho/2)$, for $\mathbb{P}$-a.e. $\omega \in \Omega$: $$\frac{\mu_\omega(\Gamma_{\rho+r}(\omega) \setminus \Gamma_{\rho-r}(\omega)) }{\mu_\omega(\Gamma_\rho(\omega)) } \le E \frac{r^\eta}{\rho^\beta}.$$
\end{subhyp}

\begin{hyp}[Decay of correlations] \label{hyp:dec}
$\exists \mathfrak{p}>1$ so that
\end{hyp}
\begin{subhyp}[Quenched] \label{hyp:quenchdec} For $\mathbb{P}$-a.e. $\omega \in \Omega$, $\forall G \in \operatorname{Lip}_{d_M}(M,\mathbb{R}), \forall H \in L^\infty(M,\mathbb{R}), \forall n \ge 1$: $$\left| \int_M G \cdot (H \circ T_\omega^n) d \mu_\omega - \mu_\omega(G) \mu_{\theta^n \omega}(H) \right| \lesssim n^{-\mathfrak{p}}  \|G\|_{\operatorname{Lip}_{d_M}} \|H\|_\infty.$$
\end{subhyp}

\begin{subhyp}[Annealed] \label{hyp:annealdec} $\forall G \in \operatorname{Lip}_{d_{\Omega \times M}}(\Omega \times M,\mathbb{R}), \forall H \in L^\infty(\Omega \times M,\mathbb{R}), \forall n \ge 1$: $$\left| \int_{\Omega \times M} G \cdot (H \circ S^n) d \hat\mu - \hat\mu(G) \hat\mu(H) \right| \lesssim n^{-\mathfrak{p}} \|G\|_{\operatorname{Lip}_{d_{\Omega \times M}}} \|H\|_\infty.$$ 
\end{subhyp}

\begin{hyp}[Hitting regular] \label{hyp:hitting}     
    \begin{equation*}
      \exists (\lambda_\ell)_{\ell \ge 1},  \sum\nolimits_{\ell =1}^\infty \lambda_\ell = 1, \sum\nolimits_{\ell =1}^\infty \ell^3 \lambda_\ell <\infty.
    \end{equation*}
\end{hyp}

\begin{hyp}[Return regular] \label{hyp:return}
    \begin{equation*}
       \exists (\alpha_\ell)_{\ell \ge 1} , \alpha_1 >0, \sum_{\ell =1}^\infty \alpha_\ell = 1, \sum_{\ell =1}^\infty \ell^2 \alpha_\ell <\infty.
    \end{equation*}
    We call $\alpha_1$ the extremal index.
\end{hyp}

\vspace{3mm}

\hspace{-4mm}\textbf{H\ref{hyp:return}'} (Pre return regular)\textbf{.} \textit{It holds that}
\begin{equation*}
    \exists (\hat\alpha_\ell)_{\ell \ge 1}, \hat\alpha_1 - \hat\alpha_2 >0, \sum_{\ell=1}^\infty \ell \hat\alpha_\ell < \infty.
\end{equation*}

Using the final implication of item VI), it is immediate that (H\ref{hyp:return}') $\Rightarrow$ (H\ref{hyp:return}), because $\alpha_1 = \hat\alpha_1 - \hat\alpha_2 >0$, $\sum_{\ell=1}^\infty \alpha_\ell = \hat\alpha_1 = 1$, and $\sum_{\ell=1}^\infty \ell^2 \alpha_\ell \le 2 \sum_{\ell=1}^\infty \ell \hat\alpha_\ell <\infty$.


Moreover, for technical conditions, we assume that the quantities appearing in the previous hypotheses harmonize so that the following constraints hold. Mostly, they hold when (polynomial) decay is sufficiently fast. 

\begin{hyp}[Parametric constraints] \label{hyp:param}
It holds that
\end{hyp}
\begin{subhyp} \label{hyp:paramincenter}
$d_0 (\mathfrak{p}-1) > \frac{ 2\left(\frac{\beta + d_1}{ \eta} \vee 1\right) + d_1 }{ d_0/d_1}$,
\end{subhyp}
\begin{subhyp} \label{hyp:parampositrhoexp}
$\frac{d_0}{\mathfrak{d}+1
} \mathfrak{p} >  2\left(\frac{\beta + d_1}{ \eta} \vee 1\right) + d_1 $,
\end{subhyp}
\begin{subhyp} \label{hyp:parampolynomialseries}
$\mathfrak{d}<\kappa d_0-1$.
\end{subhyp}

\subsection{Main results}\label{sec:mainresults}

The first result, although interesting on its own, plays mostly an auxiliary role. Valid in the general setup of section \ref{sec:gensetup}, it expresses hitting quantities ($\lambda_\ell$'s) in terms of return quantities ($\alpha_\ell$'s). This is providential because the former quantities are the ones central to the theory, but the latter quantities can be computed directly on examples. 

\begin{thrm}\label{thm:lambdaalpha}
Let $(\theta, \mathbb{P}, T_\omega, \mu_\omega, \Gamma)$ be a system as described in section \ref{sec:gensetup}, with $(\theta,\mathbb{P})$ only assumed invariant.

Then
$$\text{(H\ref{hyp:return}')}\text{ }\Rightarrow \text{ } \lambda_\ell = \frac{\alpha_\ell - \alpha_{\ell +1}}{\alpha_1} \text{ }(\ell \ge 1) \text{ and } \text{(H\ref{hyp:hitting})}.$$  
\end{thrm} 

Theorem \ref{thm:lambdaalpha} generalizes theorem 2 from \cite{haydn2020limiting} to the random situation. Its proof is basically the same, so we omit it. The interested reader can find the adapted proof in \cite{amorim2023compound}.  

\begin{remark}
    Theorem \ref{thm:lambdaalpha} implies that $\alpha_1 = (\sum_{\ell=1}^\infty \ell \lambda_\ell)^{-1}$.
\end{remark}

Let us now formulate our main result. It says that the systems prescribed in section \ref{sec:wrksetup} have compound Poissonian quenched hitting statistics.

\begin{thrm}\label{thm:main}
Let $(\theta, \mathbb{P}, T_\omega, \mu_\omega, \Gamma)$ be a system satisfying (H\ref{hyp:amb}-H\ref{hyp:dec}), (H\ref{hyp:return}') and (H\ref{hyp:param}).
 
Then: \hspace{-0.5mm}$\forall t {>} 0, \forall n {\ge} 0, \forall (\rho_m)_{m \ge 1} {\searrow} 0 \text{ with }\sum\nolimits_{m \ge 1} {\rho_m}^q {<} \infty \text{ } \text{(for some } {0} {<} {q} {<} q(d_0,d_1,\eta,\beta, \mathfrak{p})\footnotemark\text{)}$\footnotetext{A quantity to be introduced in lemma \ref{lem:variance}.} one has
\begin{equation}\label{eq:mainthmlimit}\displaystyle
 \qquad \mu_\omega(Z^{\omega,\lfloor t / \hat\mu(\Gamma_{\rho_m}) \rfloor}_{\Gamma_{\rho_m}}= n) \stackrel[m \rightarrow \infty]{\mathbb{P}\text{-a.s.}}{\longrightarrow} \operatorname{CPD}_{t \alpha_1,(\lambda_\ell)_\ell}(n), 
\end{equation}where $\operatorname{CPD}_{s, (\lambda_\ell)_{\ell}}$ is the compound Poisson distribution with intensity $s$ and multiplicity distribution $(\lambda_\ell)_{\ell}$.
\end{thrm}

\begin{remark}
    The quantity $q(d_0,d_1,\eta,\beta, \mathfrak{p})>0$ will be introduced explicitly in lemma \ref{lem:variance}.
\end{remark}

\begin{remark}
     The $\lambda_\ell$'s in the limit of equation (\ref{eq:mainthmlimit}) are those given in equation (\ref{eq:lambdadef}), whose existence follows from (H\ref{hyp:return}') and theorem \ref{thm:lambdaalpha}. 
\end{remark}
\begin{remark}
If the system has exponential asymptotics in (H\ref{hyp:dec}) and (H\ref{hyp:backcontr}), the previous conclusion is still true, but, actually, with fewer parametric conditions being required: instead of (H\ref{hyp:paramincenter})-(H\ref{hyp:parampolynomialseries}), only $\kappa d_0 >1$ is needed.    
\end{remark}

The previous theorem can be strengthened to the following one, which provides an analogous limit theorem for point processes.

\begin{thrm}\label{thm:main2}Let $(\theta, \mathbb{P}, T_\omega, \mu_\omega, \Gamma)$ be a system satisfying (H\ref{hyp:amb}-H\ref{hyp:dec}), (H\ref{hyp:return}') and (H\ref{hyp:param}).
 
Then: \hspace{-0.5mm}$\forall t {>} 0, \forall (\rho_m)_{m \ge 1} {\searrow} 0 \text{ with }\hspace{-0.5mm}\sum\nolimits_{m \ge 1} {\rho_m}^q {<} \infty \text{ } \text{(for some } {0} {<} {q} {<} q(d_0,d_1,\eta,\beta, \mathfrak{p})\text{)}$ one has
\begin{equation}\label{eq:mainthmlimit2}\displaystyle
 \qquad {Y^{\omega,\lfloor t / \hat\mu(\Gamma_{\rho_m}) \rfloor}_{\Gamma_{\rho_m}}}_* \mu_\omega \stackrel[m \rightarrow \infty]{\mathbb{P}\text{-a.s.}}{\longrightarrow} \operatorname{CPPP}_{t \alpha_1,(\lambda_\ell)_\ell} \text{in }\mathcal{P}(\mathfrak{M})\footnotemark, 
\end{equation}\footnotetext{See footnote \ref{note:topologies}.}where $\operatorname{CPPP}_{s, (\lambda_\ell)_{\ell}}$ is the compound Poisson point process with intensity $s$ and multiplicity distribution $(\lambda_\ell)_{\ell}$.
\end{thrm}

\textcolor{white}{ }

\paragraph{\textbf{Structure of the paper.}} The rest of the paper is organized into two parts: 

I) Theory: Until section \ref{sec:proofmainthm} we work to prove theorem \ref{thm:main}.

Section \ref{sec:approxthm} proves theorem \ref{thm:approx}. This result provides the skeleton of the proof of theorem \ref{thm:main}, by approximating the left side of equation (\ref{eq:mainthmlimit}). Denoting it briefly by $\mu_\omega(Z= n)$, one splits $Z$ into equally sized blocks and mimics them with an independency of random variables, whose sum forms $\tilde{Z}$. Theorem \ref{thm:approx} bounds $|\mu_\omega(Z= n)- \mu_\omega(\tilde{Z}= n)|$ by with a sum of long-range components (terms $\mathcal{R}^1$ and $\tilde{\mathcal{R}}^1$, to appear) and short-range components (terms $\mathcal{R}^2$ and $\mathcal{R}^3$, to appear).

Section \ref{sec:proofmainthm} proofs theorem \ref{thm:main}. To estimate long-range errors, it uses weak hyperbolicity features (H\ref{hyp:weakhypplaincyl},H\ref{hyp:bddderiv}), the target uniform inclusion in the adequate set (H\ref{hyp:quenchedsep}), the annulus regularity (H\ref{hyp:annulus}) and quenched decay (H\ref{hyp:quenchdec}). To estimate short-range errors, it uses structure of the covering system (H\ref{hyp:inv}), distortion (H\ref{hyp:dist}), strong hyperbolicity features (H\ref{hyp:backcontr}) and ball regularity (H\ref{hyp:ball}). Notice annealed decay was not yet used.

To control the newly arranged estimates (still carrying some $\omega$-dependency) and to show that $\mu_\omega(\tilde{Z}= n)$ goes to the desired CPD, thus closing the proof, the missing piece is an almost sure convergence result, which allows for the quenched theorem.

This almost sure convergence result is lemma \ref{lem:asconv}, proved in section \ref{sec:BClemmata} after a Borel-Cantelli argument and a variance control (lemma \ref{lem:variance}). The proof of the variance control finally uses the annealed decay of correlations (H\ref{hyp:annealdec}) and the regularity in $\omega$ of maps and targets (H\ref{hyp:lip}).

Finally, theorem \ref{thm:main2} is proved in section \ref{sec:proofmain2}. Although it implies theorem \ref{thm:main}, to make ideas more transparent, we preferred to prove \ref{thm:main} and leverage on this proof to prove theorem \ref{thm:main}. This decision can benefit users who wish to upgrade compound Poison distributions limit theorems into compound Poisson point processes limit theorems.

II) Applications: In section \ref{sec:ex} we consider certain random piecewise expanding one-dimensional systems, casting new light on the well-known deterministic dichotomy between periodic and aperiodic points, their typical extremal index formula $\operatorname{EI}= 1 - 1/ JT^p(\zeta)$, and recovering the geometric case for general Bernoulli-driven systems, but distinct behavior otherwise.

\section{An abstract approximation theorem} \label{sec:approxthm}

The following theorem approximates the probability distribution of an arbitrary sum of binary variables in terms of the distribution of a suitable sum of independent random variables. More precisely, to build the `suitable' independent random variables, one splits the first sum into smaller block-sums, and each of them is distributionally mimicked by a new random variable, with the collection of new ones being taken to be independent.


    \begin{thrm} \label{thm:approx} Consider $n \ge 0$, $L \ge n$, $N \in \mathbb{N}_{\ge 3}$ large enough so that $L \le \lfloor \frac{N}{3} \rfloor$, and $(X_i)_{i=0}^{N-1}$ arbitrary $\{0,1\}$-valued random variables on $( \mathcal{X}, \mathscr{X},\mathbb{Q})$. Denote $N':=\frac{N}{L} \in \mathbb{N}_{\ge 3}$\footnote{\label{note:fractional} Although $L$ need not divide $N$, we pretend this is the case, for simplification purposes, i.e. to neglect possible remainder terms associated with the fractional part --- which should not play a role in the asymptotics (of either the error and leading terms).} and $(Z_j)_{j=0}^{N'-1}$ given by $Z_j := \sum_{i=jL}^{(j+1)L-1} X_i$.

Let $(\tilde{Z}_j)_{j=0}^{N'-1}$ be an independency of $\mathbb{N}_{\ge 0}$-valued random variables on $ (X, \mathcal{X},\mathbb{Q})$ satisfying $\tilde{Z}_j \sim Z_j$ ($j=0,\ldots,N'-1$) and $(\tilde{Z}_j)_{j=0}^{N'-1} \perp (Z_j)_{j=0}^{N'-1}$.

Denote $\tilde{W}_a^b := \sum_{j=a}^{b} \tilde{Z}_j$ ($0 \le a \le b \le N'-1$) and $\tilde{W} := \tilde{W}_0^{N'-1}$. Similarly notation without $\sim$'s is adopted, in which case $W$ coincides with $\sum_{i=0}^{N-1} X_i$.

Then, for all $\Delta \in [1,N']$:
\begin{equation*}
    \left| \mathbb{Q} (W =n ) - \mathbb{Q} ( \tilde{W} =n  ) \right|
    \lesssim \tilde{\mathcal{R}}^1(N,L,\Delta) + \mathcal{R}^1(N,L,\Delta) + \mathcal{R}^2(N,L,\Delta) + \mathcal{R}^3(N,L,\Delta),
\end{equation*}where
\begin{equation*}
    \tilde{\mathcal{R}}^1(N,L,\Delta)= \sum_{j=0}^{N'-1} \max_{q \in [0,n]} \left| \mathbb{Q}(Z_{j} \ge 1) \mathbb{Q}(W_{j+\Delta}^{N'-1}=q) - \mathbb{Q}(Z_j\ge 1, W_{j+\Delta}^{N'-1}=q) \right|,
\end{equation*}
\begin{equation*}
    \mathcal{R}^1 (N,L,\Delta) {=} \!\sum_{j=0}^{N'-1} \max_{q \in [1,n]}  \sum_{u=1}^{q} \left| \mathbb{Q}\!\left(\!Z_j {=}{u}, W_{j+\Delta}^{N'-1}  {=} {q}{-}{u} \right) {-} \mathbb{Q}\Big(\! Z_j {=}{u} \! \Big)  \mathbb{Q}\!\left( W_{j+\Delta}^{N'-1} {=} {q}{-}{u} \right) \right|,
\end{equation*}
\begin{equation*}
    \mathcal{R}^2(N,L,\Delta) = \sum_{j=0}^{N'-1} \mathbb{Q}\left( Z_j \ge 1, W_{j+1}^{j + \Delta -1}  \ge 1 \right)\text{ and}
\end{equation*}
\begin{equation*}
    \mathcal{R}^3(N,L,\Delta) = \sum_{i=0}^N \sum_{q=0 \vee (i- {\Delta} L)}^i \mathbb{Q}(X_i=1) \mathbb{Q}(X_q=1),
\end{equation*}with the convention that, for $b > a$, $W_b^a \equiv 0$ and $\mathbb{Q}(W_b^a \ge 1) = 0$.
\end{thrm}

\begin{proof}
Using a telescopic sum and the given independence, one has 
\begin{equation*}
    \left| \mathbb{Q}(W =n) - \mathbb{Q} (\tilde{W} =n ) \right| \le \sum_{j=0}^{N'-1} \left| \mathbb{Q}(\tilde{W}_{0}^{j-1}+W_{j}^{N'-1}=n)-\mathbb{Q}(\tilde{W}_{0}^{j}+W_{j+1}^{N'-1}=n) \right|
\end{equation*}
\begin{equation*}
    \le \sum_{j=0}^{N'-1} \sum_{l=0}^n \mathbb{Q}(\tilde{W}_{0}^{j-1}=l)\left| \mathbb{Q}(W_{j}^{N'-1}=n-l)-\mathbb{Q}(\tilde{Z}_{j}+W_{j+1}^{N'-1}=n-l) \right|.
\end{equation*}

We now estimate
\begin{equation*}
\left| \mathbb{Q}(W_{j}^{N'-1}=q)-\mathbb{Q}(\tilde{Z}_{j}+W_{j+1}^{N'-1}=q) \right|
\end{equation*}
\begin{equation*}
 \le \sum_{u=0}^q \left| \mathbb{Q}(Z_j=u, W_{j+1}^{N'-1}=q-u)-\mathbb{Q}(\tilde{Z}_{j}=u,W_{j+1}^{N'-1}=q-u)   \right|
\end{equation*}
\begin{equation*}
    = \sum_{u=0}^q \left| \mathbb{Q}(Z_j=u, W_{j+1}^{N'-1}=q-u)-\mathbb{Q}(Z_{j}=u)\mathbb{Q}(W_{j+1}^{N'-1}=q-u)   \right| =: \sum_{u=0}^q | \mathcal{R}_j(q,u) |.
\end{equation*}

We single out $u = 0$ from the previous sum,
\begin{eqnarray*}
    |\mathcal{R}_j(q,0)| & = & \left| \mathbb{Q}(Z_j=0, W_{j+1}^{N'-1}=q)-\mathbb{Q}(Z_{j}=0)\mathbb{Q}(W_{j+1}^{N'-1}=q)   \right| \\
    & = & \Big| \left( \mathbb{Q}(W_{j+1}^{N'-1}=q) - \mathbb{Q}(Z_j\ge 1, W_{j+1}^{N'-1}=q) \right) \\
    & & \hspace{7mm} - \left( \mathbb{Q}(W_{j+1}^{N'-1}=q) - \mathbb{Q}(Z_{j} \ge 1) \mathbb{Q}(W_{j+1}^{N'-1}=q) \right)   \Big|\\
    & = &  \left| \mathbb{Q}(Z_{j} \ge 1) \mathbb{Q}(W_{j+1}^{N'-1}=q) - \mathbb{Q}(Z_j\ge 1, W_{j+1}^{N'-1}=q) \right|.
\end{eqnarray*}

It follows that 
\begin{equation*}
    \hspace{-70mm} \left| \mathbb{Q}(W =n) - \mathbb{Q} (\tilde{W} =n ) \right| \le \displaystyle \sum_{j=0}^{N'-1} \sum_{q=0}^{n} \sum_{u=0}^{q} |\mathcal{R}_j(q,u)|
\end{equation*}
\begin{equation*}
\begin{tabular}{@{}c  l}
     $\le$ & $\displaystyle \hspace{-1mm}n \sum_{j=0}^{N'-1} \max_{q \in [0,n]} \left| \mathbb{Q}(Z_{j} \ge 1) \mathbb{Q}(W_{j+1}^{N'-1}=q) {-} \mathbb{Q}(Z_j\ge 1, W_{j+1}^{N'-1}=q) \right| {+} \sum_{j=0}^{N'-1} \sum_{q=0}^{n} \sum_{u=1}^{q} |\mathcal{R}_j(q,u)|$\\
     $\lesssim$ & $\displaystyle \hspace{-1mm} \sum_{j=0}^{N'-1} \max_{q \in [0,n]} \left| \mathbb{Q}(Z_{j} \ge 1) \mathbb{Q}(W_{j+1}^{N'-1}=q) {-} \mathbb{Q}(Z_j\ge 1, W_{j+1}^{N'-1}=q) \right| {+} \sum_{j=0}^{N'-1} \sum_{q=0}^{n} \sum_{u=1}^{q} |\mathcal{R}_j(q,u)|$.
\end{tabular} 
\end{equation*}

The first summation will be kept on hold. We deal with the second one now. 

For $u=1,\ldots,q$, we expand $|\mathcal{R}_j(q,u)|$ by including intermediate terms with a time gap $\Delta$ and applying the triangular inequality, as follows:
\begin{eqnarray*}
    | \mathcal{R}_j(q,u) | & \le & \Big| \mathbb{Q}(Z_j=u, W_{j+1}^{N'-1}=q-u) - \mathbb{Q}(Z_j=u, W_{j+\Delta}^{N'-1}=q-u) \Big| \\
    & + & \Big| \mathbb{Q}(Z_j=u, W_{j+\Delta}^{N'-1}=q-u) - \mathbb{Q}(Z_j=u) \mathbb{Q}(W_{j+\Delta}^{N'-1}=q-u) \Big| \\
    & + & \Big| \mathbb{Q}(Z_j=u) \mathbb{Q}(W_{j+\Delta}^{N'-1}=q-u) - \mathbb{Q}(Z_{j}=u)\mathbb{Q}(W_{j+1}^{N'-1}=q-u)   \Big|,
\end{eqnarray*}where the entries in the RHS are denoted, respectively, by $|\mathcal{R}_j^2(q,u)|$, $|\mathcal{R}_j^1(q,u)|$ and $|\mathcal{R}_j^3(q,u)|$ (note the unusual order).

Then the sum of the following three terms bounds the later triple sum.

First:
\begin{equation*}
\sum_{j=0}^{N'-1} \sum_{q=0}^{n} \sum_{u=1}^{q} |\mathcal{R}_j^1(q,u)| \lesssim   \sum_{j=0}^{N'-1} \max_{q \in [1,n]} \sum_{u=1}^{q} |\mathcal{R}_j^1(q,u)| = \mathcal{R}^1(N,L,\Delta).
\end{equation*}

Second:
\begin{equation*}
\sum_{j=0}^{N'-1} \sum_{q=0}^{n} \sum_{u=1}^{q} |\mathcal{R}_j^2(q,u)| \lesssim  \sum_{j=0}^{N'-1} \max_{q \in [1,n]} \sum_{u=1}^{q} |\mathcal{R}_j^2(q,u)| 
\end{equation*}
\begin{equation*}
    \lesssim \sum_{j=0}^{N'-1} \mathbb{Q}(Z_j \ge 1, W_{j+1}^{j+\Delta-1} \ge 1) = \mathcal{R}^2(N,L,\Delta),
\end{equation*}where the step used that
\begin{equation*}
    A_u := \{Z_j=u, W_{j+1}^{N'-1}=q-u\}, B_u := \{Z_j=u, W_{j+\Delta}^{N'-1}=q-u\}
\end{equation*}
\begin{equation*}
   \Rightarrow A_u \setminus B_u, B_u \setminus A_u \subset \{Z_j=u, W_{j+1}^{j+\Delta-1}\ge 1\}
\end{equation*}
\begin{eqnarray*}
    \Rightarrow \sum_{u=1}^{q} |\mathcal{R}_j^2(q,u)| & = & \sum_{u=1}^{q} \left| \mathbb{Q}(A_u) - \mathbb{Q}(B_u) \right| \le  \sum_{u=1}^{q} \mathbb{Q}(Z_j=u, W_{j+1}^{j+\Delta-1} \ge 1)  \\
    & \le & \mathbb{Q}(Z_j \ge 1, W_{j+1}^{j+\Delta-1} \ge 1).
\end{eqnarray*}

Third:
\begin{equation*}
    \sum_{j=0}^{N'-1} \sum_{q=0}^{n} \sum_{u=1}^{q} |\mathcal{R}_j^3(q,u)| \lesssim \sum_{j=0}^{N'-1} \max_{q \in [1,n]} \sum_{u=1}^{q} |\mathcal{R}_j^3(q,u)| 
\end{equation*}
\begin{equation*}
    \lesssim  \sum_{j=0}^{N'-1} \sum_{l=jL}^{(j+1)L-1} \sum_{i=(j+1)L}^{(j+\Delta+1)L-1} \mathbb{Q}(X_i{=}1)  \mathbb{Q}(X_l{=}1) =  \sum_{i=0}^{N+\Delta L + L} \sum_{l = 0 \vee (i-L - \Delta L) }^{i-L}  \mathbb{Q}(X_l{=}1) \mathbb{Q}(X_i{=}1)
\end{equation*}
\begin{equation*}
     \le  \sum_{i=0}^N \sum_{l=0 \vee (i - {\Delta} L) }^{i} \mathbb{Q}(X_l=1) \mathbb{Q}(X_i=1) = \mathcal{R}^3(N,L, \Delta),
 \end{equation*}where the second $\lesssim$ step used the following: (with $q'=q-u$)
\begin{equation*}
\begin{tabular}{r c l}
    $\mathbb{Q}(W_{j+1}^{N'-1}{=}{q}')$ & $=$ & $\mathbb{Q}(Z_{j+1}{\ge} {1}, W_{j+1}^{N'_L-1}{=}{q}') {+} \mathbb{Q}(Z_{j+1} {=} {0}, W_{j+1}^{N'-1}{=}{q}')$\\
    $\mathbb{Q}(Z_{j+1} = 0, W_{j+1}^{N'-1}=q')$ & $=$ &$\mathbb{Q}(Z_{j+1} = 0, W_{j+2}^{N'-1}=q')$\\
    & $=$ & $\mathbb{Q}(W_{j+2}^{N'-1}=q') - \mathbb{Q}(Z_{j+1} \ge 1, W_{j+2}^{N'-1}=q')$\\
    $\Rightarrow \substack{\displaystyle |\mathbb{Q}(W_{j+1}^{N'-1}=q') \\ \displaystyle \hspace{10mm}-  \mathbb{Q}(W_{j+2}^{N'-1}=q')|}$ & $=$ &$\hspace{-6mm}\substack{\displaystyle |\mathbb{Q}(Z_{j+1}\ge 1, W_{j+1}^{N'-1}=q') \\ \displaystyle \hspace{10mm}
     - \mathbb{Q}(Z_{j+1} \ge 1, W_{j+2}^{N'-1}=q')|}$ 
\end{tabular}
\end{equation*}
but, with $A:=\{Z_{j+1} \ge 1,W_{j+1}^{N'-1}=q'\}$ and $B:=\{Z_{j+1} \ge 1, W_{j+2}^{N'-1}=q'\}$, one has $A \setminus B, B \setminus A \subset \{Z_{j+1} \ge 1\}$, implying 
\begin{equation*}
|\mathbb{Q}(W_{j+1}^{N'-1}=q')  -  \mathbb{Q}(W_{j+2}^{N'-1}=q')| \le \mathbb{Q}(Z_{j+1} \ge 1)
\end{equation*}
\begin{equation*}
    \Rightarrow |\mathbb{Q}(W_{j+l}^{N'-1}=q')  -  \mathbb{Q}(W_{j+l+1}^{N'-1}=q')| \le \mathbb{Q}(Z_{j+l} \ge 1) \le \sum_{i=(j+l)L}^{(j+l+1)L-1} \mathbb{Q}(X_i=1)
\end{equation*}
\begin{equation*}
    \Rightarrow |\mathbb{Q}(W_{j+1}^{N'-1}=q')  -  \mathbb{Q}(W_{j+\Delta}^{N'-1}=q')| \le \sum_{l=1}^{\Delta-1}\mathbb{Q}(Z_{j+l} \ge 1) \le \sum_{i=(j+1)L}^{(j+\Delta)L-1} \mathbb{Q}(X_i=1)
\end{equation*}
\begin{eqnarray*}
    \Rightarrow \sum_{u=1}^{q} |\mathcal{R}^3_j(q,u)| & \le & \sum_{u=1}^{q} \mathbb{Q}(Z_j=u) \sum_{i=(j+1)L}^{(j+\Delta)L-1} \mathbb{Q}(X_i=1)\\
    & \le & \sum_{l=jL}^{(j+1)L-1} \sum_{i=(j+1)L}^{(j+\Delta)L-1} \mathbb{Q}(X_l=1) \mathbb{Q}(X_i=1).
\end{eqnarray*}

Now we should deal with the summation we left on hold, coming from the singled-out term with $u=0$, namely,
\begin{equation*}
    \sum_{j=0}^{N'-1} \max_{q \in [0,n]} \left| \mathbb{Q}(Z_{j} \ge 1) \mathbb{Q}(W_{j+1}^{N'-1}=q) - \mathbb{Q}(Z_j\ge 1, W_{j+1}^{N'-1}=q) \right|.
\end{equation*}

Using an analogous triangular inequality trick, by adding two mixed terms that have a gap $\Delta$, and organizing them in the same order used before, one verifies that the second term is bounded by $\mathcal{R}^2(N,L,\Delta)$ and the third one is bounded by $\mathcal{R}^3(N,L,\Delta)$. So it suffices to account for the left over term
\begin{equation*}
    \tilde{\mathcal{R}}^1(N,L,\Delta) = \sum_{j=0}^{N'-1} \max_{q \in [0,n]} \left| \mathbb{Q}(Z_{j} \ge 1) \mathbb{Q}(W_{j+\Delta}^{N'-1}=q) - \mathbb{Q}(Z_j\ge 1, W_{j+\Delta}^{N'-1}=q) \right|,
\end{equation*}as desired.
 \end{proof}

\section{Borel-Cantelli type lemmata}\label{sec:BClemmata}

The objective of this section is its final lemma \ref{lem:asconv}, which provides the almost sure convergence needed to back the quenched result in the proof of theorem \ref{thm:main}. This lemma and its proof strategy was inspired in \cite{rousseau2014exponential} (lemma 9). To implement the said proof, a Borel-Cantelli argument is used with expectation control given by lemma \ref{lem:annealedentry} and variance control given by lemma \ref{lem:variance}.

\begin{lem}\label{lem:annealedentry} Let $(\theta, \mathbb{P}, T_\omega, \mu_\omega, \Gamma)$ be a system satisfying (H\ref{hyp:return}') (and so (H\ref{hyp:hitting}), by theorem \ref{thm:lambdaalpha}). Then:
\begin{equation}\label{eq:annealedentry}
\lim\limits_{L \to \infty} \overline{\varliminf\limits_{\rho \to 0}} 
\frac{\hat\mu(Z^L_{\Gamma_\rho} \ge 1)}{L\hat\mu(\Gamma_\rho)}
= ({\textstyle \sum_{\ell=1}^\infty \ell \lambda_\ell})^{-1} =\alpha_1  
\end{equation}\label{eq:annealedentry2}
and
\begin{equation}
\lim\limits_{L \to \infty} \overline{\varliminf\limits_{\rho \to 0}} \frac{\hat\mu(Z^L_{\Gamma_\rho} = n)}{L \hat\mu(\Gamma_\rho)}  = ({\textstyle \sum_{\ell=1}^\infty \ell \lambda_\ell})^{-1} \lambda_n
=\alpha_1\lambda_n  \text{ }(n\ge 1)   
\end{equation}
\end{lem}

\begin{proof}
Using (H\ref{hyp:return}') (for the following items (i.b,ii)) and (H\ref{hyp:hitting}) (items (i.a,iii-iv)), it holds that: $\forall \epsilon >0$

\begin{enumerate}[label=\roman*)]
    \item $\exists \ell_0(\epsilon) \ge 1$ so that \begin{enumerate}[label=\alph*)]
        \item$\sum_{\ell=\ell_0(\epsilon)}^\infty \ell^3 \lambda_\ell \le \epsilon, $
        \item $\forall L \ge 1$: $\sum\limits_{\ell=\ell_0(\epsilon)}^\infty {\ell} \hspace{-1mm} \stackrel[]{+}{\hat\alpha} \hspace{-3mm}\phantom{\lambda}_\ell (L) \le \sum\limits_{\ell=\ell_0(\epsilon)}^\infty {\ell}  \hat\alpha_\ell \le \epsilon.$
    \end{enumerate} 
    \item $\forall L \ge 1, \exists \rho_1(\epsilon,L), \forall \rho \le \rho_1(\epsilon,L)$: 
    \begin{equation*}
        \stackrel[]{-}{\hat\alpha} \hspace{-3.5mm}\phantom{\lambda}_\ell (L)  - \epsilon/(L^2) \le \hat\alpha_\ell(L,\rho) \le   \stackrel[]{+}{\hat\alpha} \hspace{-3.5mm}\phantom{\lambda}_\ell (L)  + \epsilon/(L^2) \text{ } (\forall \ell=1,\ldots,L)
    \end{equation*}
    \begin{equation*}
        \Rightarrow \sum_{\ell=\ell_0(\epsilon)}^{L}  \ell \hat\alpha_\ell(L,\rho) \le  \sum_{\ell=\ell_0(\epsilon)}^{L} \ell \left(\stackrel[]{+}{\hat\alpha} \hspace{-3.5mm}\phantom{\lambda}_\ell (L)  + \epsilon/(L^2) \right) \le 2\epsilon \text{ by (i).}
    \end{equation*}
    \item $\forall L \ge 1, \exists \rho_3(\epsilon,L), \forall \rho \le \rho_3(\epsilon,L)$:
    \begin{equation*}
        \stackrel[]{-}{\lambda} \hspace{-3.5mm}\phantom{\lambda}_\ell (L)  - \epsilon/(\ell_0(\epsilon))^2 \le \lambda_\ell(L,\rho) \le   \stackrel[]{+}{\lambda} \hspace{-3.5mm}\phantom{\lambda}_\ell (L)  + \epsilon/(\ell_0(\epsilon))^2 \text{ } (\forall \ell=1,\ldots,\ell_0(\epsilon)).
    \end{equation*}
    \item $\exists L_0(\epsilon) > \ell_0(\epsilon), \forall L \ge L_0(\epsilon)$:
    \begin{equation*}
        |\lambda_\ell - \stackrel[]{-}{\lambda} \hspace{-3.5mm}\phantom{\lambda}_\ell (L)| \le \epsilon/(\ell_0(\epsilon))^2, \text{ }|\lambda_\ell - \stackrel[]{+}{\lambda} \hspace{-3.5mm}\phantom{\lambda}_\ell (L)| \le \epsilon/(\ell_0(\epsilon))^2 \text{ }(\forall \ell=1,\ldots,\ell_0(\epsilon))
    \end{equation*}
    \begin{equation*}
        \Rightarrow |\stackrel[]{+}{\lambda} \hspace{-3.5mm}\phantom{\lambda}_\ell (L) - \stackrel[]{-}{\lambda} \hspace{-3.5mm}\phantom{\lambda}_\ell (L)| \le 2\epsilon/(\ell_0(\epsilon))^2 \text{ }(\forall \ell=1,\ldots,\ell_0(\epsilon)).
    \end{equation*}
    \item (due to items (iv-v)) $\exists L_0(\epsilon), \forall L \ge L_0(\epsilon), \exists \rho_3(\epsilon,L), \forall \rho \le \rho_3(\epsilon,L)$:
    \begin{equation*}
        |\lambda_\ell(L,\rho) - \stackrel[]{\Asterisk}{\lambda} \hspace{-3.5mm}\phantom{\lambda}_\ell (L)| \le 3\epsilon/ (\ell_0(\epsilon))^2\text{ }(\forall \ell=1,\ldots,\ell_0(\epsilon), \forall \Asterisk \in \{-,+\} )
    \end{equation*}
    \begin{equation*}
        \Rightarrow |\lambda_\ell(L,\rho) - \lambda_\ell | \le 4\epsilon/ (\ell_0(\epsilon))^2\text{ }(\forall \ell=1,\ldots,\ell_0(\epsilon), \forall \Asterisk \in \{-,+\} )
    \end{equation*}
    \begin{equation*}
        \Rightarrow \left| \sum_{\ell=1}^{\ell_0(\epsilon)} (\ell-1) \lambda_\ell(L,\rho) - \sum_{\ell=1}^{\ell_0(\epsilon)} (\ell-1) \lambda_\ell \right| \le \sum_{\ell=1}^{\ell_0(\epsilon)} \ell_0(\epsilon) 4 \epsilon/(\ell_0(\epsilon))^2 \le 4 \epsilon. 
    \end{equation*}
    \end{enumerate}

    Now, considering any $\epsilon < \nicefrac{1}{5} \sum_{\ell=1}^\infty \ell \lambda_\ell$, $L \ge L_0(\epsilon)$ and $\rho \le \rho_1(\epsilon,L) \wedge \rho_2(\epsilon) \wedge \rho_3(\epsilon,L)$, we evaluate the quantity of interest, ${\hat\mu(Z^L_{\Gamma_\rho} \ge 1)}/{L \hat\mu(\Gamma_\rho)}$, starting with its numerator:

    \begin{equation*}
      \hat\mu(Z^L_{\Gamma_\rho} \ge 1) = \int_\Omega \mu_\omega(Z^{\omega,L}_{\Gamma_\rho} \ge 1) d \mathbb{P}(\omega) = \int_\Omega \mu_\omega\left(\bigcup_{j=0}^{L-1} (T_\omega^j)^{-1} \Gamma_\rho(\theta^j \omega) \right) d \mathbb{P}(\omega)
    \end{equation*}
    \begin{equation*}
        \stackrel{(\star)}{=} \int_\Omega \sum_{j=0}^{L-1} \mu_\omega((T_\omega^j)^{-1} \Gamma_\rho(\theta^j \omega) ) d \mathbb{P}(\omega)  - \int_\Omega \sum_{\ell=0}^{L-1} \ell \mu_\omega(Z^{\omega,L}_{\Gamma_\rho} = \ell+1) d \mathbb{P}(\omega)
    \end{equation*}
    \begin{equation*}
        = L \hat\mu(\Gamma_\rho) - \int_\Omega \left( \sum_{\ell=0}^{L-1}  \ell \lambda^\omega_{\ell+1}(L,\rho) \right) \mu_\omega(Z^{\omega,L}_{\Gamma_\rho}>0) d \mathbb{P}(\omega)
    \end{equation*}
    \begin{eqnarray*}
        \displaystyle = L \hat\mu(\Gamma_\rho) & \displaystyle  - \int_\Omega \left( \sum_{\ell=0}^{\ell_0(\epsilon)-1}  \ell \lambda^\omega_{\ell+1}(L,\rho) \right) \mu_\omega(Z^{\omega,L}_{\Gamma_\rho}>0) d \mathbb{P}(\omega) \\
        & \displaystyle  - \int_\Omega \left( \sum_{\ell=\ell_0(\epsilon)}^{\infty}  \ell \lambda^\omega_{\ell+1}(L,\rho) \right) \mu_\omega(Z^{\omega,L}_{\Gamma_\rho}>0) d \mathbb{P}(\omega)
    \end{eqnarray*}\vspace{-5mm}
    \begin{equation*}
         = L \hat\mu(\Gamma_\rho)  - \sum_{\ell=0}^{\ell_0(\epsilon)-1} \ell \hat\mu(Z^L_{\Gamma_\rho} = \ell+1)  - \int_\Omega \left( \sum_{\ell=\ell_0(\epsilon)}^{\infty}  \ell \lambda^\omega_{\ell+1}(L,\rho) \right) \mu_\omega(Z^{\omega,L}_{\Gamma_\rho}>0) d \mathbb{P}(\omega)
    \end{equation*}
    \begin{eqnarray*}
        = L \hat\mu(\Gamma_\rho) & - &\left(\sum_{\ell=1}^{\ell_0(\epsilon)} (\ell-1) \lambda_{\ell}(L,\rho) \right)\hat\mu(Z^L_{\Gamma_\rho}>0)\\
        & - & \int_\Omega \left( \sum_{\ell=\ell_0(\epsilon)+1}^{\infty}  (\ell-1) \lambda^\omega_{\ell}(L,\rho) \right) \mu_\omega(Z^{\omega,L}_{\Gamma_\rho}>0) d \mathbb{P}(\omega)
    \end{eqnarray*}
    where $(\star)$ applied a typical Venn diagram argument using overcounting and correction.

    Then we consider the following two estimates.

    First, we have that:
    \begin{equation*}
     \sum_{\ell=1}^{\ell_0(\epsilon)} (\ell-1) \lambda_{\ell}(L,\rho) \stackrel{\text{(v)}}{\le} \sum_{\ell=1}^{\ell_0(\epsilon)} (\ell-1) \lambda_{\ell} + 4\epsilon \le \sum_{\ell=1}^{\infty} (\ell-1) \lambda_{\ell} + 5\epsilon \text{ and}
    \end{equation*}
    \begin{eqnarray*}
    \sum_{\ell=1}^{\ell_0(\epsilon)} (\ell-1) \lambda_{\ell}(L,\rho) & \stackrel{\text{(v)}}{\ge} & \sum_{\ell=1}^{\ell_0(\epsilon)} (\ell-1) \lambda_{\ell} - 4\epsilon = \sum_{\ell=1}^{\infty} (\ell-1) \lambda_{\ell} - \sum_{\ell=\ell_0(\epsilon)+1}^\infty(\ell-1) \lambda_{\ell} - 4\epsilon \\
    & \stackrel{\text{(i.a)}}{\ge} & \sum_{\ell=1}^{\infty} (\ell-1) \lambda_{\ell} - 5 \epsilon.
    \end{eqnarray*}

    Second, with $\upsilon^\omega_{\Gamma_\rho} (x) = \inf\{j\ge 0: T_\omega^j \in \Gamma_\rho(\theta^j \omega)\}$, we have that:
    \begin{equation*}
        0 \le \int_\Omega \left( \sum_{\ell=\ell_0(\epsilon)+1}^{\infty}  (\ell-1) \lambda^\omega_{\ell}(L,\rho) \right) \mu_\omega(Z^{\omega,L}_{\Gamma_\rho}>0) d \mathbb{P}(\omega) \le  \sum_{\ell=\ell_0(\epsilon)+1}^{L}  \ell \hat\mu(Z^L_{\Gamma_\rho}=\ell) 
    \end{equation*}
    \begin{equation*}
        = \sum_{\ell=\ell_0(\epsilon)+1}^{L}  \ell \sum_{j=0}^{L-1} \hat\mu(Z^L_{\Gamma_\rho}=\ell, \upsilon_{\Gamma_\rho} =j ) \le \sum_{\ell=\ell_0(\epsilon)+1}^{L}  \ell \sum_{j=0}^{L-1} \hat\mu(Z^{L-j}_{\Gamma_\rho} \circ S^j = \ell, (S^j)^{-1} \Gamma_\rho )
    \end{equation*}
    \begin{equation*}
         = \sum_{\ell=\ell_0(\epsilon)+1}^{L}  \ell \sum_{j=0}^{L-1} \alpha_\ell(L-j,\rho) \hat\mu(\Gamma_\rho)         = \left( 
        \sum_{j=0}^{L-1} \sum_{\ell=\ell_0(\epsilon)+1}^{L}  \ell  \alpha_\ell(L-j,\rho) 
        \right) \hat\mu(\Gamma_\rho) 
    \end{equation*}
    \begin{equation*}
    \le  \left[ 
        \sum_{j=0}^{L-1} \left(\sum_{\ell=\ell_0(\epsilon)+1}^{L} \hat\alpha_\ell(L-j,\rho) \right) + \ell_0(\epsilon) \hat\alpha_{\ell_0(\epsilon)}(L-j,\rho)- L\hat\alpha_{L+1}(L-j,\rho)
        \right] \hat\mu(\Gamma_\rho) 
    \end{equation*}
    \begin{equation*}
        \le \left[\sum_{j=0}^{L-1} \sum_{\ell=\ell_0(\epsilon)+1}^{L} \ell \hat\alpha_\ell(L-j,\rho) \right] \hat\mu(\Gamma_\rho) \stackrel{\text{(ii)}}{\le} 2 \epsilon L  \hat\mu(\Gamma_\rho)
    \end{equation*}

Combining what we got so far, it follows that:
\begin{eqnarray*}
    \frac{\hat\mu(Z^L_{\Gamma_\rho}\ge 1)}{L\hat\mu(\Gamma_\rho)} & \le & \frac{L\hat\mu(\Gamma_\rho) - \left( \sum_{\ell=1}^\infty (\ell-1) \lambda_\ell - 5 \epsilon \right) \hat\mu(Z^L_{\Gamma_\rho}\ge 1) }{L\hat\mu(\Gamma_\rho)}\\
    & = & 1 - \left( \sum_{\ell=1}^\infty \ell \lambda_\ell -1 - 5 \epsilon \right) \frac{\hat\mu(Z^L_{\Gamma_\rho}\ge 1)}{L\hat\mu(\Gamma_\rho)}\\
    \Rightarrow \frac{\hat\mu(Z^L_{\Gamma_\rho}\ge 1)}{L\hat\mu(\Gamma_\rho)} & \le & \frac{1}{\sum_{\ell=1}^\infty \ell \lambda_\ell - 5 \epsilon}
\end{eqnarray*}
and
\begin{eqnarray*}
     \frac{\hat\mu(Z^L_{\Gamma_\rho}\ge 1)}{L\hat\mu(\Gamma_\rho)} & \ge & \frac{L\hat\mu(\Gamma_\rho)- \left(\sum_{\ell=1}^\infty (\ell-1) \lambda_\ell + 5 \epsilon \right) \hat\mu(Z^L_{\Gamma_\rho} \ge 1) - 2 \epsilon L \hat\mu(\Gamma_\rho)  }{L\hat\mu(\Gamma_\rho)}\\
     & = & 1 - \left(\sum_{\ell=1}^\infty \ell \lambda_\ell -1 + 5 \epsilon \right)  \frac{\hat\mu(Z^L_{\Gamma_\rho}\ge 1)}{L\hat\mu(\Gamma_\rho)} - 2 \epsilon\\
     \Rightarrow  \frac{\hat\mu(Z^L_{\Gamma_\rho}\ge 1)}{L\hat\mu(\Gamma_\rho)} & \ge & \frac{1 - 2\epsilon}{\sum_{\ell=1}^\infty \ell \lambda_\ell + 5 \epsilon }
\end{eqnarray*}

Considering the final two inequalities and passing $\lim_{\epsilon \rightarrow 0}\varlimsup_{L \rightarrow \infty} \varlimsup_{\rho \rightarrow 0}$ we observe that
    \begin{equation*}
       \varlimsup_{L \rightarrow \infty} \varlimsup_{\rho \rightarrow 0}  \frac{\hat\mu(Z^L_{\Gamma_\rho}\ge 1)}{L\hat\mu(\Gamma_\rho)} = ({\textstyle \sum_{k=1}^\infty k \lambda_k})^{-1} = \alpha_1
    \end{equation*}

Alternating between $\limsup$'s and $\liminf$'s lets us reach the first desired conclusion.

Finally, to take care of the second desired conclusion, it suffices to note that
\begin{equation*}
    \frac{\hat\mu(Z^L_{\Gamma_\rho} = n)}{L\hat\mu(\Gamma_\rho)} = \frac{\hat\mu(Z^L_{\Gamma_\rho} \ge 1)}{L \hat\mu(\Gamma_\rho)}\frac{\hat\mu(Z^L_{\Gamma_\rho} = n)}{\hat\mu(Z^L_{\Gamma_\rho} >0)},
\end{equation*}then take the appropriate limits and apply the first conclusion we have just proved (to obtain $\alpha_1$), together with the definition of $\lambda_n$.
\end{proof}

\begin{lem}\label{lem:variance} Let $(\theta, \mathbb{P}, T_\omega, \mu_\omega, \Gamma)$ be a system satisfying  
 (H\ref{hyp:amb}), (H\ref{hyp:weakhypplaincyl}), (H\ref{hyp:quenchedsep}), (H\ref{hyp:lip}), (H\ref{hyp:ball}), (H\ref{hyp:annulus}),  (H\ref{hyp:quenchdec}), (H\ref{hyp:annealdec}) and (H\ref{hyp:paramincenter}).

Then: $\forall t >0, \forall n \ge 1$, $ \forall L \ge 1$, $\exists \rho_{\operatorname{var}}(L) >0$, $\forall \rho \le \rho_{\operatorname{var}}(L)$ small enough so that $N := \lfloor \frac{t}{\hat\mu(\Gamma_\rho)} \rfloor \ge 3$ and $N' := \frac{N}{L} \in \mathbb{N}_{\ge 3}$\footnote{See footnote \ref{note:fractional}.}, one has:
\begin{equation*}
\operatorname{var}_\mathbb{P}(\mathfrak{W}_\rho) \le C_{t,L} \cdot \rho^q \text{, }\forall q \in \big(0,q(d_0,d_1,\eta,\beta,\mathfrak{p}) \big),
\end{equation*}
where
\begin{equation*}
    \mathfrak{W}_\rho (\omega) := \sum_{j=0}^{N'-1} \mu_\omega(Z^{\omega}_j = n), Z^{\omega}_j := \sum_{l=jL}^{(j+1)L-1} \mathbbm{1}_{\Gamma_\rho(\theta^l \omega)} \circ T^l_\omega 
\end{equation*}and $q(d_0,d_1,\eta,\beta,\mathfrak{p})$ is a positive quantity to be presented in the proof (which can be written explicitly).
\end{lem}

\begin{proof} Let $t, n$ and $ L$ be as in the statement. Fix $\alpha \in (0,1)$. Set $\rho_{\operatorname{var}}(L) \le \rho_{\operatorname{sep}}(L) \wedge \rho_{\operatorname{dim}}$ small enough so that ${N}^\alpha < N'$. Consider $\rho \le \rho_{\operatorname{var}}(L)$ as in the statement.

For a given $j \in [0,N'-1]$, write $\omega' = \theta^{jL} \omega$ and notice that
\begin{eqnarray*}
    &\displaystyle \mathbb{E}_\mathbb{P}(\mathfrak{W}_\rho)  
    = \sum_{j=0}^{N'-1} \mathbb{E}_\mathbb{P}\left( \mu_\omega(Z^{\omega}_j = n)\right) 
    = \sum_{j=0}^{N'-1} {\mathbb{E}}_\mathbb{P} \left(\mu_{\omega'}({\textstyle \sum\nolimits_{i=0}^{L-1} \mathbbm{1}_{\Gamma_\rho(\theta^i \omega')} \circ T_{\omega'}^i {=}n )}\right)\\
   & \displaystyle  = \sum_{j=0}^{N'-1}  \mathbb{E}_\mathbb{P} (\mu_\omega(Z^{\omega}_0=n) )  = \sum_{j=0}^{N'-1} \hat\mu(Z_0 =n ) = N' \hat\mu(Z_0 =n ).
\end{eqnarray*}

Now fix $\Delta := {N}^\alpha < N'$. Then:
\begin{eqnarray*}
    \mathbb{E}_\mathbb{P}({\mathfrak{W}_\rho}^2) & = & \sum_{i,j=0}^{N'-1} \int_\Omega \mu_\omega(Z^{\omega}_i=n) \mu_\omega(Z^{\omega}_j=n) d \mathbb{P}(\omega)\\
    & = & 2 \sum_{i=0}^{N'-1} \sum_{j=i}^{(i+\Delta) \wedge (N'-1)}  \int_\Omega \mu_\omega(Z^{\omega}_i=n) \mu_\omega(Z^{\omega}_j=n) d \mathbb{P}(\omega) \\
    & + & 2 \sum_{i=0}^{N'-1} \sum_{j=(i+\Delta) \wedge (N'-1)+1}^{N'-1}  \int_\Omega \mu_\omega(Z^{\omega}_i=n) \mu_\omega(Z^{\omega}_j=n) d \mathbb{P}(\omega) \\
    & =: & (I) + (II).
\end{eqnarray*}

Immediately we get that
\begin{equation*}
    \mu_\omega(Z^{\omega}_j=n) \le \mu_\omega(Z^{\omega}_j \ge 1) \le \sum_{l=jL}^{(j+1)L-1} \mu_{\theta^l\omega}(\Gamma_\rho(\theta^l \omega)) \stackrel{\text{(H\ref{hyp:meas}.2)}}{\lesssim} L \rho^{d_0}
\end{equation*}
\begin{equation*}
    \Rightarrow (I) \lesssim L \rho^{d_0} \Delta \mathbb{E}_\mathbb{P}(\mathfrak{W}_\rho) =  \Delta \rho^{d_0} N \hat\mu(Z_0 =n ).
\end{equation*}

Most of the remaining work is to control component $(II)$.

Fix $\omega \in \Omega$ and, for a given $i \in [0, N' -1]$, write $\omega' = \theta^{iL} \omega$. Moreover, consider $r \in (0, \rho/2)$, $v \in [0,L-1]$ and denote by
\begin{equation}\label{eq:targetballs}
    U_{v,\omega'} = \Gamma_\rho(\theta^v \omega'), \text{ }\stackrel{-}{U}_{v,r,\omega'} =  B_r({U_{v,\omega'}}^c)^c,\text{ } \stackrel{+}{U}_{v,r,\omega'} = B_r(U_{v,\omega'}),
\end{equation}respectively, the $\rho$-sized target with noise $\omega'$ $v$-steps ahead; its diminishment by radius $r$; and its enlargement by radius $r$. They relate as $\stackrel{-}{U}_{v,r,\omega'}\subset U_{v,\omega'} \subset \stackrel{+}{U}_{v,r,\omega'}$.

Moreover, dynamical counterparts of those in equation (\ref{eq:targetballs}) are denote by 
\begin{equation*}
\begin{tabular}{c  c  c c @{} l}
    $\displaystyle \{Z^{\omega'}_0 = n\}$ & $\displaystyle =$ & $\displaystyle \mathcal{U}_{\omega'}$ & $\displaystyle =$ & $\displaystyle \bigsqcup_{0 \le v_1 < \ldots < v_n \le L-1}\left( \bigcap_{l=1}^n (T^{v_l}_{\omega'})^{-1} U_{v_l,\omega'} \text{ } \hspace{3mm} \cap  \bigcap_{\substack{v \in [0,L-1] \\ \setminus \{v_l : l=1,\ldots,n\} }} (T^v_{\omega'})^{-1} {U_{v,\omega'}}^c \right),$\\
     &  & $\displaystyle \stackrel{-}{\mathcal{U}}_{r,\omega'}$ & $\displaystyle =$ & $\displaystyle \bigsqcup_{0 \le v_1 < \ldots < v_n \le L-1}\left( \bigcap_{l=1}^n (T^{v_l}_{\omega'})^{-1} \stackrel{-}{U}_{v_l,\omega'} \text{ } \cap  \bigcap_{\substack{v \in [0,L-1] \\ \setminus \{v_l : l=1,\ldots,n\} }} (T^v_{\omega'})^{-1} {\stackrel{+}{U}_{v,\omega'}}^c \right),$\\
     &  & $\displaystyle \stackrel{+}{\mathcal{U}}_{r,\omega'}$ & $\displaystyle =$ & $\displaystyle \bigsqcup_{0 \le v_1 < \ldots < v_n \le L-1}\left( \bigcap_{l=1}^n (T^{v_l}_{\omega'})^{-1} \stackrel{+}{U}_{v_l,\omega'} \text{ } \cap  \bigcap_{\substack{v \in [0,L-1] \\ \setminus \{v_l : l=1,\ldots,n\} }} (T^v_{\omega'})^{-1} {\stackrel{-}{U}_{v,\omega'}}^c \right),$
\end{tabular}
\end{equation*}describing
\begin{itemize}
    \item[-]  the locus of points which hit the $\rho$-sized target exactly $n$ times during the time interval $[0,L-1]$ when given the noise $\omega'$;
    \item[-] the diminishment of the first by radius $r$, in the sense that hits are considered in a $r$-stringent way (at least $r$-inside the $\rho$-sized target) and non-hits are considered in a $r$-stringent way (at least $r$-away from the $\rho$-sized target);
    \item[-] the enlargement of the first by radius $r$, in the sense that hits are considered in a $r$-permissive way (at most $r$-away from the $\rho$-sized target) and non-hits are considered in a $r$-permissive way (at most $r$-inside the $\rho$-sized target).
\end{itemize}They relate as $\stackrel{-}{\mathcal{U}}_{r,\omega'} \subset \mathcal{U}_{\omega'} \subset \stackrel{+}{\mathcal{U}}_{r,\omega'}$.

Finally, define
\begin{eqnarray*}
    \stackrel{-}{\phi} \hspace{-3mm}\phantom{\phi}_{r}^{\omega'} \hspace{-0.5mm}(x) {=} \hspace{-0.5mm} \begin{cases}
        1, x \in \hspace{0.7mm}\stackrel{-}{\mathcal{U}}_{r,\omega'}\\ 0, x \in {\mathcal{U}_{\omega'}}^c \\
        \frac{d_M(x,{\mathcal{U}_{\omega'}}^c)}{ d_M(x,{\mathcal{U}_{\omega'}}^c) + d_M(x,\stackrel{-}{\mathcal{U}}_{r,\omega'})}\text{,}\hspace{0.5mm} x \hspace{-0.35mm}\in {\mathcal{U}_{\omega'}}^c \setminus \stackrel{-}{\mathcal{U}}_{r,\omega'}

    \end{cases} \hspace{-5mm} \stackrel{+}{\phi} \hspace{-3mm}\phantom{\phi}_{r}^{\omega'} \hspace{-0.5mm}(x) {=}  \hspace{-0.5mm} \begin{cases}
        1, x \in \mathcal{U}_{\omega'}\\ 0, x \in \hspace{0.7mm} \stackrel{+}{\mathcal{U}} \hspace{-4mm}{\phantom{\mathcal{U}}_{r,\omega'}}^c \\ 
        \frac{d_M(x,\stackrel{+}{\mathcal{U}} \hspace{-2mm}{\phantom{\mathcal{U}}_{r,\omega'}}^c)}{ d_M(x,\stackrel{+}{\mathcal{U}} \hspace{-2mm}{\phantom{\mathcal{U}}_{r,\omega'}}^c) + d_M(x,\mathcal{U}_{\omega'})}\text{,}\hspace{0.5mm} x \hspace{-0.35mm}\in \hspace{0.5mm} \stackrel{+}{\mathcal{U}} \hspace{-4mm}{\phantom{\mathcal{U}}_{r,\omega'}} \setminus \mathcal{U}_{\omega'}
    \end{cases}\hspace{-6mm}.
\end{eqnarray*}They relate as $\stackrel{-}{\phi} \hspace{-3mm}\phantom{\phi}_{r}^{\omega'} \le \mathbbm{1}_{\mathcal{U}_{\omega'}} \le \stackrel{+}{\phi} \hspace{-3mm}\phantom{\phi}_{r}^{\omega'}$. 

Using that $\operatorname{Lip}_{d_M}\big(d_M(x,\stackrel{+}{\mathcal{U}} \hspace{-4mm}{\phantom{\mathcal{U}}_{r,\omega'}})\big), \operatorname{Lip}_{d_M}\big(d_M(x, \mathcal{U}_{\omega'})\big) \le 1$, it can be checked that
\begin{equation*}
    \operatorname{Lip}_{d_M}(\stackrel{+}{\phi} \hspace{-3.4mm}\phantom{\phi}_{r}^{\omega'}) \le \frac{6 \operatorname{diam}(M)}{\Big(\min_{x \in M} [d_M(x,\stackrel{+}{\mathcal{U}} \hspace{-4mm}{\phantom{\mathcal{U}}_{r,\omega'}})+ d_M(x, \mathcal{U}_{\omega'})] \Big)^2 } \le \frac{6 \operatorname{diam}(M)}{d_{min}\big(\mathcal{U}_{\omega'}, \stackrel{+}{\mathcal{U}} \hspace{-4mm}{\phantom{\mathcal{U}}_{r,\omega'}}^c \big)^2},
\end{equation*}where $d_{min}\big(\mathcal{U}_{\omega'}, \stackrel{+}{\mathcal{U}} \hspace{-4mm}{\phantom{\mathcal{U}}_{r,\omega'}}^c \big) := \inf\{d_M(x,y) : x \in \mathcal{U}_{\omega'}, y \in \stackrel{+}{\mathcal{U}} \hspace{-4mm}{\phantom{\mathcal{U}}_{r,\omega'}}^c \}$. 

Notice that for a point $x \in \mathcal{U}_{\omega'}$ to be minimally-displaced in such a way as to reach $\stackrel{+}{\mathcal{U}} \hspace{-4mm}{\phantom{\mathcal{U}}_{r,\omega'}}^c$, either: a) some of the hits in its finite-orbit is consequently-displaced to an extent which now makes it at least $r$-away from associated $\rho$-sized target, or b) some of the non-hits in its finite-orbit is consequently-displaced to an extent which now makes it at least $r$-inside the associated $\rho$-sized target. In either case, the associated image point of $x$ has to be consequently-displaced by distance at least $r$. When the said image point being consequently-displaced happens to be the last one in the orbit of $x$, i.e., its $L-1$ iterate, by the expanding feature of the system (H\ref{hyp:weakhypplaincyl}) and since $\mathcal{U}_{\omega'} \subset \bigcup_{j=0}^{L-1} (T_{\omega'}^j)^{-1} \Gamma_{3/2 \rho}(\theta^j \omega') \stackrel{\text{(H\ref{hyp:quenchedsep})}}{\subset} \hspace{1mm} \stackrel{+}{\mathcal{C}} \hspace{-4mm} \phantom{C}_{L-1}^{\omega'}$, this is when $x$ would have to be displaced the least: no more than $r/a_{L-1}$ (use (H\ref{hyp:quenchedsep}) and (H\ref{hyp:bddderiv})). Therefore $r/a_{L-1} \le d_{min}\big(\mathcal{U}_{\omega'}, \stackrel{+}{\mathcal{U}} \hspace{-4mm}{\phantom{\mathcal{U}}_{r,\omega'}}^c \big)$, and so
\begin{equation*}
    \begin{tabular}{c}
    $\operatorname{Lip}_{d_M}(\stackrel{+}{\phi} \hspace{-3mm}\phantom{\phi}_{r}^{\omega'}) \le 6 \operatorname{diam}(M) {a_{L-1}}^2/r^2$, \\
    $\| \hspace{-1mm} \stackrel{+}{\phi} \hspace{-3mm}\phantom{\phi}_{r}^{\omega'} \|_{Lip_{d_M}} \hspace{-0.25mm}=\hspace{-0.25mm} \| \hspace{-1mm} \stackrel{+}{\phi} \hspace{-3mm}\phantom{\phi}_{r}^{\omega'} \|_\infty {\vee} \operatorname{Lip}_{d_M}(\stackrel{+}{\phi} \hspace{-3mm}\phantom{\phi}_{r}^{\omega'}) \hspace{-0.25mm}=\hspace{-0.25mm} 1 {\vee} \operatorname{Lip}_{d_M}(\stackrel{+}{\phi} \hspace{-3mm}\phantom{\phi}_{r}^{\omega'}) \hspace{-0.25mm}=\hspace{-0.25mm}  \operatorname{Lip}_{d_M}(\stackrel{+}{\phi} \hspace{-3mm}\phantom{\phi}_{r}^{\omega'}) \hspace{-0.25mm}\le\hspace{-0.25mm} 6 \operatorname{diam}(M) {a_{L-1}}^2/r^2$,
    \end{tabular}
\end{equation*}
where the last equality follows from $\rho$ sufficiently small. 

Now we start looking at $(II)$ directly:
\begin{equation*}
    \left| \int_\Omega \mu_\omega(Z^{\omega}_j=n) \mu_\omega(Z^{\omega}_i=n) d \mathbb{P}(\omega) - \int_\Omega \mu_\omega(Z^{\omega}_j=n) \mu_{\omega'}(\stackrel{+}{\phi} \hspace{-3mm}\phantom{\phi}_{r}^{\omega'}) d \mathbb{P}(\omega) \right|
\end{equation*}
\begin{equation*}
    = \left| \int_\Omega \mu_\omega(Z^{\omega}_j=n) \mu_{\omega'}(\mathbbm{1}_{\mathcal{U}_{\omega'}}) d \mathbb{P}(\omega) - \int_\Omega \mu_\omega(Z^{\omega}_j=n) \mu_{\omega'}(\stackrel{+}{\phi} \hspace{-3mm}\phantom{\phi}_{r}^{\omega'}) d \mathbb{P}(\omega) \right|
\end{equation*}
\begin{equation*}
    \lesssim \int_\Omega \mu_{\theta^j \omega} (Z^{\theta^j\omega}_0 = n) L \frac{r^\eta}{\rho^\beta} d \mathbb{P}(\omega)=  L \frac{r^\eta}{\rho^\beta} \hat\mu(Z_0 = n),
\end{equation*}where the $\lesssim$ is because
\begin{equation*}
    \mu_{\omega'}(\stackrel{+}{\phi} \hspace{-3mm}\phantom{\phi}_{r}^{\omega'}) \le \mu_{\omega'}(\stackrel{+}{\mathcal{U}}_{r,\omega'} \setminus \stackrel{-}{\mathcal{U}}_{r,\omega'}) \le \sum_{v=0}^{L-1} \mu_{\theta^v \omega}(\stackrel{+}{U}_{v,r,\omega'} \setminus \stackrel{-}{U}_{v,r,\omega'}) \stackrel{\text{(H\ref{hyp:annulus})}}{\lesssim} L \frac{r^\eta}{\rho^\beta}.
\end{equation*}

The approximating term that appeared above is transformed as follows:
\begin{equation*}
    \left| \int_\Omega \mu_\omega(Z^{\omega}_j=n) \mu_{\omega'}(\stackrel{+}{\phi} \hspace{-3mm}\phantom{\phi}_{r}^{\omega'}) d \mathbb{P}(\omega) - \int_\Omega \mu_{\omega'}(\mathbbm{1}_{\{Z^{ \omega'}_{j-i} =n\}} \stackrel{+}{\phi} \hspace{-3mm}\phantom{\phi}_{r}^{\omega'}  ) d \mathbb{P}(\omega) \right|
\end{equation*}
\begin{equation*}
    = \left| \int_\Omega \mu_{\omega'}(Z^{\omega'}_{j-i}=n) \mu_{\omega'}(\stackrel{+}{\phi} \hspace{-3mm}\phantom{\phi}_{r}^{\omega'}) d \mathbb{P}(\omega) - \int_\Omega \mu_{\omega'}(\mathbbm{1}_{\{Z^{ \theta^{(j-i)L} \omega'}_{0} \circ T_{\omega'}^{(j-i)L} =n\}} \stackrel{+}{\phi} \hspace{-3mm}\phantom{\phi}_{r}^{\omega'}  ) d \mathbb{P}(\omega) \right|
\end{equation*}
\begin{equation*}
    =  \int_\Omega \left| \mu_{\theta^{(j-i)L} \omega'}(Z^{\theta^{(j-i)L} \omega'}_{0}=n) \mu_{\omega'}(\stackrel{+}{\phi} \hspace{-3mm}\phantom{\phi}_{r}^{\omega'})  -  \mu_{\omega'}(\mathbbm{1}_{\{Z^{ \theta^{(j-i)L} \omega'}_{0}  =n\}} \circ T_{\omega'}^{(j-i)L} \stackrel{+}{\phi} \hspace{-3mm}\phantom{\phi}_{r}^{\omega'}  ) \right| d \mathbb{P}(\omega) 
\end{equation*}
\begin{equation*}
    \stackrel{(H\ref{hyp:quenchdec})}{\lesssim} \int_\Omega ((j-i)L)^{-\mathfrak{p}}  \| \hspace{-1mm} \stackrel{+}{\phi} \hspace{-3mm}\phantom{\phi}_{r}^{\omega'}  \|_{\operatorname{Lip}_{d_M}} d \mathbb{P} (\omega) \lesssim\footnotemark \text{ } ((j-i)L)^{-\mathfrak{p}} 
    \frac{{a_{L-1}}^2}{r^2}.
\end{equation*}\footnotetext{Notice that $\| \hspace{-1mm} \stackrel{+}{\phi} \hspace{-3mm}\phantom{\phi}_{r}^{\omega'}  \|_{\operatorname{Lip}_{d_M}} \lesssim {a_{L-1}}^2/r^2$ a.s. is enough to justify the above inequality. However,  our hypotheses imply this is true for every $\omega$. This might seem an excess, but later in the proof we will need the inequality for every $\omega$. See the next footnote.}

Whereas the new approximating term which appeared above is transformed as follows:
\begin{equation*}
    \int_\Omega \mu_{\omega'}(\mathbbm{1}_{\{Z^{ \omega'}_{j-i} =n\}} \stackrel{+}{\phi} \hspace{-3mm}\phantom{\phi}_{r}^{\omega'}  ) d \mathbb{P}(\omega) = \int_{\Omega \times M} \mathbbm{1}_{\{Z_{j-i} =n\}} \stackrel{+}{\phi} \hspace{-3mm}\phantom{\phi}_{r}   d \hat\mu = \int_{\Omega \times M} \mathbbm{1}_{\{Z_{0} =n\}} \circ S^{(j-i)L} \stackrel{+}{\phi} \hspace{-3mm}\phantom{\phi}_{r}   d \hat\mu
\end{equation*}
and
\begin{equation*}
    \left| \int_{\Omega \times M} \mathbbm{1}_{\{Z_{0} =n\}} \circ S^{(j-i)L} \stackrel{+}{\phi} \hspace{-3mm}\phantom{\phi}_{r}   d \hat\mu - \int_{\Omega \times M} \mathbbm{1}_{\{Z_0 =n\}} d \hat\mu \cdot \int_{\Omega \times M} \stackrel{+}{\phi} \hspace{-3mm}\phantom{\phi}_{r} d \hat\mu \right|
\end{equation*}
\begin{equation*}
    \stackrel{(H\ref{hyp:annealdec})}{\lesssim} ((j-i)L)^{-\mathfrak{p}} { \| \hspace{-1mm}\stackrel{+}{\phi} \hspace{-3mm}\phantom{\phi}_{r}\|_{\operatorname{Lip}_{d_{\Omega \times M}}}} \lesssim ((j-i)L)^{-\mathfrak{p}}   \frac{\displaystyle {a_{L-1}}^2}{\displaystyle r^2},
\end{equation*}where, recalling that $\hspace{-1mm} \stackrel{+}{\phi} \hspace{-3mm}\phantom{\phi}_{r}^{\omega} = \hspace{1mm}  \stackrel{+}{\phi} \hspace{-3mm}\phantom{\phi}_{r}^{\omega} (n,L,\rho)$, we have used that
\begin{equation*}
\begin{tabular}{c c l}
    $\displaystyle \operatorname{Lip}_{d_{\Omega \times M}}(  \stackrel{+}{\phi} \hspace{-3mm}\phantom{\phi}_{r}^{}  )$ & $=$ &  $\displaystyle \sup_{(\omega_1,x_1) \neq (\omega_2,x_2)}  \frac{\displaystyle | \hspace{-1mm}\stackrel{+}{\phi} \hspace{-3mm}\phantom{\phi}_{r}^{\omega_1}(x_1) - \stackrel{+}{\phi} \hspace{-3mm}\phantom{\phi}_{r}^{\omega_2}(x_2)|}{\displaystyle d_\Omega(\omega_1,\omega_2) \vee d_M(x_1,x_2)}$\\
    & $\le$ & $\displaystyle \sup_{x_1}\sup_{\omega_1 \neq \omega_2}  \frac{\displaystyle | \hspace{-1mm}\stackrel{+}{\phi} \hspace{-3mm}\phantom{\phi}_{r}^{\omega_1}(x_1) - \stackrel{+}{\phi} \hspace{-3mm}\phantom{\phi}_{r}^{\omega_2}(x_1)|}{\displaystyle d_\Omega(\omega_1,\omega_2)} + \sup_{\omega_2} \sup_{x_1 \neq x_2} \frac{\displaystyle| \hspace{-1mm}\stackrel{+}{\phi} \hspace{-3mm}\phantom{\phi}_{r}^{\omega_2}(x_1) - \stackrel{+}{\phi} \hspace{-3mm}\phantom{\phi}_{r}^{\omega_2}(x_2)|}{\displaystyle d_M(x_1,x_2)} $ \\
    & $\le\footnotemark$ & $\displaystyle \frac{\displaystyle {a_{L-1}}^2}{\displaystyle r^2} +  \sup_{x} \sup_{\omega_1 \neq \omega_2}  \frac{\displaystyle | \hspace{-1mm}\stackrel{+}{\phi} \hspace{-3mm}\phantom{\phi}_{r}^{\omega_1}(x) - \stackrel{+}{\phi} \hspace{-3mm}\phantom{\phi}_{r}^{\omega_2}(x)|}{\displaystyle d_\Omega(\omega_1,\omega_2)}$\\
    & $\stackrel{(\star)}{\le}$ & $\displaystyle \frac{\displaystyle {a_{L-1}}^2}{\displaystyle r^2} + \frac{\displaystyle (\alpha^L \beta+ \gamma) {{a_{L-1}}^2} }{\displaystyle r^2} \lesssim \frac{\alpha^L {a_{L-1}}^2}{r^2}$,
\end{tabular}
\end{equation*}\footnotetext{Here one needs $\| \hspace{-1mm} \stackrel{+}{\phi} \hspace{-3mm}\phantom{\phi}_{r}^{\omega'}  \|_{\operatorname{Lip}_{d_M}} \le {a_{L_1}}^2/r^2$ for every $\omega$. See the previous footnote.}with $(\star)$ following from $\omega \mapsto \stackrel{+}{\phi} \hspace{-3mm}\phantom{\phi}_{r}^{\omega}(x)$ being a locally Lipschitz function whose associated local Lipschitz constants are bounded by $\frac{ (\alpha^L \beta+ \gamma) {{a_{L-1}}^2} }{ r^2}$, where $\alpha = \operatorname{Lip}(\theta) \vee 1$, $\beta = \operatorname{Lip}(\Gamma) \vee 1$, $\gamma = \sup_{x \in M} \operatorname{Lip}(T_\cdot(x) : \Omega \to M)$. This is verified in the following paragraph.

Fix $x \in M$ and consider $\omega \in \Omega$. In case $x \in \operatorname{int} (\mathcal{U}_{\omega})$ (or $\operatorname{int }(\stackrel{+}{\mathcal{U}} \hspace{-4mm}{\phantom{\mathcal{U}}_{r,\omega}}^c)$), there is $u_x(\omega)>0$ so that $\tilde\omega \in B_{u_x(\omega)}(\omega)$ implies $x \in \operatorname{int} \mathcal{U}_{\tilde\omega} \hspace{0.7mm}$ (or $\operatorname{int }\stackrel{+}{\mathcal{U}} \hspace{-4mm}{\phantom{\mathcal{U}}_{r,\tilde\omega}}^c$), so the function of interest is locally constant. In case $x \in \operatorname{int}(\stackrel{+}{\mathcal{U}} \hspace{-4mm}{\phantom{\mathcal{U}}_{r,\omega}}^c \setminus \mathcal{U}_{\omega})$, it boils down to understand how the linear interpolation within $\stackrel{+}{\phi} \hspace{-3mm}\phantom{\phi}_{r}^{}$ varies with $\tilde\omega \in B_{u'_x(\omega)}(\omega)$,  where $u'_x(\omega)$ is that for which $\tilde\omega \in B_{u'_x(\omega)}(\omega)$ implies $x \in \operatorname{int}(\stackrel{+}{\mathcal{U}} \hspace{-4mm}{\phantom{\mathcal{U}}_{r,\tilde\omega}}^c \setminus \mathcal{U}_{\tilde\omega})$. For this purpose, we first evaluate the Lipschitz constant of $\tilde\omega \in B_{u'_x(\omega)}(\omega) \mapsto d(x,\mathcal{U}_{\tilde\omega})$ and $\tilde\omega \in B_{u'_x(\omega)}(\omega) \mapsto d(x,\stackrel{+}{\mathcal{U}} \hspace{-4mm}{\phantom{\mathcal{U}}_{r,\tilde\omega}}^c)$:\\
i)
\begin{equation*}
\begin{tabular}{c}
     $|d(x,\mathcal{U}_{\omega}) - d(x,\mathcal{U}_{\tilde\omega})| \le d_H(\mathcal{U}_{\omega},\mathcal{U}_{\tilde\omega}) $\\
     $\le \left( \substack{\displaystyle (  \sup\nolimits_{\omega} \operatorname{Lip} (T_\omega ^{-1}: \mathscr{P}(M) \to \mathscr{P}(M) ) \vee 1)^L \cdot (\operatorname{Lip}(\Gamma) \vee 1) \cdot (\operatorname{Lip}(\theta)\vee 1)^L \\  \displaystyle + \sup\limits_{A \in \mathscr{P}(M)}  \operatorname{Lip} ({T_\cdot}^{-1} A : \Omega \to \mathscr{P}(M))} \right) d_\Omega(\omega,\tilde\omega)$ \\
     $\le (\alpha^L \beta + \gamma) d_\Omega(\omega,\tilde\omega).$
\end{tabular}
\end{equation*}
since $$\operatorname{Lip}\left(\bigcap{:} \mathscr{P}(M) {\times} \mathscr{P}(M) {\to} \mathscr{P}(M) \right) {\le} 1,$$ $$\operatorname{Lip}\left(\bigcup{:} \mathscr{P}(M) {\times} \mathscr{P}(M) {\to} \mathscr{P}(M) \right) {\le} 1,$$ $$\operatorname{Lip} \left(B_\rho{:} \mathscr{P}(M) {\to} \mathscr{P}(M)\right) \le 1,$$
\begin{equation*}
    \sup\nolimits_{\omega} \operatorname{Lip} \left(T_\omega ^{-1}: \mathscr{P}(M) \to \mathscr{P}(M) \right) \le 1 / \inf_{\omega \in \Omega} \inf_{\xi \in C_1^\omega} \operatorname{CoLip}(T_\omega|_\xi: \xi \to M)  \le 1
\end{equation*}and
\begin{equation*}
    \sup\limits_{A \in \mathscr{P}(M)}  \operatorname{Lip} \left({T_\cdot}^{-1} A {:} {\Omega} {\to} {\mathscr{P}}(M)\right) {\le} \frac{\displaystyle \sup_{x \in M} \operatorname{Lip}(T_\cdot(x) : \Omega \to M)}{ \displaystyle \inf_{\omega \in \Omega} \inf_{\xi \in C^\omega_1} \operatorname{CoLip}(T_\omega|_\xi: \xi \to M) } {\le}  \sup_{x \in M} \operatorname{Lip}(T_\cdot(x) {:} {\Omega} {\to} {M}),
\end{equation*}where $\operatorname{CoLip(T)} = \inf_{x \neq y} \frac{d(Tx,Ty)}{d(x,y)}$.\\
ii) Similarly,
\begin{equation*}
    |d(x,\stackrel{+}{\mathcal{U}} \hspace{-4mm}{\phantom{\mathcal{U}}_{r,\omega}}^c) - d(x,\stackrel{+}{\mathcal{U}} \hspace{-4mm}{\phantom{\mathcal{U}}_{r,\tilde\omega}}^c)| \le (\alpha^L \beta + \gamma) d_\Omega(\omega,\tilde\omega),
\end{equation*}since also $\operatorname{Lip}(B_r{:} \mathscr{P}(M) {\to} \mathscr{P}(M)) \le 1$. 

To conclude justifying $(\star)$, one repeats the calculations for the Lipschitz constant of a quotient and applies (i) and (ii) to get that 
\begin{equation*}
\begin{tabular}{c c l}
     $\displaystyle \operatorname{Lip}_{d_
     \Omega}\left( \frac{\displaystyle d_M(x,\stackrel{+}{\mathcal{U}} \hspace{-4mm}{\phantom{\mathcal{U}}_{r,\omega'}}^c)}{\displaystyle  d_M(x,\stackrel{+}{\mathcal{U}} \hspace{-4mm}{\phantom{\mathcal{U}}_{r,\omega'}}^c) + d_M(x,\mathcal{U}_{\omega'})} \right)$ & $\le$ & $\displaystyle \frac{\displaystyle 4 \operatorname{diam}(M) (\alpha^L \beta+ \gamma)  }{\displaystyle d_{min}(\stackrel{+}{\mathcal{U}} \hspace{-4mm}{\phantom{\mathcal{U}}_{r,\omega'}}^c,\mathcal{U}_{\omega'}) ^2}  d_\Omega(\omega,\tilde\omega)$\\
     & $\lesssim$ & $\displaystyle \frac{\displaystyle ( \alpha^L \beta+ \gamma) {{a_{L-1}}^2} }{\displaystyle r^2} d_\Omega(\omega,\tilde\omega).$
\end{tabular}
\end{equation*}

Finally, we notice that
\begin{equation*}
    \left| \hat\mu(Z_0 {=}n)  \hat\mu(\stackrel{+}{\phi} \hspace{-3mm}\phantom{\phi}_{r}) {-} \hat\mu(Z_0 {=}n)^2 \right| {\le} \hat\mu(Z_0 {=}n) \int_\Omega \mu_\omega(\stackrel{+}{\phi} \hspace{-3mm}\phantom{\phi}_{r}^\omega - \mathbbm{1}_{\mathcal{U}_\omega}) d \mathbb{P}(\omega) \stackrel{\text{(H\ref{hyp:annulus})}}{\lesssim} L \frac{r^\eta}{\rho^\beta} \hat\mu(Z_0{=}n).
\end{equation*}

Combining the previous four steps, we arrive at
\begin{equation*}
    \left|  \int_\Omega \mu_\omega(Z^{\omega}_j=n) \mu_\omega(Z^{\omega}_i=n) d \mathbb{P}(\omega) - \hat\mu(Z_0 =n)^2 \right|
    \lesssim L \frac{r^\eta}{\rho^\beta} \hat\mu(Z_0=n) + ((j-i)L)^{-\mathfrak{p}}   \frac{\alpha^L {a_{L-1}}^2}{r^2},
\end{equation*}
which implies
\begin{equation*}
\begin{tabular}{c @{} c @{} l}
    $\displaystyle (II)$ & $\lesssim$ & $\displaystyle \sum_{i=0}^{N'-1} \sum_{j=(i+\Delta) \wedge (N'-1) +1}^{N'-1} \left( \hat\mu(Z_0 =n)^2 + L \frac{r^\eta}{\rho^\beta} \hat\mu(Z_0=n) + ((j-i)L)^{-\mathfrak{p}} \frac{\alpha^L {a_{L-1}}^2}{r^2}  \right)$\\
    & $\lesssim$ & $\displaystyle N' (N' - \Delta) \left(  \hat\mu(Z_0 =n)^2 + L \frac{r^\eta}{\rho^\beta} \hat\mu(Z_0=n) \right) + N' \frac{\alpha^L {a_{L-1}}^2}{r^2} (\Delta L)^{- \mathfrak{p}+1}.$
\end{tabular}
\end{equation*}

Then we can conclude the following about the variance:
\begin{eqnarray*}
\begin{tabular}{c @{} c @{} l}
$\displaystyle \operatorname{var}_\mathbb{P}(\mathfrak{W}_\rho)$ & ${=}$ & $\displaystyle \mathbb{E}_\mathbb{P}({\mathfrak{W}_\rho}^2) - (\mathbb{E}_\mathbb{P}(\mathfrak{W}_\rho ))^2$\\
    & $\lesssim$ & $\displaystyle \Delta \rho^{d_0} N \hat\mu(Z_0 =n )$\\
    & ${+}$ & $\displaystyle N' (N' {-} \Delta) \left(  \hat\mu(Z_0 {=}n)^2 {+} L \frac{r^\eta}{\rho^\beta} \hat\mu(Z_0{=}n) \right) {+} N' \frac{\alpha^L {a_{L-1}}^2}{r^2} (\Delta L)^{- \mathfrak{p}+1}$\\
    & ${-}$ &  $\displaystyle {N'}^2 \hat\mu(Z_0 =n )^2.$\\
    & $\lesssim$ & $\displaystyle \Delta \rho^{d_0} N \hat\mu(Z_0 =n ) + {N'}^2 L \frac{r^\eta}{\rho^\beta} \hat\mu(Z_0=n) + N' \frac{\alpha^L {a_{L-1}}^2}{r^2} (\Delta L)^{- \mathfrak{p}+1}$\\
    & ${\stackrel{(\star)}{{\lesssim}}}$ & $\displaystyle {N}^\alpha L \rho^{d_0} + N \frac{r^\eta}{\rho^\beta}  + \alpha^L {a_{L-1}}^2 L^{-\mathfrak{p}} \frac{N {N}^{\alpha(-\mathfrak{p}+1)} }{r^2}$\\
    & ${\stackrel{(\star \star)}{{\lesssim}}}$  & $\displaystyle \frac{t^\alpha}{\hat\mu(\Gamma_\rho)^\alpha} L \rho^{d_0} +  \frac{t}{\hat\mu(\Gamma_\rho)} \rho^{w \eta - \beta} + \alpha^L {a_{L-1}}^2 L^{-\mathfrak{p}} \frac{t}{\hat\mu(\Gamma_\rho)} \frac{t^{\alpha(-\mathfrak{p}+1)}}{\hat\mu(\Gamma_\rho)^{\alpha(-\mathfrak{p}+1)}}  \rho^{-2w}$\\
    & ${\stackrel{\text{(H\ref{hyp:ball})}}{{\lesssim}}}$  & $\displaystyle L \rho^{d_0-\alpha d_1} + \rho^{w \eta - \beta - d_1} + \alpha^L {a_{L-1}}^2 L^{-\mathfrak{p}} \rho^{\alpha d_0 (\mathfrak{p}-1) - 2w -d_1 }$\\
    & ${\stackrel{(\star \star \star)}{{\lesssim}}}$ & $\displaystyle \rho^{d_0-\alpha d_1} + \rho^{w \eta - \beta - d_1} +  \rho^{\alpha d_0 (\mathfrak{p}-1) - 2w -d_1 },$
\end{tabular}
\end{eqnarray*}where $(\star)$ uses $N' \hat\mu(Z_0 =n ) \le N L^{-1} \hat\mu(Z_0 \ge 1 ) \le N L^{-1} L \hat\mu(\Gamma_\rho)\lesssim t$ and $t$ is incorporated into the $\lesssim$ sign; $(\star \star)$ uses the choice $r := \rho^w$ for a given $w > 1$; and $(\star \star \star)$ incorporates $L$ dependent quantities on $\lesssim$. Notice that $t$ and $L$ dependent constants being incorporated inside $\lesssim$ is associated to the use of a constant $C_{t,L}$ in the statement.

Finally, we need to choose $(\alpha,w) \in (0,1) \times (1,\infty)$ so that 
\begin{equation*}
    \begin{cases}
        d_0 > \alpha d_1 \\
        w \eta > \beta + d_1\\
        \alpha d_0 (\mathfrak{p}-1) > 2w + d_1
    \end{cases} \text{i.e. } 
    \begin{cases}
        \alpha < \frac{d_0}{d_1} \wedge 1 = \frac{d_0}{d_1} \\
        w > \frac{\beta + d_1}{\eta} \vee 1\\
        w < \frac{\alpha d_0 (\mathfrak{p}-1) -d_1}{2}
    \end{cases},
\end{equation*}which admits a solution if, and only if,
\begin{equation*}
    \frac{\beta + d_1}{\eta} \vee 1 = \frac{\frac{d_0}{d_1} d_0(\mathfrak{p}-1) -d_1}{2} \Leftrightarrow d_0(\mathfrak{p}-1) > \frac{2\left(\frac{\beta + d_1}{\eta} \vee 1\right) + d_1}{d_0/d_1}.
\end{equation*}

This is guaranteed by the parametric constraint (H\ref{hyp:paramincenter}), so there exists some solution $(\alpha_*,w_*)$ to the system. Actually, the space of solutions forms a triangle and one can select $(\alpha_*,w_*)$ as its incenter, a function of $d_0,d_1,\eta,\beta$ and $\mathfrak{p}$, whereas the strictly positive margin this choice opens in the inequalities of the original system is denoted by $q(d_0,d_1,\eta,\beta,\mathfrak{p})$. With such a choice, we obtain that
\begin{equation*}
    \operatorname{var}_\mathbb{P}(\mathfrak{W}_\rho) \le C_{t,L} \cdot \rho^{q(d_0,d_1,\eta,\beta,\mathfrak{p})} \le C_{t,L} \cdot \rho^q \text{ }, \forall q \in \left( 0, q(d_0,d_1,\eta,\beta,\mathfrak{p})\right).
\end{equation*}
\end{proof}

 \begin{lem}\label{lem:asconv} Let $(\theta, \mathbb{P}, T_\omega, \mu_\omega, \Gamma)$ be a system satisfying 
 (H\ref{hyp:weakhypplaincyl}), (H\ref{hyp:quenchedsep}), (H\ref{hyp:ball}), (H\ref{hyp:annulus}),  (H\ref{hyp:quenchdec}),  (H\ref{hyp:annealdec}), (H\ref{hyp:return}') and (H\ref{hyp:paramincenter}).

 Then: $\forall t {>}0, \forall n {\ge} 1, \forall (\rho_m)_{m \ge 1} {\searrow} 0$ with $\sum_{m\ge 1} {\rho_m}^q{<}\infty$ (for 
 some $0 {<} q {<}  q(d_0,d_1,\eta,\beta,\mathfrak{p})$), denoting $N = \lfloor \frac{t}{\hat\mu(\Gamma_{\rho_m})} \rfloor$ and $N' = \frac{N }{L}$\footnote{See footnote \ref{note:fractional}.}, one has:

1)
\begin{equation*}
        \lim\limits_{L \to \infty} \varliminf_{m \to \infty} \hspace{-6.5mm}\overline{\phantom{\varliminf}}  \sum_{j=0}^{N'-1} \mathbb{\mu}_\omega (Z^{\omega}_j=n) = t\alpha_1\lambda_n\text{, $\mathbb{P}$-a.s.}
\end{equation*} 

2)
\begin{equation*}
        \lim\limits_{L \to \infty} \varliminf_{m \to \infty} \hspace{-6.5mm}\overline{\phantom{\varliminf}}  \sum_{j=0}^{N'-1} \mathbb{\mu}_\omega (Z^{\omega}_j \ge 1) = t\alpha_1\text{, $\mathbb{P}$-a.s.}
\end{equation*} 

3)\begin{equation*}
    \lim_{m \to \infty} \sum_{j=0}^{N-1} \mu_{\theta^j \omega}(\Gamma_{\rho_m} (\theta^j\omega)) = t \text{, $\mathbb{P}$-a.s.},
\end{equation*}
where $Z^{\omega}_j = \sum_{l=jL}^{(j+1)L-1} \mathbbm{1}_{\Gamma_{\rho_m}(\theta^l \omega)} \circ T_\omega^l$. 
\end{lem}
 
\begin{proof}
Let $t, n$ and $ (\rho_m)_{m \ge 1}$ be as in the statement. Consider $L \ge 1$ and $m$ large enough so that $\rho_m \le \rho_{\operatorname{var}}(L)$, $N \ge 3$ and $N' \ge 3$. Denote also $\mathfrak{W}_\rho(\omega)=
\sum_{j=0}^{N'-1} \mu_\omega (Z^{\omega}_j=n)$.

Using Chebycheff's inequality combined with lemma \ref{lem:variance}, we get that
\begin{equation*}
     \mathbb{P}\!\left(\left| \mathfrak{W}_\rho -\mathbb{E}_\mathbb{P}(\mathfrak{W}_\rho)\right|>a\right)
 \le \frac{\operatorname{var}_\mathbb{P}(\mathfrak{W}_\rho)}{a^2}
 \le \frac{C_{t,L}}{a^2}\rho^q,
\end{equation*}and therefore, since $\sum_{m \ge 1} {\rho_m}^q < \infty$, Borel-Cantelli lemma let us conclude that
\begin{equation*}
    \lim_{m \to \infty} \left| \mathfrak{W}_{\rho_m} -\mathbb{E}_\mathbb{P}(\mathfrak{W}_{\rho_m})\right| = 0\text{, $\mathbb{P}$-a.s. }
\end{equation*}

On the other hand, 
\begin{equation*}
    \mathbb{E}_\mathbb{P}(\mathfrak{W}_{\rho_m}) = \frac{1}{L} \frac{t}{ \hat\mu(\Gamma_{\rho_m}) }   \hat\mu(Z_0 = n)   = t \frac{\hat\mu(Z^{L}_{\Gamma_{\rho_m}}  \ge 1)}{L \hat\mu(\Gamma_{\rho_m})} \frac{\hat\mu(Z^{L}_{\Gamma_{\rho_m}} = n)}{\hat\mu(Z^{L}_{\Gamma_{\rho_m}}  \ge 1)} ,
\end{equation*}so, by lemma \ref{lem:annealedentry} and the definition of $\lambda_n$, we have that
\begin{equation*}
    \lim\limits_{L \to \infty} \varliminf_{m \to \infty} \hspace{-6.5mm}\overline{\phantom{\varliminf}}  \hspace{2mm}\mathbb{E}_\mathbb{P}(\mathfrak{W}_{\rho_m}) = t \alpha_1 \lambda_n
\end{equation*}and therefore, combining the previous two centered limits, conclusion (1) follows:
\begin{equation*}
    \lim\limits_{L \to \infty} \varliminf_{m \to \infty} \hspace{-6.5mm}\overline{\phantom{\varliminf}} \hspace{2mm}  \mathfrak{W}_{\rho_m} = t \alpha_1 \lambda_n\text{, $\mathbb{P}$-a.s.}
\end{equation*}

For (2), it suffices to repeat the argument noticing that the new expectation will be driven by $t \frac{\hat\mu(Z^{L}_{\Gamma_{\rho_m}}  \ge 1)}{L \hat\mu(\Gamma_{\rho_m})}$, whose double limit is $t \alpha_1$.

For (3), it suffices to fix $L=1$ and $n=1$ in the above argument, and after the Borel-Cantelli step, notice that
\begin{equation*}
    \mathbb{E}_\mathbb{P}(\mathfrak{W}_{\rho_m}) = t \frac{\hat\mu(\Gamma_{\rho_m})}{\hat\mu(\Gamma_{\rho_m})}  \stackrel{m \to \infty}{\rightarrow} t.
\end{equation*}
 \end{proof}

\section{Proof of theorem \ref{thm:main}}\label{sec:proofmainthm}

\subsection{Applying the abstract approximation theorem} \label{set_up_T1}

Let $t >0$, $n \ge 1$ ($n=0$ is the leftover case) and $\omega \in \Omega$ be any.

Fix, once and for all, $(\rho_m)_{m \ge 1} \searrow 0$ fast enough so that $\sum_{m \ge 1} {\rho_m}^q < \infty$, for some $0< q < q(d_0,d_1,\eta,\beta,\mathfrak{p})$. 

Consider $L \ge n$, which should not be chosen as a function of previous variables.

Define $N := \big\lfloor \frac{t}{\hat\mu(\Gamma_{\rho_m})} \big\rfloor$ and $N'_{m,L} := \frac{N_{m}}{L}  \in \mathbb{N}_{\ge 3}$\footnote{See footnote \ref{note:fractional}.}. Let $v \in (0, d_0)$ and set $\Delta := {\rho_m}^{-v}$. We will consider $m$ large enough (depending on $L$) so that $N \ge 3$, $\Delta \ge 2$, $\rho \le \rho_{\operatorname{var}}(L)$, $L \le \lfloor \frac{N}{3} \rfloor$ and $\Delta < N'$. 

We want to study
\begin{equation*}
    \mu_\omega(Z^{\omega,N}_{\Gamma_{\rho_m} } = n) = \mu_\omega(\textstyle\sum\nolimits_{i=0}^{N-1} I^{\omega,m}_i =n )= \mu_\omega \left(  \sum\limits_{j=0}^{N'-1} \sum\limits_{i=jL}^{(j+1)L-1} I^{\omega,m}_i \right),
\end{equation*}
where $I^{\omega,m}_i = \mathbbm{1}_{\Gamma_{\rho_m}(\theta^i \omega)} \circ T^i_\omega$.

Theorem \ref{thm:approx} can be readily applied and gives
\begin{equation*}
    \left| \mu_\omega\left(Z^{\omega,N_m}_{\Gamma_{\rho_m}} = n\right) - \mu_\omega\left( \sum_{j=0}^{N'-1} \tilde{Z}^{\omega}_j = n \right) \right|
\end{equation*}
\begin{equation*}
    \lesssim \tilde{\mathcal{R}}^1_{\omega,m}(N, L,\Delta) + \mathcal{R}^1_{\omega,m}(N, L,\Delta) + \mathcal{R}^2_{\omega,m}(N,L,\Delta) + \mathcal{R}^3_{\omega,m}(N,L,\Delta) ,
\end{equation*}where $\tilde{Z}^{\omega}_j$ mimics $Z^{\omega}_j = \sum_{l=jL}^{(j+1)L-1} I_l^{\omega,m}$.

For the next sections, sections \ref{sec:R1estimate} to \ref{sec:leading}, it is enough to consider $\omega$ restricted to a $\mathbb{P}$-full measure set.

\subsection{Estimating the error \texorpdfstring{$\mathcal{R}^1$}{R1}} \label{sec:R1estimate}

Recall that
\begin{equation*}
    \mathcal{R}^1_{\omega,m}(N, L,\Delta)=
\end{equation*}
\begin{equation*}
\resizebox{1\hsize}{!}{
\begin{tabular}{c}
$\displaystyle \!\sum_{j=0}^{N'-1} \max_{q \in [1,n]}  \sum_{u=1}^{q} \left| \mu_{\omega}\!\left(\!Z^{\omega}_j {=}{u}, \sum_{k=j+\Delta_m}^{N'-1} Z^{\omega}_k  {=} {q}{-}{u} \right) {-} \mu_\omega\Big(\! Z^{\omega}_j {=}{u} \! \Big)  \mu_\omega\!\left( \sum_{k=j+\Delta_m}^{N'-1} Z^{\omega}_k {=} {q}{-}{u} \right) \right|.$
\end{tabular}}
\end{equation*}

Now recycle the construction and notation used in the proof of lemma \ref{lem:variance} to control the term $(II)$: for a given $j \in [0,N'-1]$, writing $\omega' = \theta^{jL} \omega$ and considering $r \in (0,  {\rho_m}/2)$, $v \in [0,L-1]$, we once again have the objects: $U_{v,\omega'}, \stackrel{-}{U}_{v,r,\omega'}, \stackrel{+}{U}_{v,r,\omega'}, \mathcal{U}_{\omega'}, \stackrel{-}{\mathcal{U}}_{r,\omega'}, \stackrel{+}{\mathcal{U}}_{r,\omega'}, \stackrel{-}{\phi} \hspace{-3mm}\phantom{\phi}_{r}^{\omega'}$ and $ \stackrel{+}{\phi} \hspace{-3mm}\phantom{\phi}_{r}^{\omega'}$. Then:
\begin{equation*}
    \left| \mu_{\omega}\!\left(\!Z^{\omega}_j {=}{u}, \sum_{k=j+\Delta}^{N'-1} Z^{\omega}_k  {=} {q}{-}{u} \right) {-} \mu_\omega\Big(\! Z^{\omega}_j {=}{u} \! \Big)  \mu_\omega\!\left( \sum_{k=j+\Delta}^{N'-1} Z^{\omega}_k {=} {q}{-}{u} \right) \right|
\end{equation*}
\begin{equation*}
     \hspace{-62mm} = \Bigg| \mu_\omega\left( \sum_{i=jL}^{(j+1)L-1} I_i^{\omega,m} = u, \sum_{i=(j+\Delta)L}^{N-1} I_i^{\omega,m} = q-u\right)
\end{equation*}
\begin{equation*}
    \hspace{45mm} - \mu_\omega\left( \sum_{i=jL}^{(j+1)L-1} I_i^{\omega,m} = u \right) \mu_\omega\left( \sum_{i=(j+\Delta)L}^{N-1} I_i^{\omega,m} = q-u\right)\Bigg|
\end{equation*}
\begin{equation*}
    \hspace{-70mm} = \Bigg| \mu_{\omega'}\left( \sum_{i=0}^{L-1} I_i^{\omega',m} = u, \sum_{i=\Delta L}^{(N -1) - jL} I_i^{\omega',m} = q-u\right)
\end{equation*}
\begin{equation*}
    \hspace{30mm} - \mu_{\omega'}\left( \sum_{i=0}^{L-1} I_i^{\omega',m} = u\right) \mu_{\theta^{\Delta L} \omega'}\left(\sum_{i=0}^{(N-1) - (j+ \Delta)L} I_i^{\theta^{\Delta L}\omega',m} = q-u\right) \Bigg|.
\end{equation*} 
\begin{equation*}
    = \left| \mu_{\omega'}\left( \mathbbm{1}_{\mathcal{U}_{\omega'}} \mathbbm{1}_{\{V_j^{\omega,\Delta}=q-u\}} \circ T_{\omega'}^{\Delta L}\right) - \mu_{\omega'}\left( \mathbbm{1}_{\mathcal{U}_{\omega'}}  \right) \mu_{\theta^{\Delta L} \omega'}\left( \mathbbm{1}_{\{V_j^{\omega,\Delta}=q-u\}} \right) \right|
\end{equation*}where we used that $V_j^{\omega,\Delta} := \sum_{i=0}^{(N-1) - (j+ \Delta)L} I_i^{\theta^{\Delta L}\omega',m}$, and thus $\sum_{i=\Delta L}^{(N-1) - jL} I_i^{\omega',m} = V_j^{\omega,\Delta} \circ T_{\omega'}^{\Delta L}$,
\begin{equation*}
    \le \left| \mu_{\omega'}\left( \stackrel{\pm}{\phi} \hspace{-3mm}\phantom{\phi}_{r}^{\omega'} \mathbbm{1}_{\{V_j^{\omega,\Delta}=q-u\}} \circ T_{\omega'}^{\Delta L}\right) - \mu_{\omega'}\left( \mathbbm{1}_{\mathcal{U}_{\omega'}}  \right) \mu_{\theta^{\Delta L} \omega'}\left( \mathbbm{1}_{\{V_j^{\omega,\Delta}=q-u\}} \right) \right|,
\end{equation*}
where $\stackrel{\pm}{\phi} \hspace{-3mm}\phantom{\phi}_{r}^{\omega'}$ means that either $\stackrel{+}{\phi} \hspace{-3mm}\phantom{\phi}_{r}^{\omega'}$ or $\stackrel{-}{\phi} \hspace{-3mm}\phantom{\phi}_{r}^{\omega'}$ will make the inequality true,
\begin{equation*}
    \begin{tabular}{c l}
    $\le$ & $\left| \mu_{\omega'} \Big( \stackrel{\pm}{\phi} \hspace{-3mm}\phantom{\phi}_{r}^{\omega'} \mathbbm{1}_{\{V_j^{\omega,\Delta}=q-u\}} \circ T_{\omega'}^{\Delta L}\Big) -  \mu_{\omega'} \Big( \hspace{-1mm}\stackrel{\pm}{\phi} \hspace{-3mm}\phantom{\phi}_{r}^{\omega'}\Big) \mu_{\theta^{\Delta L} \omega' } \Big( \mathbbm{1}_{\{V_j^{\omega,\Delta}=q-u\}} \Big) \right|$\\
    $+$ & $\left| \left[ \mu_{\omega'}\Big( \hspace{-1mm}\stackrel{\pm}{\phi} \hspace{-3mm}\phantom{\phi}_{r}^{\omega'}\Big) - \mu_{\omega'}\left( \mathbbm{1}_{\mathcal{U}_{\omega'}}  \right) \right] \mu_{\theta^{\Delta L} \omega'}\left( \mathbbm{1}_{\{V_j^{\omega,\Delta}=q-u\}} \right) \right|$\\
    $=:$ & $(A) + (B)$.
    \end{tabular}
\end{equation*}

Now notice that
\begin{equation*}
    (A) \lesssim (\Delta L)^{-\mathfrak{p}}  \big\| \hspace{-1mm}\stackrel{\pm}{\phi} \hspace{-3mm}\phantom{\phi}_{r}^{\omega'}\big\|_{\operatorname{Lip}_{d_M}} 1 \lesssim (\Delta L)^{-\mathfrak{p}}  {a_{L-1}}^2/r^2,
\end{equation*}where the first estimate used (H\ref{hyp:quenchdec}) while the later used (H\ref{hyp:weakhypplaincyl}), (H\ref{hyp:quenchedsep}) and (H\ref{hyp:bddderiv}), as in the quenched argument in the proof of lemma \ref{lem:variance}\footnote{In the present passage, the a.s. validity of $\| \hspace{-1mm} \stackrel{+}{\phi} \hspace{-3mm}\phantom{\phi}_{r}^{\omega'}  \|_{\operatorname{Lip}_{d_M}} \le {a_{L-1}}^2/r^2$ would be enough, but, after recalling the argument of lemma \ref{lem:variance} we see that it actually holds for every $\omega$. The validity for every $\omega $ was important back then, but not here.}. 

Moreover,
\begin{equation*}
    (B) \le \mu_{\theta^{\Delta L} \omega'}\left( V_j^{\omega,\Delta}{=}{q}{-}{u} \right) \mu_{\omega'}(\stackrel{+}{\mathcal{U}}_{r,\omega'} \setminus \stackrel{-}{\mathcal{U}}_{r,\omega'} )  \stackrel{\text{(H\ref{hyp:annulus})}}{\lesssim} \mu_{\theta^{\Delta L} \omega'}\left( V_j^{\omega,\Delta}{=}{q}{-}{u} \right) L \frac{r^\eta}{{\rho_m}^\beta}.
\end{equation*}

Therefore
\begin{equation*}
    \mathcal{R}^1_{\omega,m}(N, L,\Delta) \lesssim  \!\sum_{j=0}^{N'-1} \max_{q \in [1,n]}  \sum_{u=1}^{q}  \left[ (\Delta L)^{-\mathfrak{p}} \frac{{a_{L-1}}^2}{r^2} + \mu_{\theta^{\Delta L} \omega'}\left( V_j^{\omega,\Delta}{=}{q}{-}{u} \right) L \frac{r^\eta}{{\rho_m}^\beta} \right]    
\end{equation*}
\begin{equation*}
    \le  \!\sum_{j=0}^{N'-1} \sum_{u=1}^{L} (\Delta L)^{-\mathfrak{p}}  \frac{{a_{L-1}}^2}{r^2} + L \frac{r^\eta}{{\rho_m}^\beta} \!\sum_{j=0}^{N'_{m,L}-1} \max_{q \in [1,N]}  \mu_{\theta^{\Delta L} \omega'}\left( V_j^{\omega,\Delta} \in [0,q] \right)
\end{equation*}
\begin{equation*}
    \lesssim N (\Delta L)^{-\mathfrak{p}}  \frac{{a_{L-1}}^2}{r^2} + N \frac{r^\eta}{{\rho_m}^\beta} \le N (\Delta L)^{-\mathfrak{p}}  \frac{{a_{L-1}}^2}{r^2} + L N \frac{r^\eta}{{\rho_m}^\beta},
\end{equation*}
where $V_j^{\omega,\Delta}$ takes values between $0$ and $N -(j+\Delta)L \le  N$.

\subsection{Estimating the error \texorpdfstring{$\tilde{\mathcal{R}}^1$}{R1tilde}}\label{sec:R0estimate}

This section is going to follow the lines of the previous one, with minor modifications. 

Recall that
\begin{equation*}
    \tilde{\mathcal{R}}^1_{\omega,m}(N, L,\Delta)=
\end{equation*}
\begin{equation*}
    \!\sum_{j=0}^{N'-1} \max_{q \in [0,n]}   \left| \mu_{\omega}\!\left(\!Z^{\omega}_j {\ge}{1}, \sum_{k=j+\Delta}^{N'-1} Z^{\omega}_k  {=} {q} \right) {-} \mu_\omega\Big(\! Z^{\omega}_j {\ge}{1} \! \Big)  \mu_\omega\!\left( \sum_{k=j+\Delta}^{N'-1} Z^{\omega}_k {=} {q} \right) \right|.
\end{equation*}

For a given $j \in [0,N'-1]$, writing $\omega' = \theta^{jL}$ and considering $r \in (\rho_m/2)$, $v \in [0,L]$, recalling the objects introduced in the proof of lemma \ref{lem:variance}, we reuse $U_{v,\omega'}, \stackrel{-}{U}_{v,r,\omega'}$ and $ \stackrel{+}{U}_{v,r,\omega'}$, whereas $ \mathcal{U}_{\omega'}, \stackrel{-}{\mathcal{U}}_{r,\omega'}$ and $\stackrel{+}{\mathcal{U}}_{r,\omega'}$ are modified by including a union $\bigcup_{n=1}^L$ before the original definitions therein (in particular, $\{Z_0^{\omega'} \ge 1 \} = \mathcal{U}_{\omega'}$), while $ \stackrel{-}{\phi} \hspace{-3mm}\phantom{\phi}_{r}^{\omega'}$ and $ \stackrel{+}{\phi} \hspace{-3mm}\phantom{\phi}_{r}^{\omega'}$ are kept the same (but considering the previous modification).

Following the same steps and notation from the previous section, we get that
\begin{equation*}
\begin{tabular}{c}
     $\displaystyle \left| \mu_{\omega}\!\left(\!Z^{\omega}_j {\ge}{1}, \sum_{k=j+\Delta}^{N'-1} Z^{\omega}_k  {=} {q} \right) {-} \mu_\omega\Big(\! Z^{\omega}_j {\ge}{1} \! \Big)  \mu_\omega\!\left( \sum_{k=j+\Delta}^{N'-1} Z^{\omega}_k {=} {q} \right) \right|$ \\
      $\displaystyle \le \left| \mu_{\omega'} \Big( \stackrel{\pm}{\phi} \hspace{-3mm}\phantom{\phi}_{r}^{\omega'} \mathbbm{1}_{\{V_j^{\omega,\Delta}=q\}} \circ T_{\omega'}^{\Delta L}\Big) -  \mu_{\omega'} \Big( \hspace{-1mm}\stackrel{\pm}{\phi} \hspace{-3mm}\phantom{\phi}_{r}^{\omega'}\Big) \mu_{\theta^{\Delta L} \omega' } \Big( \mathbbm{1}_{\{V_j^{\omega,\Delta}=q\}} \Big) \right|$\\
    $\displaystyle + \left| \left[ \mu_{\omega'}\Big( \hspace{-1mm}\stackrel{\pm}{\phi} \hspace{-3mm}\phantom{\phi}_{r}^{\omega'}\Big) - \mu_{\omega'}\left( \mathbbm{1}_{\mathcal{U}_{\omega'}}  \right) \right] \mu_{\theta^{\Delta L} \omega'}\left( \mathbbm{1}_{\{V_j^{\omega,\Delta}=q\}} \right) \right|$\\
    $=: (A) + (B)$.
\end{tabular}
\end{equation*}

As before,
\begin{equation*}
    (A) \lesssim (\Delta L)^{-\mathfrak{p}}  \big\| \hspace{-1mm}\stackrel{\pm}{\phi} \hspace{-3mm}\phantom{\phi}_{r}^{\omega'}\big\|_{\operatorname{Lip}_{d_M}} 1 \lesssim  (\Delta L)^{-\mathfrak{p}}  {a_{L-1}}^2/r^2.
\end{equation*}For the first inequality we use (H\ref{hyp:quenchdec}). For the second, we adapt the previous reasoning as follows. Intuitively, the Lipschitz constant of, say, the modified function $ \stackrel{-}{\phi} \hspace{-3mm}\phantom{\phi}_{r}^{\omega'}$ is bounded by the inverse of $d\big( \hspace{-1mm}\stackrel{-}{\mathcal{U}} \hspace{-4.5mm}{\phantom{\mathcal{U}}_{r,\omega'}}, \mathcal{U}_{\omega'}^c \big)$. For a point to $x \in  \mathcal{U}_{\omega'}^c$, with no hits, to be minimally displaced to $ \stackrel{-}{\mathcal{U}} \hspace{-4.5mm}{\phantom{\mathcal{U}}_{r,\omega'}}$, among $x$ itself being displaced or the consequently-displaced points in its orbit, a) at least one $r$-stringent hit has to be created while b) the other instances should turned into $r$-stringent non-hits (if they are not already). The situation where this would occur with minimal displacement is one where (b) starts already fulfilled and only (a) has to be accomplished by displacing $x$ in such that its $L-1$ iterate changes from a non-hit to a $r$-stringent hit. This can be made with a minimum displacement of $r/a_{L-1}$, where again we use (H\ref{hyp:weakhypplaincyl}), (H\ref{hyp:bddderiv}) and (H\ref{hyp:quenchedsep}).

Moreover,
\begin{equation*}
    (B) \le \mu_{\theta^{\Delta L} \omega'}\left( V_j^{\omega,\Delta}=q \right) \mu_{\omega'}(\stackrel{+}{\mathcal{U}}_{r,\omega'} \setminus \stackrel{-}{\mathcal{U}}_{r,\omega'} )  \stackrel{\text{(H\ref{hyp:annulus})}}{\lesssim} \mu_{\theta^{\Delta L} \omega'}\left( V_j^{\omega,\Delta}=q \right) L^2 \frac{r^\eta}{{\rho_m}^\beta}.
\end{equation*}

Therefore
\begin{equation*}
    \tilde{\mathcal{R}}^1_{\omega,m}(N, L,\Delta) \lesssim  \!\sum_{j=0}^{N'-1} \max_{q \in [0,n]}    \left[ (\Delta L)^{-\mathfrak{p}} \frac{{a_{L-1}}^2}{r^2}  + \mu_{\theta^{\Delta L} \omega'}\left( V_j^{\omega,\Delta}=q \right) L^2 \frac{r^\eta}{{\rho_m}^\beta} \right]    
\end{equation*}
\begin{equation*}
    \le  \!\sum_{j=0}^{N'-1} (\Delta L)^{-\mathfrak{p}}  \frac{{a_{L-1}}^2}{r^2} + L^2 \frac{r^\eta}{{\rho_m}^\beta} \!\sum_{j=0}^{N'-1} \max_{q \in [0,N]}  \mu_{\theta^{\Delta L} \omega'}\left( V_j^{\omega,\Delta} \in [0,q] \right)
\end{equation*}
\begin{equation*}
    \lesssim N' (\Delta L)^{-\mathfrak{p}}  \frac{{a_{L-1}}^2}{r^2} + L N \frac{r^\eta}{{\rho_m}^\beta} \le N (\Delta L)^{-\mathfrak{p}}  \frac{{a_{L-1}}^2}{r^2} + L N \frac{r^\eta}{{\rho_m}^\beta},
\end{equation*}
where $V_j^{\omega,\Delta}$ takes values between $0$ and $N -(j+1)L  \le N$.

\subsection{Estimating the error \texorpdfstring{$\mathcal{R}^2$}{R2}}\label{sec:R2estimate}

To start
\begin{equation*}
    \mathcal{R}^2_{\omega,m}(N,L,\Delta) = \sum_{j=0}^{N'-1} \mu_\omega(Z^{\omega}_j \ge 1, \sum_{k=j+1}^{j+\Delta-1} Z^{\omega}_k \ge 1)
    \le \sum_{j=0}^{N'-1} \sum_{k=j+1}^{j+\Delta-1}\mu_\omega(Z^{\omega}_j \ge 1,  Z^{\omega}_k \ge 1) 
\end{equation*}
where we reverse the double sum and single out the $k=j+1$ terms, to get
\begin{equation*}
 \sum_{k=1}^{N'+\Delta-2} \sum_{j=(k-\Delta+1)\vee 0}^{(k-2)\wedge(N'-1)} \mu_\omega(Z^{\omega}_j \ge 1,  Z^{\omega}_k \ge 1) + \sum_{k=1}^{N'} \mu_\omega(Z^{\omega}_{k-1} \ge 1,  Z^{\omega}_k \ge 1) =: \text{ }(I) + (II)
\end{equation*}

To estimate $(I)$ we notice that:

\begin{equation*}
\resizebox{1\hsize}{!}{
\begin{tabular}{@{}c @{} c @{} l @{}}
   $(I)$ & $\le$ & $\displaystyle \sum_{k=1}^{N'+\Delta-2} \sum_{j=(k-\Delta+1)\vee 0}^{(k-2)\wedge(N'-1)} \sum_{i=jL}^{(j+1)L-1} \sum_{l=kL}^{(k+1)L-1} \mu_\omega\left( (T_\omega^i)^{-1} \Gamma_{\rho_m}(\theta^i \omega) \cap (T_\omega^l)^{-1} \Gamma_{\rho_m}(\theta^l \omega) \right) \text{ }(l > i)$\\
    & $\displaystyle \le$ & $\displaystyle \sum_{k=1}^{N'+\Delta-2} \sum_{j=(k-\Delta+1)\vee 0}^{(k-2)\wedge(N'-1)} \sum_{i=jL}^{(j+1)L-1} \sum_{l=kL}^{(k+1)L-1} \mu_{\omega'}\left(  \Gamma_{\rho_m}( \omega') \cap (T_{\omega'}^{l-i})^{-1} \Gamma_{\rho_m}(\theta^{l-i} \omega') \hspace{0.5mm} \cap \hspace{-0.5mm} \stackrel{++}{\mathcal{C}} \hspace{-4mm} \phantom{C}_{l-i}^{\omega'} \right) \text{ }$ \\
    & $\displaystyle +$ & $\displaystyle \sum_{k=1}^{N'+\Delta-2} \sum_{j=(k-\Delta+1)\vee 0}^{(k-2)\wedge(N'-1)} \sum_{i=jL}^{(j+1)L-1} \sum_{l=kL}^{(k+1)L-1} \mu_{\omega'}\left(  \Gamma_{\rho_m}( \omega') \cap (T_{\omega'}^{l-i})^{-1} \Gamma_{\rho_m}(\theta^{l-i} \omega') \hspace{0.5mm} \cap \left[ \hspace{-0.5mm} \stackrel{+-}{\mathcal{C}} \hspace{-3.5mm} \phantom{C}_{l-i}^{\omega'} \cup \hspace{-0.5mm} \stackrel{-}{\mathcal{C}} \hspace{-3mm} \phantom{C}_{l-i}^{\omega'} \right] \right)$\\
    & $=:$ & $ (I_{\operatorname{good}}) + (I_{\operatorname{bad}})$
\end{tabular}}
\end{equation*}
where $\omega' := \theta^i \omega$.

To estimate $(I_{\operatorname{good}})$ we begin evaluating the following:
\begin{equation*}
    \mu_{\omega'}\left( \stackrel{++}{\mathcal{C}} \hspace{-4mm} \phantom{C}_{l-i}^{\omega'} \cap \Gamma_{\rho_m}( \omega') \cap (T_{\omega'}^{l-i})^{-1} \Gamma_{\rho_m}(\theta^{l-i} \omega') \right)
\end{equation*}
\begin{equation*}
    \le  \sum_{\substack{\xi = \varphi(\operatorname{dom}(\varphi)) \in \stackrel{++}{C} \hspace{-2.5mm} \phantom{C}^{\omega'}_{l-i}: \\ \xi \cap \Gamma_{\rho_m}(\omega') \neq \emptyset}} \frac{\mu_{\omega'}|_\xi \left( \xi \cap (T_{\omega'}^{l-i})^{-1} \Gamma_{\rho_m}(\theta^{l-i} \omega') \right)}{\mu_{\omega'}|_\xi(\xi)} \mu_{\omega'}(\xi),
\end{equation*}where, from (H\ref{hyp:dist}), $\varphi \in \operatorname{IB}(T_{\omega'}^{l-i})$ implies $\mu_{\omega'}|_{\varphi(\operatorname{dom}(\varphi))} = {J_\varphi}^{-1} \left[ \varphi_*(\mu_{\theta^{l-i}\omega'}|_{\operatorname{dom}(\varphi)}) \right] $, and so
\begin{equation*}
    \le  \sum_{\xi \text{ as above}}  \frac{ \displaystyle \left[{J_\varphi}^{-1} \big[ \varphi_*(\mu_{\theta^{l-i}\omega'}|_{\operatorname{dom}(\varphi)}) \big] \right] \left( \varphi(\operatorname{dom}(\varphi)) \cap (T_{\omega'}^{l-i})^{-1} \Gamma_{\rho_m}(\theta^{l-i} \omega') \right)}{\displaystyle \left[{J_\varphi}^{-1} \big[ \varphi_*(\mu_{\theta^{l-i}\omega'}|_{\operatorname{dom}(\varphi)}) \big] \right](\varphi(\operatorname{dom}(\varphi)))}  \mu_{\omega'}(\xi)
\end{equation*}
\begin{equation*}
    \le  \sum_{\xi \text{ as above}} \frac{\sup_{x \in \xi} {J_\varphi}^{-1}(x) }{\inf_{x \in \xi} {J_\varphi}^{-1}(x)} 
    \frac{ \displaystyle   \mu_{\theta^{l-i}\omega'}|_{\operatorname{dom}(\varphi)}   \left( \operatorname{dom}(\varphi) \cap {\varphi^{-1}} (T_{\omega'}^{l-i})^{-1} \Gamma_{\rho_m}(\theta^{l-i} \omega') \right)}{\displaystyle 
     \mu_{\theta^{l-i}\omega'}|_{\operatorname{dom}(\varphi)}  (\operatorname{dom}(\varphi))}  \mu_{\omega'}(\xi)
\end{equation*}
\begin{equation*}
\resizebox{1\hsize}{!}{
\begin{tabular}{@{}c @{}}
    $\displaystyle \stackrel[\text{(H\ref{hyp:nondegen})}]{\text{(H\ref{hyp:dist})}}{\lesssim} (l-i)^{\mathfrak{d}} \iota^{-1} \mu_{\theta^l \omega} ( \Gamma_{\rho_m}(\theta^l \omega) ) \sum_{\xi \text{ as above}}   \mu_{\omega'}(\xi) \stackrel{\text{(H\ref{hyp:cov})}}{\le}  (l-i)^{\mathfrak{d}} \iota^{-1} \mu_{\theta^l \omega} ( \Gamma_{\rho_m}(\theta^l \omega) ) \mathcal{N} \mu_{\omega'}\left( \bigcup_{\xi \text{ as above}}  \xi  \right)$
\end{tabular}}
\end{equation*}
\begin{equation*}
    \stackrel{\text{(H\ref{hyp:backcontr})}}{\le} (l-i)^{\mathfrak{d}} \iota^{-1} \mu_{\theta^l \omega} ( \Gamma_{\rho_m}(\theta^l \omega) ) \mathcal{N} \mu_{\omega'}\left( B_{D (l-i)^{-\kappa}}(\Gamma_{\rho_m}(\omega') )   \right)
\end{equation*}
\begin{equation*}
\resizebox{1\hsize}{!}{
\begin{tabular}{@{}c @{}}
    $\displaystyle \stackrel{\text{(H\ref{hyp:ball})}}{\le} (l-i)^{\mathfrak{d}} \iota^{-1} \mu_{\theta^l \omega} ( \Gamma_{\rho_m}(\theta^l \omega) ) \mathcal{N} C_0 (\rho_m + D (l-i)^{-\kappa} )^{d_0} \lesssim  \mu_{\theta^l \omega} ( \Gamma_{\rho_m}(\theta^l \omega) ) (l-i)^{\mathfrak{d}}  \left[ {\rho_m}^{d_0} {+} (l{-}i)^{-\kappa d_0} \right].$
\end{tabular}}
\end{equation*}

Then
\begin{equation*}
\resizebox{1\hsize}{!}{
\begin{tabular}{@{}c @{} c @{} l @{}}
       $(I_{\operatorname{good}})$  &  $\le$ & $\displaystyle \sum_{k=1}^{N'+\Delta-2} \sum_{j=(k-\Delta+1)\vee 0}^{(k-2)\wedge(N'-1)} \sum_{i=jL}^{(j+1)L-1} \sum_{l=kL}^{(k+1)L-1} \mu_{\theta^l \omega} ( \Gamma_{\rho_m}(\theta^l \omega) ) (l-i)^{\mathfrak{d}}  \left[ {\rho_m}^{d_0} {+} (l{-}i)^{-\kappa d_0} \right]$  \\
       & $=$  & $\displaystyle \sum_{k=1}^{N'+\Delta-2} \sum_{j=(k-\Delta+1)\vee 0}^{(k-2)\wedge(N'-1)} \sum_{l=kL}^{(k+1)L-1} \left(  \mu_{\theta^l \omega} ( \Gamma_{\rho_m}(\theta^l \omega) )  \sum_{i=jL}^{(j+1)L-1}  (l-i)^{\mathfrak{d}} \left[ {\rho_m}^{d_0} {+} (l{-}i)^{-\kappa d_0} \right] \right) $
\end{tabular}}
\end{equation*}
where, for each $l$ fixed, as $i$ runs, we have $l-i \in [kL -jL -L +1,kL-jL+L-1]$, so
\begin{equation*}
    \le \sum_{k=1}^{N'+\Delta-2} \sum_{j=(k-\Delta+1)\vee 0}^{(k-2)\wedge(N'-1)} \sum_{l=kL}^{(k+1)L-1} \left(  \mu_{\theta^l \omega} ( \Gamma_{\rho_m}(\theta^l \omega) )  \sum_{s=kL-jL-L+1}^{kL-jL+L-1}  s^{\mathfrak{d}} \left[ {\rho_m}^{d_0} {+} s^{-\kappa d_0} \right] \right)
\end{equation*}
\begin{equation*}
    = \sum_{k=1}^{N'+\Delta-2} \sum_{j=(k-\Delta+1)\vee 0}^{(k-2)\wedge(N'-1)} \left[ \left( \sum_{l=kL}^{(k+1)L-1}  \mu_{\theta^l \omega} ( \Gamma_{\rho_m}(\theta^l \omega) )   \right) \left( \sum_{s=kL-jL-L+1}^{kL-jL+L-1} s^{\mathfrak{d}} \left[ {\rho_m}^{d_0} {+} s^{-\kappa d_0} \right] \right) \right]
\end{equation*}
\begin{equation*}
    \hspace{8mm}\le \sum_{k=1}^{N'+\Delta-2} \left( \sum_{l=kL}^{(k+1)L-1}  \mu_{\theta^l \omega} ( \Gamma_{\rho_m}(\theta^l \omega) ) \right) \left(\sum_{j=(k-\Delta+1)\vee 0}^{(k-2)\wedge(N'-1)}  \sum_{s=kL-jL-L+1}^{kL-jL+L-1}  s^{\mathfrak{d}} \left[ {\rho_m}^{d_0} {+} s^{-\kappa d_0} \right]\right)
\end{equation*}
where $s \in [L+1,3 \Delta L]$\footnote{The interval where $s$ ranges has length $2L$ and it is translated by $L$ when $j$ moves one unit, therefore the original and the new interval overlap by half, so eventual repetitions are more than compensated by a factor of two.}, so
\begin{equation*}
    \lesssim \sum_{k=1}^{N'+\Delta-2} \left( \sum_{l=kL}^{(k+1)L-1}  \mu_{\theta^l \omega} ( \Gamma_{\rho_m}(\theta^l \omega) ) \right) \left(\sum_{u=L+1}^{3\Delta L}   u^{\mathfrak{d}} \left[ u^{-\kappa d_0} {+} {\rho_m}^{d_0} \right]\right)
\end{equation*}
\begin{equation*}
    \lesssim \sum_{k=1}^{N'+\Delta-2} \left( \sum_{l=kL}^{(k+1)L-1}  \mu_{\theta^l \omega} ( \Gamma_{\rho_m}(\theta^l \omega) ) \right) \left(L^{\mathfrak{d}-\kappa d_0 +1} + ( \Delta L)^{\mathfrak{d}+1} {\rho_m}^{d_0} \right),
\end{equation*}where for the first term in the square bracket we have used that, for $\zeta>1$, $\sum_{n=m}^\infty n^{-\zeta} \lesssim m^{-\zeta +1}$ together with $\mathfrak{d} - \kappa d_0 < -1$, which is guaranteed by (H\ref{hyp:parampolynomialseries}), whereas for the second we have used that $u^{\mathfrak{d}}$ is increasing and the summation interval is bounded above by $3\Delta L$.

We will leave $(I_{\operatorname{bad}})$ to the end.

For $(II)$, we consider $L' <L $ and proceed as follows
\begin{equation*}
    \sum_{k=1}^{N'} \mu_\omega(Z^{\omega}_{k-1} \ge 1,  Z^{\omega}_k \ge 1) = \sum_{k=1}^{N'} \mu_\omega\left(\sum_{i=(k-1)L}^{kL-1} I_i^{\omega,m} \ge 1, \sum_{l=kL}^{kL+L'-1} I_l^{\omega,m} + \sum_{l=kL+L'}^{(k+1)L-1} I_l^{\omega,m} \ge 1\right)
\end{equation*}
\begin{equation*}
    \le \sum_{k=1}^{N'} \mu_\omega\left(\sum_{i=(k-1)L}^{kL-1} I_i^{\omega,m} \ge 1, \sum_{l=kL}^{kL+L'-1} I_l^{\omega,m} \ge 1\right) +  \mu_\omega\left(\sum_{i=(k-1)L}^{kL-1} I_i^{\omega,m} \ge 1,  \sum_{l=kL+L'}^{(k+1)L-1} I_l^{\omega,m} \ge 1\right)
\end{equation*}and, denoting $\omega' = \theta^i \omega$,
\begin{equation*}
    \begin{tabular}{c @{} l}
       $\le$  & $\displaystyle \sum_{k=1}^{N'} \sum_{l=kL}^{kL+L'-1} \mu_{\theta^l \omega}\left( \Gamma_{\rho_m}(\theta^l \omega)\right)$ \\
        $+$ & $\displaystyle \sum_{k=1}^{N'+\Delta-2} \sum_{i=(k-1)L}^{kL-1} \sum_{l=kL+L'}^{(k+1)L-1}\mu_{\omega'}\left(\Gamma_{\rho_m}(\omega') \cap (T_{\omega'}^{l-i})^{-1} \Gamma_{\rho_m}(\theta^{l-i}\omega') \hspace{0.5mm} \cap \hspace{-0.5mm} \stackrel{++}{\mathcal{C}} \hspace{-4mm} \phantom{C}_{l-i}^{\omega'} \right)$\\
        $+$ & $\displaystyle \sum_{k=1}^{N'} \sum_{i=(k-1)L}^{kL-1} \sum_{l=kL+L'}^{(k+1)L-1}\mu_{\omega'}\left( \Gamma_{\rho_m}(\omega') \cap (T_{\omega'}^{l-i})^{-1} \Gamma_{\rho_m}(\theta^{l-i}\omega') \hspace{0.5mm} \cap \left[ \hspace{-0.5mm} \stackrel{+-}{\mathcal{C}} \hspace{-4mm} \phantom{C}_{l-i}^{\omega'} \cup \stackrel{-}{\mathcal{C}} \hspace{-3mm} \phantom{C}_{l-i}^{\omega'}  \right] \right)$\\
        $=:$ & $(II_{\operatorname{rest}}) + (II_{\operatorname{good}}) + (II_{\operatorname{bad}})$.
    \end{tabular}
\end{equation*}

The term $(II_{\operatorname{rest}})$ will not be improved, whereas the term $(II_{\operatorname{good}})$ is approached just like $(I_{\operatorname{good}})$, as follows:
\begin{equation*}
    (II_{\operatorname{good}}) \lesssim \sum_{k=1}^{N'} \sum_{l=kL+L'}^{(k+1)L-1} \left( \mu_{\theta^l \omega} (\Gamma_{\rho_m}(\theta^l \omega)) \sum_{i=(k-1)L}^{kL-1} (l-i)^{\mathfrak{d}} \left[ {\rho_m}^{d_0} + (l-i)^{-\kappa d_0} \right]  \right)
\end{equation*}
where, for each $l$ fixed, as $i$ runs, we have $l-i \in [L' +1,2L-1]$, so
\begin{equation*}
    \le \sum_{k=1}^{N'} \left( \sum_{l=kL+L'}^{(k+1)L-1} \mu_{\theta^l \omega} (\Gamma_{\rho_m}(\theta^l \omega))  \right)  \left(  \sum_{u=L'+1}^{2L-1} u^{\mathfrak{d}} \left[ {\rho_m}^{d_0} + u^{-\kappa d_0} \right]  \right)
\end{equation*}
\begin{equation*}
    \lesssim \sum_{k=1}^{N'} \left( \sum_{l=kL}^{(k+1)L-1} \mu_{\theta^l \omega} (\Gamma_{\rho_m}(\theta^l \omega))  \right) \left( {L'}^{\mathfrak{d}-\kappa d_0+1} +L^{\mathfrak{d}+1} {\rho_m}^{d_0} \right).
\end{equation*}

Now we combine $(I_{\operatorname{bad}})$ and $(II_{\operatorname{bad}})$ and their domain of summation\footnote{Notice that the initial $L'$-strip of the first component of the original summation has already been singled out inside $(II_{\operatorname{rest}})$.} to see that
\begin{equation*}
\begin{tabular}{c c l }
    $ (I_{\operatorname{bad}}) + (II_{\operatorname{bad}})$ & $\lesssim$ &  $\displaystyle \sum_{i=0}^{N-1} \sum_{l=i+L+L'}^{i+\Delta L} \mu_{\theta^i \omega} \left( \Big[ \hspace{-1mm} \stackrel{+-}{\mathcal{C}} \hspace{-4mm} \phantom{C}_{l-i}^{\theta^i \omega} \cup \stackrel{-}{\mathcal{C}} \hspace{-3mm} \phantom{C}_{l-i}^{\theta^i \omega} \Big] \cap \Gamma_{\rho_m}(\theta^i \omega) \right)$ \\
    & $=$ & $\displaystyle \sum_{i=0}^{N-1} \sum_{s=L+L'}^{\Delta L} \mu_{\theta^i \omega} \left(\Big[ \hspace{-1mm} \stackrel{+-}{\mathcal{C}} \hspace{-4mm} \phantom{C}_{s}^{\theta^i \omega} \cup \stackrel{-}{\mathcal{C}} \hspace{-3mm} \phantom{C}_{s}^{\theta^i \omega} \Big]  \cap \Gamma_{\rho_m}(\theta^i \omega) \right)$ \\ 
     & $\le$&  $\displaystyle \sum_{s=L'}^{\Delta L} \sum_{i=0}^{N-1}  \mu_{\theta^i \omega} \left( \Big[ \hspace{-1mm} \stackrel{+-}{\mathcal{C}} \hspace{-4mm} \phantom{C}_{s}^{\theta^i \omega} \cup \stackrel{-}{\mathcal{C}} \hspace{-3mm} \phantom{C}_{s}^{\theta^i \omega} \Big] \cap \Gamma_{\rho_m}(\theta^i \omega) \right)$.
\end{tabular}
\end{equation*}

Combining the bounds of $(I_{\operatorname{good}})$ and $(II_{\operatorname{good}})$, we conclude that
\begin{equation*}
\begin{tabular}{c @{} c @{} l}
    $\mathcal{R}^2_{\omega,m}(N, L, \Delta)$ & $\lesssim$ & $\displaystyle  \sum_{l=0}^{5N-1}  \mu_{\theta^l \omega} ( \Gamma_{\rho_m}(\theta^l \omega) )  \left({L'}^{\mathfrak{d}-\kappa d_0 +1} + ( \Delta L)^{\mathfrak{d}+1} {\rho_m}^{d_0} \right)$\\
     & $+$ & $\displaystyle \sum_{k=1}^{N'} \sum_{l=kL}^{kL+L'-1} \mu_{\theta^l \omega}\left( \Gamma_{\rho_m}(\theta^l \omega)\right) \hspace{-0.5mm}{+}\hspace{-0.5mm} \sum_{s=L'}^{\Delta L} \sum_{i=0}^{N-1}  \mu_{\theta^i \omega} \left(\Big[ \hspace{-1mm} \stackrel{+-}{\mathcal{C}} \hspace{-4mm} \phantom{C}_{s}^{\theta^i \omega} \cup \stackrel{-}{\mathcal{C}} \hspace{-3mm} \phantom{C}_{s}^{\theta^i \omega} \Big] {\cap} \Gamma_{\rho_m}(\theta^i \omega) \right)$.
\end{tabular}
\end{equation*}

\subsection{Estimating the error \texorpdfstring{$\mathcal{R}^3$}{R3}}\label{sec:R3estimate}

Here we use (H\ref{hyp:ball}) to see that
\begin{equation*}
\begin{tabular}{c  c  l}
    $\displaystyle \mathcal{R}_{\omega,m}^3(N,L,\Delta)$ & $=$ & $\displaystyle \sum_{i=0}^{N-1}\sum_{\ell=0\vee (i-{ \Delta}L)}^i 
   \mu_{\theta^i\omega}(\Gamma_{\rho_m}(\theta^i\omega)) \mu_{\theta^\ell\omega}(\Gamma_{\rho_m}(\theta^\ell\omega)))$  \\
     & $\lesssim$ & $\displaystyle { \Delta}L\text{ }{\rho_m}^{d_0}\sum_{i=0}^{N-1} \mu_{\theta^i\omega}(\Gamma_{\rho_m}(\theta^i\omega))$,
\end{tabular} 
\end{equation*}
which, noticing that $\Delta L \le (\Delta L)^{\mathfrak{d}+1}$, reveals to be bounded above by $ \mathcal{R}_{\omega,m}^2(N_m,L,\Delta_m)$.

\subsection{Controlling the total error}\label{sec:totalerror} Put $r={\rho_m}^{w}$ ($w>1$) 
and $L'=L^\alpha$ ($0<\alpha<1$). Then 
\begin{equation*}
   \hspace{-77mm} \left| \mu_\omega\left(Z^{\omega,N}_{\Gamma_{\rho_m}} = n\right) - \mu_\omega\left( \sum_{j=0}^{N'-1} \tilde{Z}^{\omega}_j = n \right) \right| 
\end{equation*}
\begin{equation}\label{eq:totalerrorcontrol1}
  \hspace{25mm}  \begin{tabular}{c @{}  l}
       $\lesssim$  &  $\displaystyle \hspace{2mm}{a_{L-1}}^2  {\rho_m}^{\mathfrak{p}v - 2w -d_1} + L {\rho_m}^{w \eta-\beta-d_1}$ \\
       $+$ &  $\displaystyle  \sum_{l=0}^{5N-1}  \mu_{\theta^l \omega} ( \Gamma_{\rho_m}(\theta^l \omega) )  \left({L'}^{\mathfrak{d}-\kappa d_0 +1} + L^{\mathfrak{d}+1} {\rho_m}^{d_0 - v (\mathfrak{d}+1)} \right)$\\
      $+$ & $\displaystyle \sum_{k=1}^{N'} \sum_{l=kL}^{kL+L'-1} \mu_{\theta^l \omega}\left( \Gamma_{\rho_m}(\theta^l \omega)\right) \hspace{-0.5mm}{+}\hspace{-0.5mm} \sum_{s=L'}^{\Delta_m L} \sum_{i=0}^{N-1}  \mu_{\theta^i \omega}\left(\Big[ \hspace{-1mm} \stackrel{+-}{\mathcal{C}} \hspace{-4mm} \phantom{C}_{s}^{\theta^i \omega} \cup \stackrel{-}{\mathcal{C}} \hspace{-3mm} \phantom{C}_{s}^{\theta^i \omega} \Big] \cap \Gamma_{\rho_m}(\theta^i \omega)\right)$,
    \end{tabular}
\end{equation}where in the first line of the RHS accounts for both $\mathcal{R}^1$ and $\tilde{\mathcal{R}}^1$.

Now we fine-tune parameters $v \in (0,d_0)$ ($\Delta = {\rho_m}^{-v}$) and $w>1$ ($r= {\rho_m}^w$). In the last equation, we need the exponents accompanying $\rho$ to be strictly positive. In particular, we need
\begin{equation*}
    w > \frac{\beta+d_1}{\eta} \vee 1 \text{, } p v - 2w - d_1 > 0 \text{ and } d_0-v(\mathfrak{d}+1)>0.
\end{equation*}The space of solutions $(w,v) \in (1,\infty) \times (0,d_0)$ to those inequalities is non-empty triangle if $\mathfrak{p} > \frac{2\left( \frac{ \beta + d_1}{\eta} \vee 1 \right) + d_1}{d_0/(\mathfrak{d}+1)},$ which is guaranteed by (H\ref{hyp:parampositrhoexp}). Let's fix any such solution $(w,v)$.

We will take double limits of the type $\varlimsup_{L \to \infty} \varlimsup_{m \to \infty}$ on the RHS equation (\ref{eq:totalerrorcontrol1}). To fist take $\varlimsup_{m \to \infty}$, we use that, by lemma \ref{lem:asconv},
\begin{equation*}
    \lim_{m \to \infty} \sum_{l=0}^{5N-1}  \mu_{\theta^l \omega} ( \Gamma_{\rho_m}(\theta^l \omega) ) =5t \text{, $\mathbb{P}$-a.s.}
\end{equation*}and, by similar arguments\footnote{Adapting the argument of lemma \ref{lem:asconv} item (III) to the new term, we see that the new $\mathbb{P}$-expectation is $t L^{\alpha-1}$, but the variance lemma used therein, lemma \ref{lem:variance}, would need to be adapted as well, what we omitted.},
\begin{equation*}
     \lim_{m \to \infty }\sum_{k=1}^{N'} \sum_{l=kL}^{kL+L'-1} \mu_{\theta^l \omega}\left( \Gamma_{\rho_m}(\theta^l \omega)\right) = t L^{\alpha-1} \text{, $\mathbb{P}$-a.s..}
\end{equation*}

Finally, using hypothesis (H\ref{hyp:quenchedsep2}) and noticing that $\mathfrak{d}- \kappa d_0 +1 < 0$ (by (H\ref{hyp:parampolynomialseries})) and $\alpha-1 < 0$ (by design), we conclude that the RHS of equation (\ref{eq:totalerrorcontrol1}) under the double limit $\varlimsup_{L \to \infty} \varlimsup_{m \to \infty}$ goes to $0$. The same thing occurs if we adopt the double limits $\varliminf_{L \to \infty} \varlimsup_{m \to \infty}$, $\varlimsup_{L \to \infty} \varliminf_{m \to \infty}$ and $\varliminf_{L \to \infty} \varliminf_{m \to \infty}$. Therefore
\begin{equation*}
    \lim\limits_{L \to \infty} \varliminf_{m \to \infty} \hspace{-6.5mm}\overline{\phantom{\varliminf}} \hspace{2mm}
    \left| \mu_\omega\left(Z^{\omega,N}_{\Gamma_{\rho_m}} = n\right) - \mu_\omega\left( \sum_{j=0}^{N'-1} \tilde{Z}^{\omega}_j = n \right) \right| =0 \text{, $\mathbb{P}$-a.s.}. 
\end{equation*}

\subsection{Convergence of the leading term to the compound Poisson distribution}\label{sec:leading}
It remains to show that $\mu_\omega\left( \sum_{j=0}^{N'-1} \tilde{Z}^{\omega}_j = n \right)$ to $\operatorname{CPD}_{t \alpha_1,(\lambda_\ell)_\ell}(n)$.

Due to the independence and distributional properties of the $\tilde{Z}^{\omega}_j$'s (see theorem \ref{thm:approx}):
\begin{equation*}
    \mu_\omega\left( \sum_{j=0}^{N'-1} \tilde{Z}^{\omega}_j {=} n \right) 
    {=} \sum_{l=1}^{n} \sum_{0 \le j_1 < \ldots < j_l \le N'-1} \left( \prod_{\substack{j \in [0,N'-1] \\ \setminus \{j_i:i=1,\ldots,l\}}} \mu_{\omega}(Z^{\omega}_j {=} 0) \cdot \hspace{-2mm}\sum_{\substack{(n_1,\ldots,n_l) \in \mathbb{N}_{\ge 1}^l \\ n_1 + \ldots + n_l = n} } \prod_{i=1}^l \mu_\omega(Z^{\omega}_{j_i} {=} n_i ) \right)
\end{equation*}
\begin{equation*}
     \stackrel{(\star)}{=}(1+o(1)) \prod_{j=0}^{N'-1} \mu_\omega(Z^{\omega}_j = 0) \sum_{l=1}^n \frac{1}{l!} \sum_{\substack{j_i \in [0,N'-1] \\ i=1,\ldots,l}} \sum_{\substack{(n_1,\ldots,n_l) \in \mathbb{N}_{\ge 1}^l \\ n_1 + \ldots + n_l = n} } \prod_{i=1}^l \mu_\omega(Z^{\omega}_{j_i} = n_i )
\end{equation*}
\begin{equation*}
    \stackrel{(\star \star)}{=} (1+o(1)) \prod_{j=0}^{N'-1} \mu_\omega(Z^{\omega}_j = 0) \sum_{l=1}^n \frac{1}{l!} \sum_{\substack{(n_1,\ldots,n_l) \in \mathbb{N}_{\ge 1}^l \\ n_1 + \ldots + n_l = n} } \prod_{i=1}^l \left( \sum_{j=0}^{N'-1} \mu_\omega(Z^{\omega}_j = n_i ) \right),
\end{equation*}where i) $o(1)$ refers to a function $g(\omega,m,L)$ so that $\varlimsup_{L \to \infty} \varlimsup_{m \to \infty} |g(\omega,m,L)| = 0$, 
 $\mathbb{P}$-a.s.; ii) equality $(\star)$ included $1/l!$ to account for $j_i$'s not being anymore increasing and used that the error terms that come from different $j_i$'s being equal are small, as one can see in the case when two $j_i$ agree; and iii) equality $(\star \star)$ uses that a product of sums distributes as a sum of products.

We then notice that, by lemma \ref{lem:asconv},
\begin{equation*}
    \lim\limits_{L \to \infty} \varliminf_{m \to \infty} \hspace{-6.5mm}\overline{\phantom{\varliminf}}  \sum_{j=0}^{N'-1} \mathbb{\mu}^\omega (Z^{\omega}_j=n_i) = t\alpha_1\lambda_{n_i}\text{, $\mathbb{P}$-a.s.}
\end{equation*}
and
\begin{equation*}
    \lim\limits_{L \to \infty} \varliminf_{m \to \infty} \hspace{-8mm}\overline{\phantom{\varliminf}} \prod_{j=0}^{N'-1} \mu_\omega(Z^{\omega}_j = 0) = \lim\limits_{L \to \infty} \varliminf_{m \to \infty} \hspace{-8mm}\overline{\phantom{\varliminf}} \operatorname{exp} \left(  \sum_{j=0}^{N'-1} \ln \left(1-\mu_\omega(Z^{\omega}_j \ge 1)\right) \right)
\end{equation*}
\begin{equation*}
    = \lim\limits_{L \to \infty} \varliminf_{m \to \infty} \hspace{-8mm}\overline{\phantom{\varliminf}} \operatorname{exp} \left(  \sum_{j=0}^{N'-1} -\mu_\omega(Z^{\omega}_j \ge 1) + o(1) \right) = e^{-t \alpha_1}\text{, }\mathbb{P}\text{-a.s.}.
\end{equation*}

Therefore 
\begin{equation*}
    \lim\limits_{L \to \infty} \varliminf_{m \to \infty} \hspace{-8mm}\overline{\phantom{\varliminf}} \hspace{2mm} \left|\mu_\omega\left( \sum_{j=0}^{N'-1} \tilde{Z}^{\omega}_j = n \right) - e^{-t \alpha_1} \sum_{l=1}^n \frac{(t \alpha_1)^l}{l!} \sum_{\substack{(n_1,\ldots,n_l) \in \mathbb{N}_{\ge 1}^l \\ n_1 + \ldots + n_l = n} } \prod_{i=1}^l \lambda_{n_i} \right| = 0\text{, }\mathbb{P}\text{-a.s.}
\end{equation*}
\begin{equation*}
    \Leftrightarrow \lim\limits_{L \to \infty} \varliminf_{m \to \infty} \hspace{-8mm}\overline{\phantom{\varliminf}} \hspace{2mm} \left|\mu_\omega\left( \sum_{j=0}^{N'-1} \tilde{Z}^{\omega}_j = n \right) - \operatorname{CPD}_{t \alpha_1,(\lambda_\ell)_\ell}(n) \right| = 0\text{, }\mathbb{P}\text{-a.s.},
\end{equation*}where the equivalence is because the former term is precisely the density of such a compound Poisson distribution (see equation (\ref{eq:cpddensity})). 

As a consequence, we can conclude the proof with
\begin{equation*}
    \hspace{-60mm} \varlimsup_{m \to \infty} \left| \mu_\omega(Z^{\omega,N}_{\Gamma_{\rho_m}} = n) - \operatorname{CPD}_{t \alpha_1,(\lambda_\ell)_\ell}(\{n\}) \right|
\end{equation*}
\begin{equation*}
\hspace{45mm}\begin{tabular}{c l}
     $\le$ & $\displaystyle \lim_{L \to \infty} \varlimsup_{m \to \infty} \left| \mu_\omega(Z^{\omega,N}_{\Gamma_{\rho_m}} = n) - \mu_\omega\left( \sum_{j=0}^{N'-1} \tilde{Z}^{\omega}_j = n \right) \right|$  \\
    & \\
     $+$ & $\displaystyle \lim_{L \to \infty}\varlimsup_{m \to \infty} \left| \mu_\omega\left( \sum_{j=0}^{N'-1} \tilde{Z}^{\omega}_j = n \right) - \operatorname{CPD}_{t \alpha_1,(\lambda_\ell)_\ell}(n) \right| $\\
     $=$ & $0$, $\mathbb{P}$-a.s.
\end{tabular}
\end{equation*}


\section{Proof of theorem \ref{thm:main2}}\label{sec:proofmain2}

By \cite{kallenberg1983random} theorem 4.2, it suffices to show that for any $k {\ge} 1$, $0 {\le} a_1 {<} b_1 {\le} {\ldots} {\le} a_k {<} b_k {\le} 1$ and $n_1, {\ldots}, n_k {\ge} 0$:
\begin{equation*}
    \hspace{-40mm}\mu_\omega\left( Y_{\rho_m}^{\omega,\lfloor \frac{t}{\hat\mu(\Gamma_{\rho_m})} \rfloor}\left( [a_1,b_1) \right) =n_1, \ldots, Y_{\rho_m}^{\omega,\lfloor \frac{t}{\hat\mu(\Gamma_{\rho_m})} \rfloor}\left( [a_k,b_k) \right) =n_k \right) 
\end{equation*}
\begin{equation}\label{eq:cpppcheck}
    \hspace{60mm}\stackrel[m \to \infty]{\mathbb{P}\text{-a.s.}}{\longrightarrow} \mathbb{Q}\Big(N([a_1,b_1))=n_1, \ldots, N([a_k,b_k))=n_k  \Big),
\end{equation}where $N: (\mathcal{X}, \mathscr{X}, \mathbb{Q}) \to \mathfrak{M}$ with $N_* \mathbb{Q} = CPPP_{t \alpha_1, (\lambda_\ell)_\ell}$ complies with definition \ref{def:cppp}.

To simplify the presentation, we consider $k=2$ and that, when needed, fractions divided by $\hat\mu(\Gamma_{\rho_m})$ or $L$ already make an integer.

Write, for $q=1,2$,
\begin{equation*}
    \begin{tabular}{c c c c}
        $\displaystyle A_q = \frac{\displaystyle a_q t}{\displaystyle \hat\mu(\Gamma_{\rho_m})}$ &  $\displaystyle B_q = \frac{\displaystyle b_q t}{\displaystyle \hat\mu(\Gamma_{\rho_m})}$ &  $N_q = \frac{\displaystyle (b_q - a_q)t}{\displaystyle \hat\mu(\Gamma_{\rho_m})}$ & $N'_q = \frac{\displaystyle N_q}{\displaystyle L}$\\
         & & $N = \frac{\displaystyle t}{\displaystyle \hat\mu(\Gamma_{\rho_m})}$ & $N' = \frac{\displaystyle N}{\displaystyle L}$
    \end{tabular}.
\end{equation*}

So the left side of equation (\ref{eq:cpppcheck}) becomes 
\begin{equation*}
    \mu_\omega \left( \sum_{i=A_1}^{B_1 -1} I_{i}^{\omega,m} =n_1, \sum_{i=A_2}^{B_2 -1} I_{i}^{\omega,m} =n_2  \right)
\end{equation*}
\begin{equation*}
     = \mu_\omega \Bigg( \sum_{j=0}^{N'_1-1} \underbrace{\sum_{i=A_1+jL}^{A_1+(j+1)L -1} I_{i}^{\omega,m}}_{\displaystyle Z_{j}^{\omega,1}} =n_1, \sum_{j=0}^{N'_2-1} \underbrace{\sum_{i=A_2+jL}^{A_2+(j+1)L -1} I_{i}^{\omega,m}}_{\displaystyle Z_{j}^{\omega,2}} =n_2  \Bigg).
\end{equation*}

With $\mathcal{I} = \{(q,j): q=1,2, j=0,\ldots, N'_q-1\}$, the family of random variables $(Z_{j}^{\omega,q})_{(q,j) \in \mathcal{I}}$ is mimicked by an independency $(\tilde{Z}_{j}^{\omega,q})_{(q,j) \in \mathcal{I}}$, $(Z_{j}^{\omega,q})_{(q,j) \in \mathcal{I}} \perp (\tilde{Z}_{j}^{\omega,q})_{(q,j) \in \mathcal{I}}$, $Z_{j}^{\omega,q} \sim \tilde{Z}_{j}^{\omega,q}$, for $(q,j) \in \mathcal{I}$.

In analogy to the approximation theorem, we then want to bound
\begin{equation}\label{eq:cpppestimate}
    \left| \mu_\omega \Bigg( \sum_{j=0}^{N'_1-1}  Z_{j}^{\omega,1} =n_1, \sum_{j=0}^{N'_2-1} Z_{j}^{\omega,2} =n_2  \Bigg) - \mu_\omega \Bigg( \sum_{j=0}^{N'_1-1}  \tilde{Z}_{j}^{\omega,1} =n_1, \sum_{j=0}^{N'_2-1} \tilde{Z}_{j}^{\omega,2} =n_2  \Bigg) \right|.
\end{equation}

Denote $\tilde{W}_{a,b}^{\omega,q} = \sum_{j=a}^b \tilde{Z}_j^{\omega,q}$ and similarly without $\sim$'s.

Then
\begin{equation*}
    \begin{tabular}{c c l}
       $\displaystyle \text{eq.}(\ref{eq:cpppestimate})$  & $\le$ & $\left | \mu_\omega\left(W_{0,N'_1-1}^{\omega,1}=n_1, W_{0,N'_2-1}^{\omega,2}=n_2\right) - \mu_\omega\left(\tilde{W}_{0,N'_1-1}^{\omega,1}=n_1, W_{0,N'_2-1}^{\omega,2}=n_2\right) \right|$  \\
         & $+$ & $\left| \mu_\omega\left(\tilde{W}_{0,N'_1-1}^{\omega,1}=n_1, W_{0,N'_2-1}^{\omega,2}=n_2\right) - \mu_\omega\left(\tilde{W}_{0,N'_1-1}^{\omega,1}=n_1, \tilde{W}_{0,N'_2-1}^{\omega,2}=n_2\right) \right|$\\
         & $=:$ & $(\vartriangle) + (\triangledown)$.
    \end{tabular}
\end{equation*}

We consider $(\vartriangle)$ first. Repeating the telescoping argument in the proof of theorem \ref{thm:approx}, we have
\begin{equation*}
\begin{tabular}{c c l l}
    $\displaystyle (\vartriangle)$ & $\displaystyle \le$ & $\displaystyle \sum_{j=0}^{N'_1-1}$ & $\displaystyle \left|\substack{\displaystyle \mu_\omega(\tilde{W}_{0,j-1}^{\omega,1} + W_{j,N'_1-1}^{\omega,1}=n_1, W_{0,N'_2-1}^{\omega,2}=n_2 ) \\ \displaystyle -\mu_\omega(\tilde{W}_{0,j}^{\omega,1} + W_{j+1,N'_1-1}^{\omega,1}=n_1, W_{0,N'_2-1}^{\omega,2}=n_2 )} \right|$ \\
    & $\displaystyle \le$ & $\displaystyle \sum_{j=0}^{N'_1-1} \sum_{l=0}^{n_1}$&$\displaystyle \mu_\omega(\tilde{W}_{0,j-1}^{\omega,1}=l) \left|\substack{\displaystyle \mu_\omega(W_{j,N'_1-1}^{\omega,1}=n_1-l, W_{0,N'_2-1}^{\omega,2}=n_1) \\ \displaystyle - \mu_\omega(\tilde{Z}_{j}^{\omega,1} + W_{j+1,N'_1-1}^{\omega,1}=n_1-l, W_{0,N'_2-1}^{\omega,2}=n_1)}\right|$\\
     & $\le$ & $\displaystyle \sum_{j=0}^{N'_1-1} \sum_{q=0}^{n_1} \sum_{u=0}^q $ & $\displaystyle \left|\substack{\displaystyle \mu_\omega(Z_{j}^{\omega,1}=u,W_{j+1,N'_1-1}^{\omega,1}=q-u, W_{0,N'_2-1}^{\omega,1}=n_2) \\ \displaystyle - \mu_\omega(Z_{j}^{\omega,1}=u)\mu_\omega(W_{j+1,N'_1-1}^{\omega,1}=q-u, W_{0,N'_2-1}^{\omega,1}=n_2)}\right|$
\end{tabular}
\end{equation*}

One has to single out $u=0$ from $u\in[1,q]$. We focus on the principal part $u\in[1,q]$, which can be bounded by the sum of the following three terms (note the unusual order).

\begin{equation*}
    (\vartriangle)_2 \le \sum_{j=0}^{N'_1-1} \sum_{q=1}^{n_1} \sum_{u=1}^q \left|\substack{\displaystyle \mu_\omega(Z_{j}^{\omega,1}=u,W_{j+1,N'_1-1}^{\omega,1}=q-u, W_{0,N'_2-1}^{\omega,1}=n_2) \\ \displaystyle - \mu_\omega(Z_{j}^{\omega,1}=u,W_{j+\Delta,N'_1-1}^{\omega,1}=q-u, W_{0,N'_2-1}^{\omega,1}=n_2) }\right|,
\end{equation*}

\begin{equation*}
    (\vartriangle)_1 \le \sum_{j=0}^{N'_1-1} \max_{q \in [1,n_1]} \sum_{u=1}^q \left|\substack{\displaystyle \mu_\omega(Z_{j}^{\omega,1}=u,W_{j+\Delta,N'_1-1}^{\omega,1}=q-u, W_{0,N'_2-1}^{\omega,1}=n_2) \\ \displaystyle - \mu_\omega(Z_{j}^{\omega,1}=u)\mu_\omega(W_{j+\Delta,N'_1-1}^{\omega,1}=q-u, W_{0,N'_2-1}^{\omega,1}=n_2) }\right|,
\end{equation*}

\begin{equation*}
    (\vartriangle)_3 \lesssim \sum_{j=0}^{N'_1-1} \sum_{q=1}^{n_1} \sum_{u=1}^q \left|\substack{\displaystyle \mu_\omega(Z_{j}^{\omega,1}=u)\mu_\omega(W_{j+\Delta,N'_1-1}^{\omega,1}=q-u, W_{0,N'_2-1}^{\omega,1}=n_2) \\ \displaystyle - \mu_\omega(Z_{j}^{\omega,1}=u)\mu_\omega(W_{j+1,N'_1-1}^{\omega,1}=q-u, W_{0,N'_2-1}^{\omega,1}=n_2)}\right|.
\end{equation*}

The bound of $(\vartriangle)_1$ can be handled pretty much as in the proof of theorem \ref{thm:main}. Minor modifications are needed and we discuss them now. Notice that the first term inside absolute value in $(\vartriangle)_1$ can written as
\begin{equation*}
    \mu_\omega\left(\mathbbm{1}_{ \{ \sum_{i=A_1+jL}^{A_1+(j+1)L-1} I_i^{\omega,m} =u \} } 
    \mathbbm{1}_{ \{ \sum_{i=A_1+(j+\Delta)L}^{B_1-1} I_i^{\omega,m} =q-u \} }
    \mathbbm{1}_{ \{ \sum_{i=A_2}^{B_2-1} I_i^{\omega,m} =n_2 \} }\right)
\end{equation*}and, with $\omega'= \theta^{A_1+jL}\omega$,
\begin{equation*}
    =\mu_{\omega'}\left(\mathbbm{1}_{ \{ \sum_{i=0}^{L-1} I_i^{\omega',m} =u \} } 
    \mathbbm{1}_{ \{ \sum_{i=\Delta L}^{N_1-1-jL} I_i^{\omega',m} =q-u \} }
    \mathbbm{1}_{ \{ \sum_{i=A_2-A_1-jL)}^{B_2-1-A_1-jL)} I_i^{\omega',m} =n_2 \} }\right)
\end{equation*}
\begin{equation*}
    = \mu_{\omega'}\left(\mathbbm{1}_{ \{ \sum_{i=0}^{L-1} I_i^{\omega',m} =u \} } \left[ 
    \mathbbm{1}_{ \{ \sum_{i=0}^{N_1-1-jL-\Delta L} I_i^{\theta^{\Delta L} \omega',m} =q-u \} \cap \{ \sum_{i=A_2-A_1-jL-\Delta L)}^{B_2-1-A_1-jL-\Delta L)} I_i^{\theta^{\Delta L}\omega',m} =n_2 \} } \right] \circ T_{\omega'}^{\Delta L}\right),
\end{equation*}where the last step is because
\begin{equation*}
    A_2-A_1 -jL \ge A_2-A_1 - (N'_1-1)L \ge A_2 - A_1 - (B_1 - A_1) = A_2 - B_1 \ge \Delta L,   
\end{equation*}with the latter inequality following from being $\Delta L \in [1, (A_2 - B_1)/L]$ after choosing $\Delta := {\rho_m}^{-v}$, for some $v \in (0,d_0)$ and considering $m$ large enough (dependent on $L$) so that the first inequality below holds.
\begin{equation*}
    {\rho_m}^{-v} \le L^{-1}\frac{(a_2-b_1)t}{C {\rho_m}^{d_0}} \le L^{-1}\frac{(a_2-b_1)t}{\hat\mu(\Gamma_{\rho_m})} = \frac{A_2-B_1}{L}. 
\end{equation*}

With the positive separation $\Delta L$, we can follow the treatment of $\mathcal{R}^1_{\omega,m}(N,L,\Delta)$ in section \ref{sec:R1estimate}: a) the function not composed with the dynamics should be given a Lipschitz approximation (and it is the same function that appeared before), b) the function composed with the dynamics is more complicated, but we only care about its sup norm, which is $1$ anyway, c) quenched decay of correlations can be applied again, proceeding just as before.  

To control the singled out term $u=0$ one repeats the strategy in the proof of theorem \ref{thm:approx} with what we did above to control the principal part. Using the notation from the proof of theorem \ref{thm:approx}, errors with $\sim$'s will appear, only the first of which still matters at the end (the others are dominated by the respective errors without $\sim$'s). We omit this part.

The bound of $(\vartriangle)_2$, just like in the proof of theorem \ref{thm:main}, is estimated from above by
\begin{equation*}
    \sum_{j=0}^{N'_1-1} \max_{q \in [1,n_1]} \sum_{u=1}^q \mu_\omega(Z_{j}^{\omega,1}=u,W_{j+1,j+\Delta-1}^{\omega,1}\ge 1, W_{0,N'_2-1}^{\omega,1}=n_2),
\end{equation*}
\begin{equation*}
    \le \sum_{j=0}^{N'_1-1} \mu_\omega(Z_{j}^{\omega,1} \ge 1,W_{j+1,j+\Delta-1}^{\omega,1}\ge 1),
\end{equation*}which is pretty much identical to $\mathcal{R}^2_{\omega,m}(N,L,\Delta)$ and can be controlled just like we did in section \ref{sec:R2estimate}.

We also omit the discussion of $(\vartriangle)_3$, which should be treated analogously.

So the error terms associated with $(\vartriangle)$ end up being treated just like the errors already controlled in the proof of theorem \ref{thm:main}.

Now we consider $(\triangledown)$. Repeating the telescopic argument once more, we have
\begin{equation*}
\begin{tabular}{c c l @{} l}
    $\displaystyle (\triangledown)$ & $\displaystyle \le$ & $\displaystyle \sum_{j=0}^{N'_2-1}$ & $\displaystyle \left|\substack{\displaystyle \mu_\omega(\tilde{W}_{0,N'_1-1}^{\omega,1}=n_1, \tilde{W}_{0,j-1}^{\omega,2} + W_{j,N'_2-1}^{\omega,2} =n_2 ) \\ \displaystyle - \mu_\omega(\tilde{W}_{0,N'_1-1}^{\omega,1}=n_1, \tilde{W}_{0,j}^{\omega,2} + W_{j+1,N'_2-1}^{\omega,2} =n_2 )} \right|$ \\
    & $\displaystyle \le$ & $\displaystyle \sum_{j=0}^{N'_2-1} \sum_{l=0}^{n_2}$&$\displaystyle \mu_\omega(\tilde{W}_{0,j-1}^{\omega,2}=l) \left|\substack{\displaystyle \mu_\omega(\tilde{W}_{0,N'_1-1}^{\omega,1}=n_1, W_{j,N'_2-1}^{\omega,2}=n_2-l ) \\ \displaystyle -\mu_\omega(\tilde{W}_{0,N'_1-1}^{\omega,1}=n_1, \tilde{Z}_j^{\omega,2} + W_{j+1,N'_2-1}^{\omega,2} =n_2)}\right|$\\
     & $\le$ & $\displaystyle \sum_{j=0}^{N'_2-1} \sum_{q=0}^{n_2} $ & $\displaystyle \left|\substack{\displaystyle \mu_\omega(\tilde{W}_{0,N'_1-1}^{\omega,1}=n_1, Z_j^{\omega,2} + W_{j+1,N'_2-1}^{\omega,2} =q) \\ \displaystyle - \mu_\omega(\tilde{W}_{0,N'_1-1}^{\omega,1}=n_1, \tilde{Z}_j^{\omega,2} + W_{j+1,N'_2-1}^{\omega,2} =q)}\right|$\\
     & $\le$ & $\displaystyle \sum_{j=0}^{N'_2-1} \sum_{q=0}^{n_2} \sum_{u=0}^q $ & $\displaystyle \left|\mu_\omega( Z_j^{\omega,2} {=}u, W_{j+1,N'_2-1}^{\omega,2} {=}q{-}u) - \mu_\omega(Z_j^{\omega,2} {=}u)\mu_\omega( W_{j+1,N'_2-1}^{\omega,2} {=}q{-}u)\right|.$ 
\end{tabular}
\end{equation*}

The latter expression is essentially the same of that encountered at the end of the telescopic argument in the proof of theorem \ref{thm:approx}. Therefore it can be bounded in the same manner, with errors $\tilde{\mathcal{R}}^1, \mathcal{R}^1, \mathcal{R}^2$ and $\mathcal{R}^3$, which can then be controlled just as in the proof of theorem \ref{thm:main}.

Finally, using independency and section \ref{sec:leading}, the leading term appearing on the second part of equation \ref{eq:cpppestimate} converges, as desired, to
\begin{equation*}
    CPD_{(b_1-a_1)t\alpha_1, (\lambda_\ell)_\ell}(n_1) \cdot CPD_{(b_2-a_2)t\alpha_1, (\lambda_\ell)_\ell}(n_2).
\end{equation*}

\section{Application: random piecewise expanding one-dimensional systems}\label{sec:ex}

We consider a class of random piecewise expanding one-dimensional systems \linebreak$(\theta, \mathbb{P}, T_\omega, \mu_\omega, \Gamma)$  prescribed by the following conditions. Elements in this class immediately comprise a system as in the general setup of section \ref{sec:gensetup} and will check that they also comprise a system as in the working setup of section \ref{sec:wrksetup} (i.e., satisfying hypotheses (H\ref{hyp:amb}-H\ref{hyp:param})).

\begin{cond}\label{cond:finitemaps}
Consider finitely many maps of the unit interval (or circle), $T_v: M \rightarrow M$, for $v \in \{0,\ldots,u-1\}$. For ease of exposition, say that $u=2$. They carry a family of open intervals $A_v = (\zeta_{v,i})_{i=1}^{I_v}$ ($I_v < \infty$) so that $M \setminus \bigcup_{i=1}^{I_v} \zeta_{v,i}$ is finite and $T_v|_{\zeta_{v,i}}$ is surjective and $C^2$-differentiable with
\begin{equation*}
    \begin{tabular}{c c c c c c c c c}
       $1$ & $<$ & $d_{\min}$ & $\le$ & $\inf\{|{T_v}' (x)| : x \in \zeta_{v,i}, v=1,\ldots,I_v, i=0,\ldots,u-1\},$ & & \\
         & & & & $\sup\{|{T_v}'' (x)| : x \in \zeta_{v,i}, v=1,\ldots,I_v, i=0,\ldots,u-1\}$ & $\le$ & $c_{\max}$ & $<$ & $\infty$.
    \end{tabular}
\end{equation*}
\end{cond}

For $n \ge 1$, let $A_n^\omega = \bigvee_{j=0}^{n-1}(T_\omega^j)^{-1} A_{\pi_j(\omega)}$. For $n=0$, we adopt the convention $A^\omega_0 = \{(0,1)\}$ ($\forall \omega \in \Omega$). Write $\mathcal{A}^\omega_n = \bigcup_{\zeta \in A^\omega_n} \zeta$ (co-finite) and, for $x \in \mathcal{A}^\omega_n$, denote by $A_n^\omega(x)$ the element of $A_n^\omega$ containing $x$. In particular, $x \in \mathcal{A}^\omega_{n}$ implies that $x$ is a point of differentiability for $T_\omega^n$.

\begin{cond}\label{cond:drawing}
Let $\Omega = \{0,1\}^{\mathbb{Z}}$. Set $T_\omega := T_{\pi_0(\omega)}$, where $\pi_j(\omega) = \omega_j$ $(j \in \mathbb{Z})$. Consider $\theta: \Omega \rightarrow \Omega$ to be the bilateral shift map.   
\end{cond}

\begin{cond}\label{cond:drivingmeas} Consider $\mathbb{P} \in \mathcal{P}_\theta(\Omega)$ an equilibrium state associated to a Lipschitz potential. Usual instances are Bernoulli and Markov measures.
\end{cond}






\begin{cond}\label{cond:goodtarget}
Consider $\Gamma(\omega) = \{x(\omega)\}$ ($\omega \in \Omega$), where $x: \Omega \rightarrow M$ is a random variable taking values either $x_0$ or $x_1$ (possibly coincident) in the form $x(\omega) = x_{\pi_0(\omega)}$, with $\{x_0,x_1\} \subset \bigcap_{\omega \in \Omega} \bigcap_{l = 1}^\infty  \mathcal{A}^{\omega}_l\text{ }$\footnote{The intersection $\bigcap_{l = 1}^\infty \mathcal{A}^{\omega}_l$ is a co-countable set ($\forall \omega \in \Omega$).} (which needs to be a non-empty set). 

Moreover, for each $\omega \in \Omega$, with the minimal period 
\begin{equation*}
    m(\omega):= \min\{ m \ge 1 : T_\omega^{m(\omega)} x(\omega) = x(\theta^{m(\omega)} \omega) \} \in \mathbb{N}_{\ge 1} \cup \{\infty\},
\end{equation*}one defines the number of finite-periods occurring along the $\omega$ fiber ($K(\omega) \in 
\mathbb{N}_{\ge 0} \cup \{\infty\}$) and the associated sequence of such periods ($(m_j(\omega))_{j=0}^{K(\omega)-1} \subset \mathbb{N}_{\ge 1}$), using the conventions $m_{-1}(\omega) :\equiv 0$ and $\max \emptyset := 0$, letting

\begin{equation*}
    K(\omega) := \max\left\{k \ge 1 :     \begin{tabular}{l @{} l}
         $m_0(\omega) := m(\omega)$ & $\in \mathbb{N}_{\ge 1}$ \\
         $m_1(\omega) := m(\theta^{m_0(\omega)}\omega)$ & $\in \mathbb{N}_{\ge 1}$\\
         $m_2(\omega) := m(\theta^{m_1(\omega)+ m_0(\omega)}\omega)$ & $\in \mathbb{N}_{\ge 1}$\\
         \hspace{11mm}$\ldots$ \\
         $m_{k-1}(\omega) := m(\theta^{m_{k-2}(\omega)+ \ldots + m_0(\omega)}\omega)$ & $\in \mathbb{N}_{\ge 1}$
    \end{tabular} \right\} \in \mathbb{N}_{\ge 0} \cup \{\infty\}.
\end{equation*}

In particular, writing $M_j(\omega) := \sum_{k=0}^{j-1} m_k(\omega)$ for $1 \le j \le K(\omega)$ (with $M_0(\omega) :\equiv 0$), one has:
\begin{equation*}
    x(\omega) \xmapsto[]{\displaystyle T_{\omega}^{m_0(\omega)}} x(\theta^{M_1(\omega)}\omega) \xmapsto[]{\displaystyle T_{\theta^{M_1(\omega)}\omega}^{m_1(\omega)}} x(\theta^{M_2(\omega) }\omega) \xmapsto[]{\displaystyle T_{\theta^{M_2(\omega)}\omega}^{m_2(\omega)}} x(\theta^{M_3(\omega)}\omega) \text{ }\ldots\text{ }.
\end{equation*}

We conclude (C\ref{cond:goodtarget}) assuming that the target satisfies the dynamical condition that
\begin{equation*}
    \sup\{ m_j(\omega) : \omega \in \Omega, j = 0, \ldots, K(\omega)-1\} =: M_\Gamma < \infty,
\end{equation*}where the convention $\sup \emptyset := 0$ is adopted.
\end{cond}

\begin{cond}\label{cond:leb}
     Consider that there exists $r >0$, $K,Q>1$ and $\beta \in (0,1]$ so that $\mu_\omega = h_\omega \operatorname{Leb}$ forms a quasi-invariant family satisfying: i) $(\omega,x) \mapsto h_\omega(x)$ is measurable, ii) $K^{-1} \le h_\omega|_{B_r(x(\omega))} \le K$ a.s., and iii)  $h_\omega|_{B_r(x(\omega))} \in \operatorname{Hol}_\beta(M)$ with $H_\beta(h_\omega|_{B_r(x(\omega))})$ a.s.. See remark \ref{rmk:cones}.
\end{cond}


The following result says that theorem \ref{thm:main} applies to systems in the class (C\ref{cond:finitemaps}-C\ref{cond:leb}) and, in particular, they have quenched limit entry distributions in the compound Poisson class with the needed statistical quantities presented explicitly. 

\begin{thrm}\label{thm:expanding}
Let $(\theta, \mathbb{P}, T_\omega, \mu_\omega, \Gamma)$ be a system satisfying conditions (C\ref{cond:finitemaps}-C\ref{cond:leb}). Then the hypotheses of theorem \ref{thm:main} are satisfied with 
\begin{equation*}
\resizebox{1\hsize}{!}{
\begin{tabular}{@{}c @{}}
$\displaystyle \alpha_\ell = \ \hspace{2mm}\begin{cases}
    \frac{\displaystyle h_\omega(x(\omega))}{\displaystyle \int_\Omega h_\omega(x(\omega)) d \mathbb{P}(\omega) } \left[ \left(JT_\omega^{M_{\ell-1}(\omega)}(x(\omega))\right)^{-1}  - \left(JT_\omega^{M_{\ell}(\omega)}(x(\omega))\right)^{-1} \right] \text{, if }\ell \le K(\omega) \\
        \frac{\displaystyle h_\omega(x(\omega)) }{\displaystyle \int_\Omega h_\omega(x(\omega))  d \mathbb{P}(\omega) } \left[ \left(JT_\omega^{M_{\ell-1}(\omega)}(x(\omega))\right)^{-1} \right] \hspace{40mm}  \text{, if }\ell=K(\omega)+1 \\
        0\hspace{115mm}\text{, if }\ell \ge K(\omega)+2
    \end{cases} d\mathbb{P}(\omega).$
\end{tabular}}
\end{equation*}The quantities $\alpha_\ell$ comply with  (H\ref{hyp:return}) and theorem \ref{thm:lambdaalpha}, allowing for $\lambda_\ell = (\alpha_\ell - \alpha_{\ell+1})/\alpha_1$ to hold.

In particular: $\forall t {>} 0, \forall (\rho_m)_{m \ge 1} {\searrow} 0 \text{ with }\sum\nolimits_{m \ge 1} {\rho_m}^q {<} \infty \text{ } \text{(for some } 0 {<} q {<} 1$) one has
\begin{equation*}\displaystyle
 \qquad \mu_\omega(Z^{\omega,\lfloor t / \hat\mu(\Gamma_{\rho_m}) \rfloor}_{\Gamma_{\rho_m}}= n) \stackrel[m \rightarrow \infty]{\mathbb{P}\text{-a.s.}}{\longrightarrow} \operatorname{CPD}_{t \alpha_1,(\lambda_\ell)_\ell}(n) \text{ }(\forall n \ge 0),
\end{equation*}and
\begin{equation}\displaystyle
 \qquad {Y^{\omega,\lfloor t / \hat\mu(\Gamma_{\rho_m}) \rfloor}_{\Gamma_{\rho_m}}}_* \mu_\omega \stackrel[m \rightarrow \infty]{\mathbb{P}\text{-a.s.}}{\longrightarrow} \operatorname{CPPP}_{t \alpha_1,(\lambda_\ell)_\ell} \text{in }\mathcal{P}(\mathfrak{M}). 
\end{equation}
\end{thrm}

We will prove the theorem after a few remarks on relevant subclasses within (C\ref{cond:finitemaps}-C\ref{cond:leb}) and examples.

\begin{remark}\label{rmk:linear}
    When the maps $T_v$ are piecewise expanding linear maps, they preserve Lebesgue and conditions (C\ref{cond:finitemaps})-(C\ref{cond:drivingmeas}),(C\ref{cond:leb}) are immediately satisfied.
\end{remark}

    To illustrate condition (C\ref{cond:goodtarget}), or, better said, condition $M_\Gamma < \infty$, we can look at deterministic targets $x(\omega) \equiv x$. Two noticeable cases occur:

\begin{enumerate}[label=\roman*)]
    \item \textit{Pure periodic points} $x$: when there is some $m_* = m_*(x) \ge 1$ so that $x$ is (minimally) fixed by any concatenations of $m_*$ maps in $(T_v)_{v=0}^{u-1}$. In this case, $m(\omega) \equiv m_*, K(\omega) \equiv \infty, m_j(\omega) \equiv m_*$ and $M_\Gamma = m_*$.
    
    \vspace{3mm}
    It is convenient to represent these types of examples with diagrams (that can neglect topological information), where the deterministic target $x$ is highlighted with a green ball, each arrow indicates how each map $T_v$ acts, blue cycles indicate cycles that avoid the target, purple paths indicate paths between the blue cycles and the target and yellow cycles indicate cycles that include the target (but are not obtained composing blue cycles with purple paths).

    \begin{figure}[H]
    \centering
    \parbox{6.9cm}{\centering
    \includegraphics[width=3.5cm]{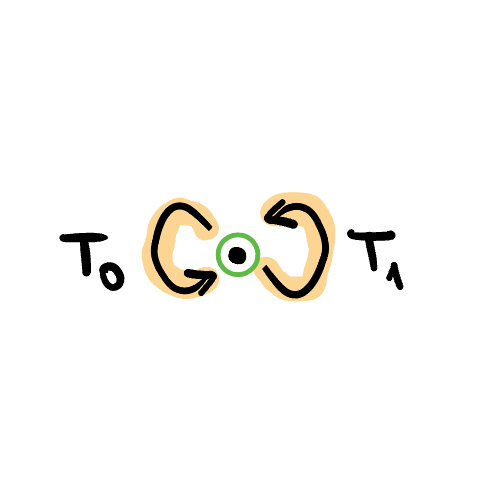}
    \caption{(a) Pure \\one-periodic diagram.}
    \label{fig:diagram11}}
    \qquad
    \begin{minipage}{6.9cm} \centering
    \includegraphics[width=3.5cm]{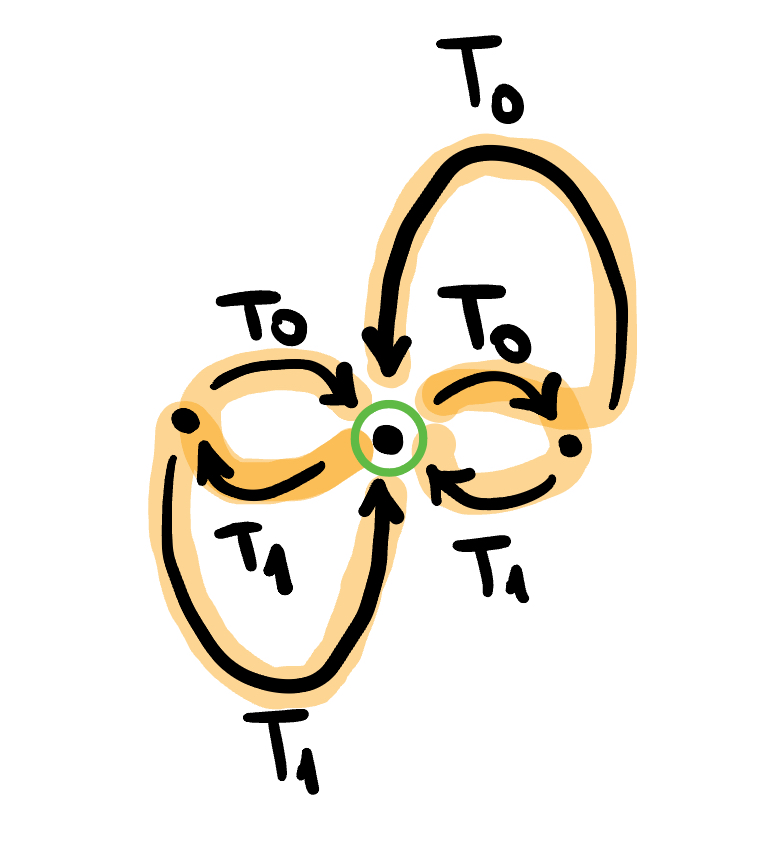}
    \caption{(b) Pure \\two-periodic diagram.}
    \label{fig:diagram12}
    \end{minipage}
    \end{figure}

    Considering remark \ref{rmk:linear}, we can easily present \underline{explicit examples} of systems complying with cases (a) and (b) above. In both examples, $x(\omega) \equiv 1/2$ and all maps preserve Lebesgue. Constructions of this kind are possible for any $m_* \ge 1$ and $u \ge 1$.

    \begin{figure}[H]
    \centering
    \parbox{6.9cm}{\centering
    \includegraphics[width=5cm]{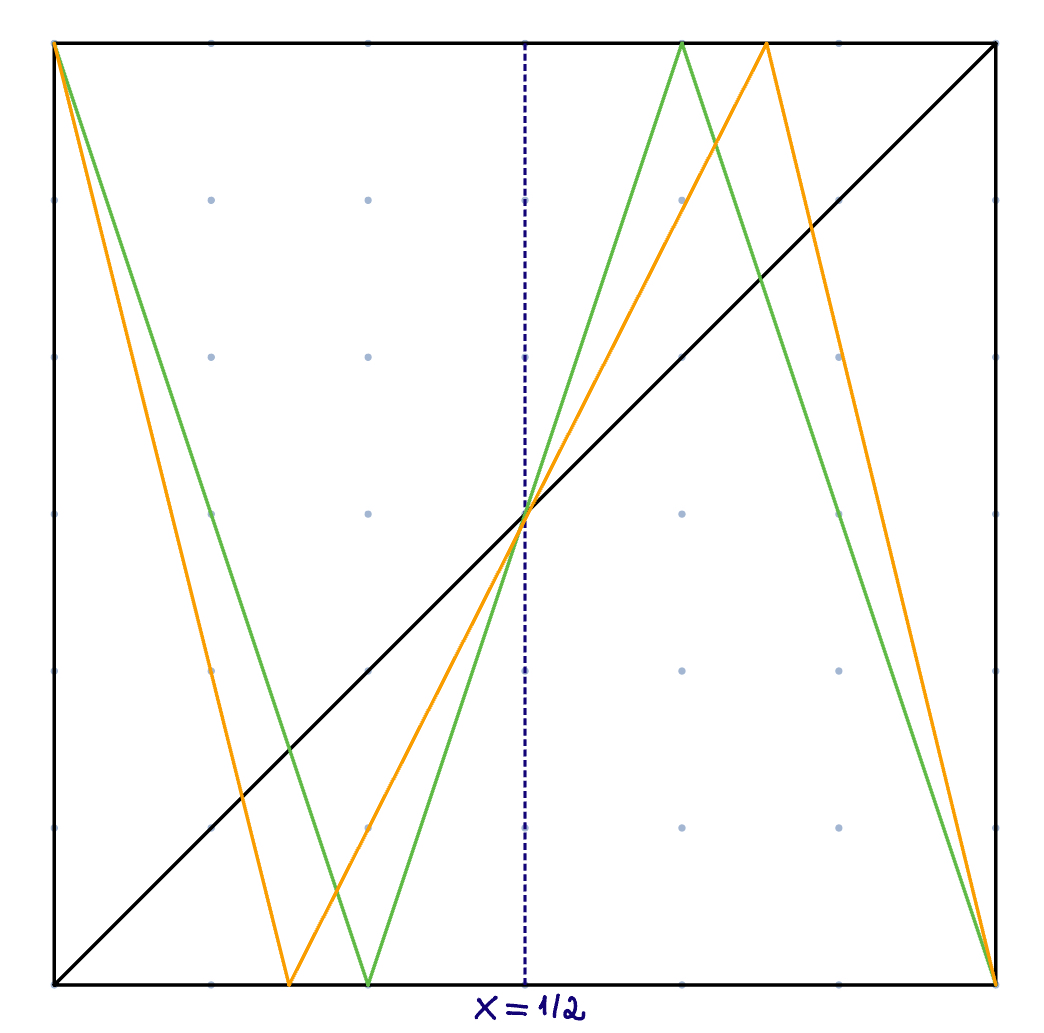}
    \caption{(a) A pure \\one-periodic system.\protect\footnotemark}
    \label{fig:graf1}}
    \qquad
    \begin{minipage}{6.9cm} \centering
    \includegraphics[width=5cm]{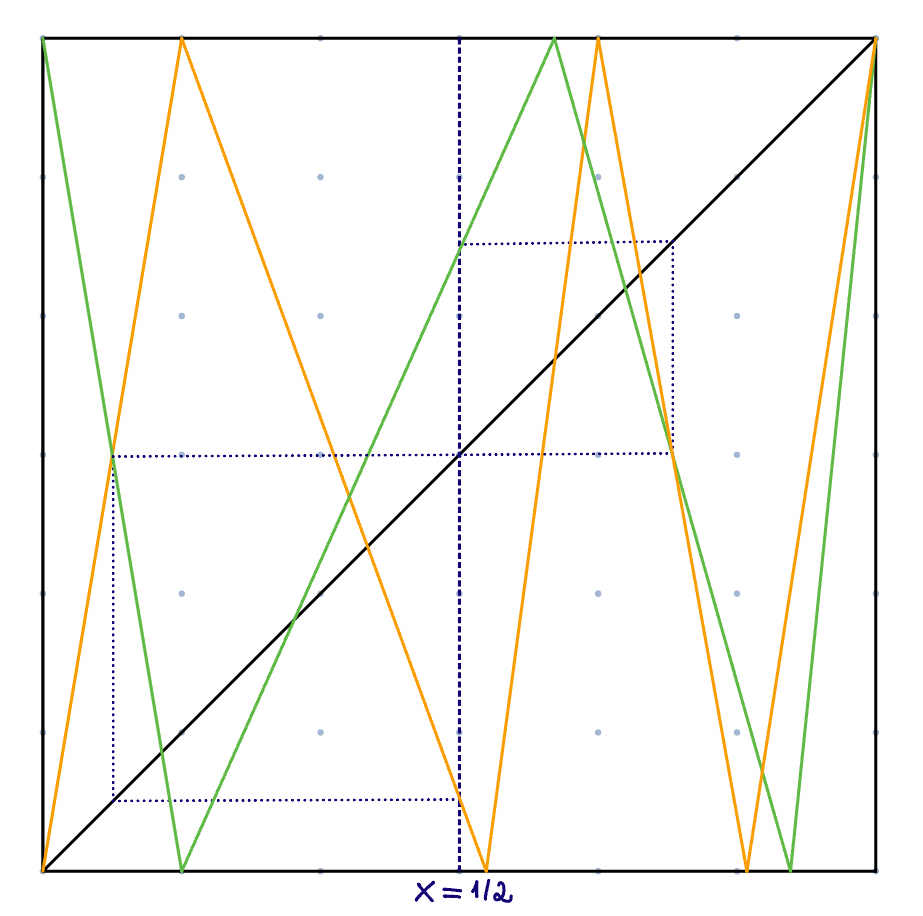}
    \caption{(b) A pure \\two-periodic system.}
    \label{fig:graf2}
    \end{minipage}
    \end{figure}
    
    \item \textit{Pure aperiodic points} $x$: when $x$ is not fixed by any finite concatenation of maps in $(T_v)_{v=0}^{u-1}$. In this case, $m(\omega) \equiv \infty, K(\omega) \equiv 0$ and $M_\Gamma=0$. 

    Here are some compatible diagrams in this case:

    \begin{figure}[H]
    \centering
    \includegraphics[width=0.6\textwidth]{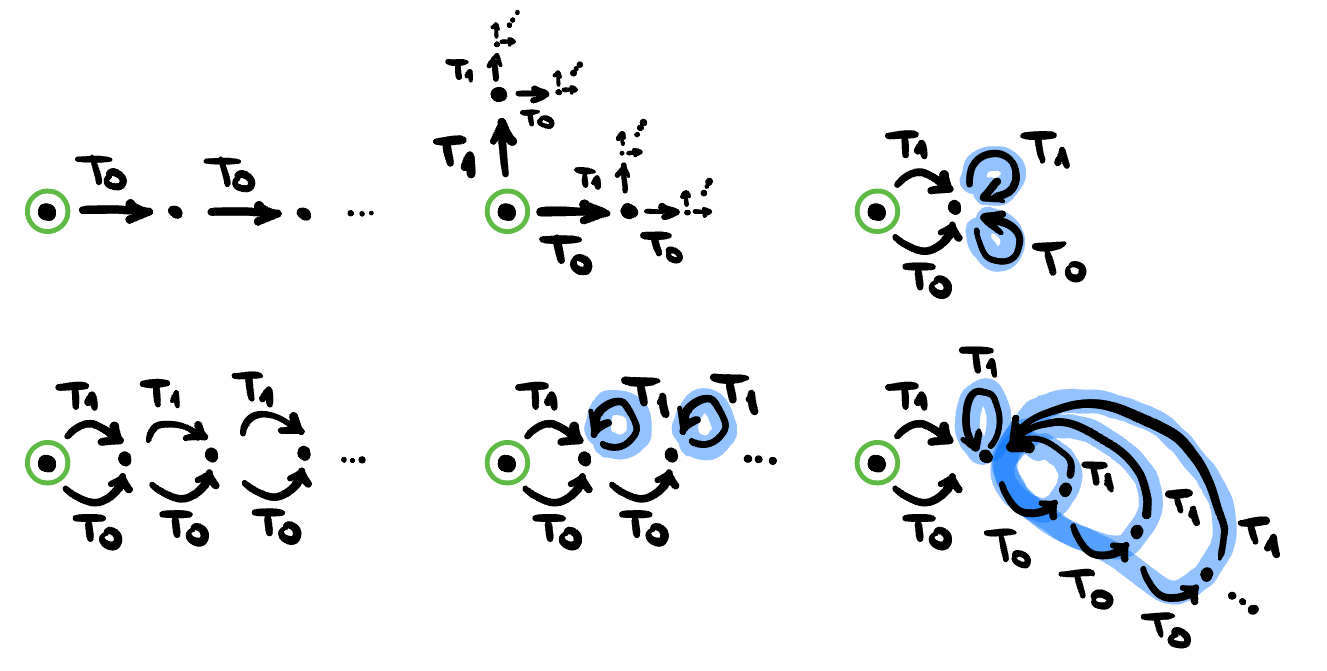}
    \caption{Some pure aperiodic diagrams.}
    \label{fig:diagram2}
    \end{figure}

    Explicit examples realizing these structures (or exhibiting these sorts of behaviors) can be tricky to construct\footnote{We are not claiming that every (possible) diagram compatible with (ii) can be realized by examples in the class (C\ref{cond:finitemaps}-C\ref{cond:leb}).}, especially when the diagram is infinite and one has to control the behavior of infinitely many iterates of the system\footnote{In this direction, beta maps with irrational translation and rational (random) targets were studied in \cite{atnip2023compound}. They do not fit exactly in the class (C\ref{cond:finitemaps}-C\ref{cond:leb}) because they do not have subjective branches. However, they can be dealt with here by considering their action on $S^1$ rather than on $[0,1]$. See remark \ref{rmk:beta}.}. Notice, however, that, once the maps are fixed, the set of pure aperiodic $x$'s is generic, because it is given by $$M \setminus \bigcup_{p \ge 1}\bigcup_{(v_0,\ldots,v_{p-1}) \in \{0,\ldots,u-1\}^{p}}\operatorname{Fix}(T_{v_{p-1}}\circ \ldots \circ T_{v_0}),$$ which is co-countable.

    For a finite diagram such as the last one in the first row, we can consider the following explicit example:

    \begin{figure}[H]
    \centering
    \includegraphics[width=0.35\textwidth]{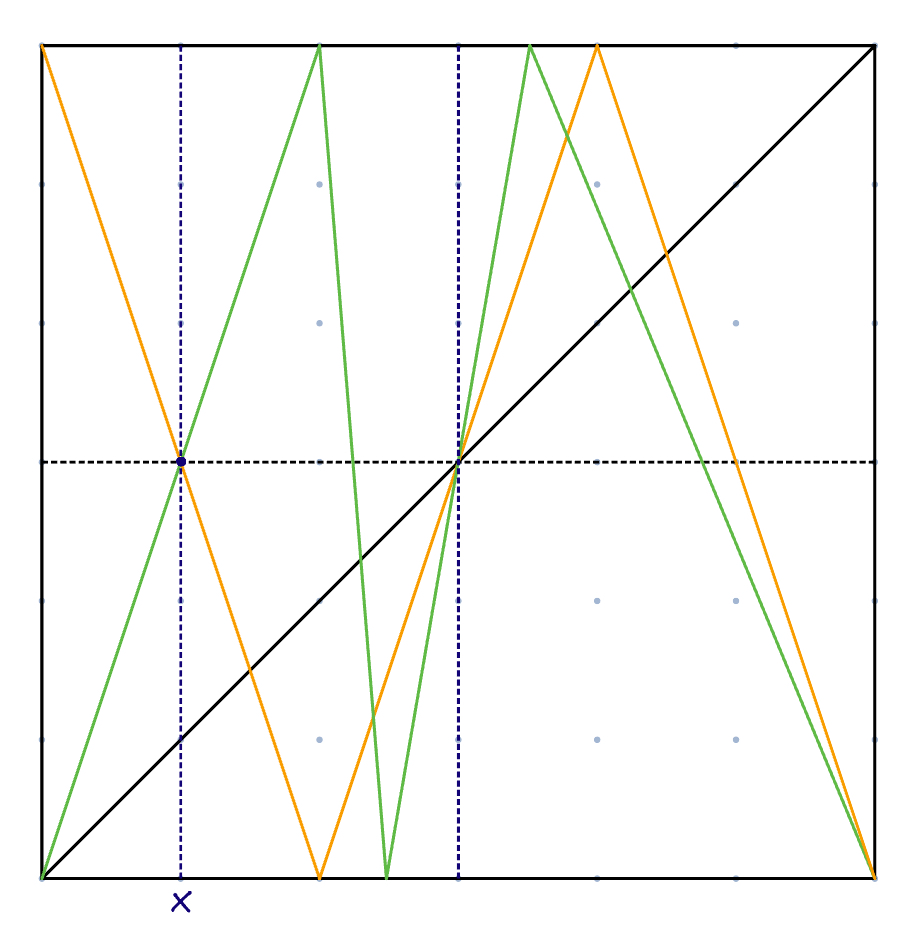}
    \caption{A pure aperiodic system.}
    \label{fig:graf3}
    \end{figure}
    \item \textit{Hybrid}. This is the general case. They can combine the behavior in (i) and (ii) while still verifying $M_\Gamma < \infty$. Here are some possible diagrams in this case:
    \begin{figure}[H]
    \centering
    \includegraphics[width=0.85\textwidth]{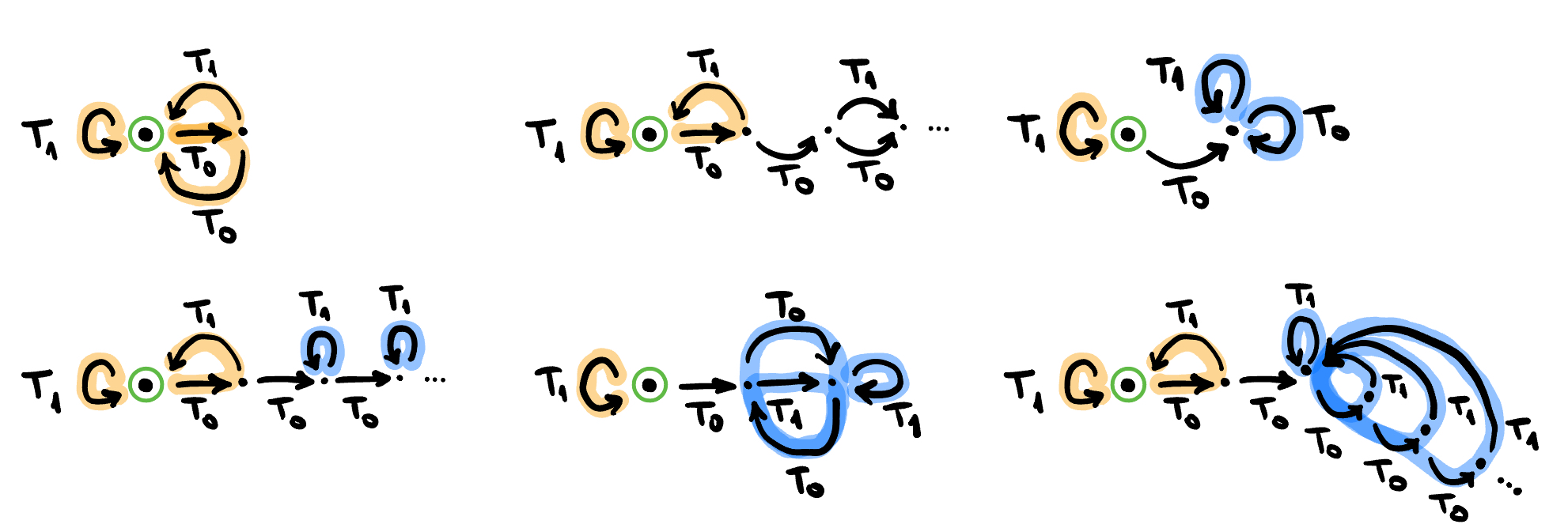}
    \caption{Some hybrid diagrams}
    \label{fig:diagram3}
    \end{figure}    

    For a finite diagram such as the last one in the first row, we can consider the following explicit example:
    \begin{figure}[H]
    \centering
    \includegraphics[width=0.35\textwidth]{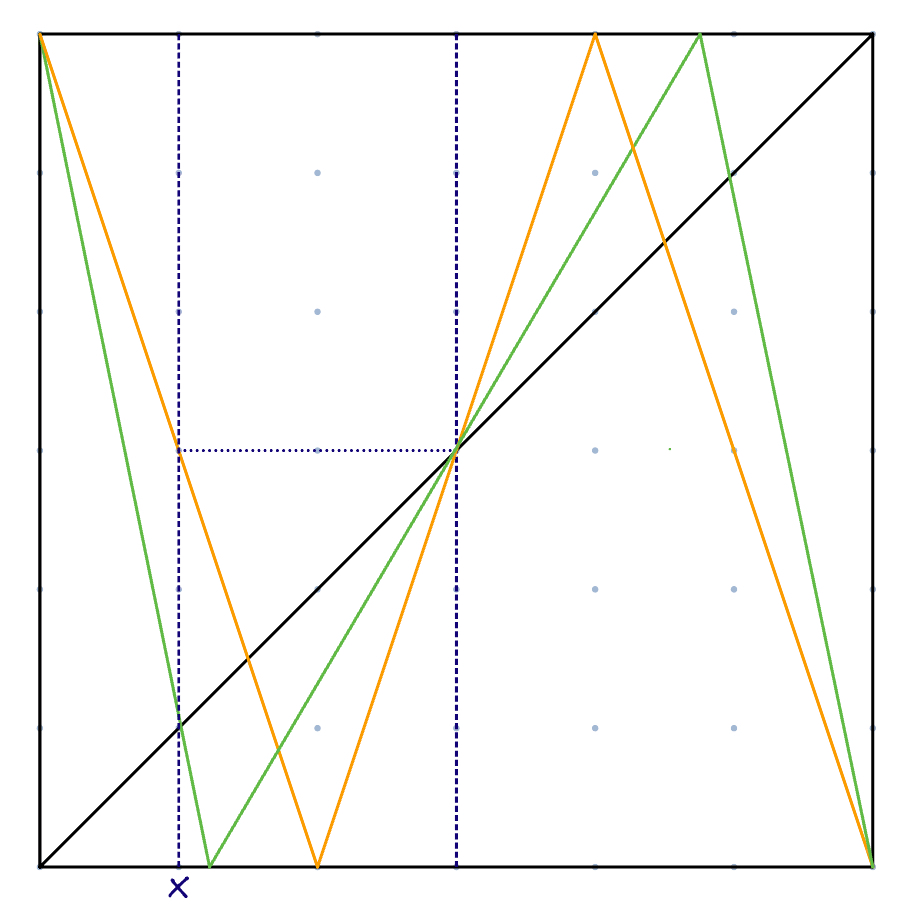}
    \caption{A hybrid system.}
    \label{fig:graf4}
    \end{figure}
    \item \textit{Non-examples}. Here are some diagrams which do not satisfy $M_\Gamma < \infty$. 

    \begin{figure}[H]
    \centering
    \includegraphics[width=0.65\textwidth]{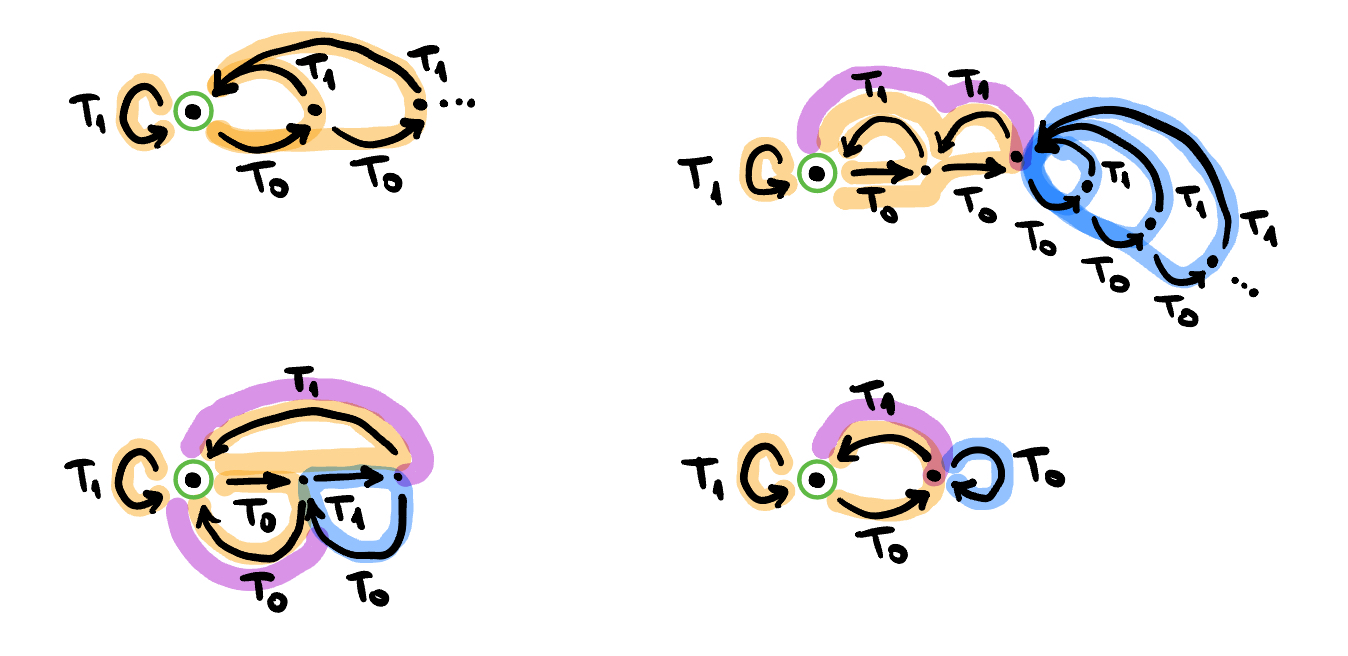}
    \caption{Some non-examples diagrams.}
    \label{fig:diagram4}
    \end{figure}

    Notice that whenever a purple path occurs arbitrarily large periods can be formed. But this can occur without purple paths as well, as in the first diagram. Moreover, this can occur both with infinite diagrams (the first two) and with finite diagrams (the last two).  
\end{enumerate}

\begin{remark}
    It is not being claimed that systems as in (iv) are not covered by the theory presented in theorem \ref{thm:main}. It is just being said that systems as in (iv) are not treated with the techniques used in this section (to calculate underlying $\alpha_\ell$'s).
\end{remark}

\begin{proof}[Proof of theorem \ref{thm:expanding}]
    
It is enough to check that conditions (C\ref{cond:finitemaps})-(C\ref{cond:leb}) imply the hypotheses (H\ref{hyp:amb}-H\ref{hyp:dec}, H\ref{hyp:return}-H\ref{hyp:param}) of section \ref{sec:wrksetup}.

Here we check just (H\ref{hyp:return}) and the rest are left for the reader (who should choose $B_R(y_k^{\omega,n}) \equiv (0,1), \mathfrak{d}=0, \stackrel{-}{\mathcal{C}} \hspace{-3mm} \phantom{\mathcal{C}}_n^\omega, \stackrel{+-}{\mathcal{C}} \hspace{-4mm} \phantom{\mathcal{C}}_n^\omega \equiv \emptyset, d_0,d_1=1, \kappa \in \mathbb{R}_{>1}, \mathfrak{p} \in \mathbb{R}_{>1}$). 

We start calculating $\alpha_\ell$'s. Consider $\ell \ge 1$ and $\omega \in \Omega$ (eventually taken in a set of full measure). 

Consider
\begin{equation}\label{eq:Lomegalarge}
    L \ge M_{\ell \wedge K(\omega)}(\omega).
\end{equation}

Then take $\rho_0(\omega,L)=\rho_0(\pi_0(\omega),\ldots,\pi_L(\omega))$ small enough so that $\rho \le \rho_0(\omega,L)$ implies

\begin{equation}\label{eq:smallrho1}
    T_\omega^i B_\rho(x(\omega)) \cap B_\rho(x(\theta^i \omega)) = \emptyset \text{, } \forall i \in [1,L] \setminus \{M_k(\omega) : k \in [1,K(\omega)]\},
\end{equation}which can be guaranteed noticing that \begin{enumerate}[label=\alph*)]
    \item returns occur precisely in the instants $\{M_k(\omega) : k \in [1,K(\omega)]\}$ and not in between (by minimality),
    \item $T_\omega^i$ is continuous on $x(\omega)$ ($\forall i \ge 1$), a.s., because, by (C\ref{cond:goodtarget}), one has $$x(\omega) \in \{x_0,x_1\} \subset \bigcap_{l=1}^\infty \mathcal{A}^\omega_{l} \subset \mathcal{A}^\omega_{i}\text{, a.s.}$$
\end{enumerate}

Because of the previous constraint, one could have started with $L$'s of the form $L = M_{q_L \wedge K(\omega)}(\omega)$, $q_L \ge \ell$ (so still satisfying equation (\ref{eq:Lomegalarge})), in the sense that other choices of $L$ are superfluous from the viewpoint of the quantity we will study, $Z^{\omega,L}_{\Gamma_\rho}$. Then one could restrict $\rho_0(\omega,L)$ further so that $\rho \le \rho_0(\omega,L)$ implies:
\begin{equation}\label{eq:smallrho2}
 T_{\theta^{M_{k'}(\omega)}\omega }^{ \overbrace{\scriptstyle M_{k}(\omega) - M_{k'}(\omega)}^{M_{k-k'}\left(\theta^{M_{k'}(\omega)} \omega \right)} } B_\rho(x(\theta^{M_{k'}(\omega)}\omega)) \subset A^{\theta^{M_{k}(\omega)}\omega}_{M_{q_L \wedge K(\omega)}(\omega) - M_k(\omega) }\left(x(\theta^{M_{k}(\omega)}\omega)\right),\hspace{-0.5mm}
\substack{\displaystyle \forall k',k \in [0,q_L \wedge K(\omega)]\\ \displaystyle k' \le k},
\end{equation}
which can be guaranteed noticing that
\begin{enumerate}[label=\alph*)]
    \item $T_{\theta^{M_{k'}(\omega)}\omega }^{ \scriptstyle M_{k-k'}\left(\theta^{M_{k'}(\omega)} \omega \right) } x(\theta^{M_{k'}(\omega)}\omega) = x(\theta^{M_{k-k'}\left(\theta^{M_{k'}(\omega)} \omega \right)} \theta^{M_{k'}  (\omega)}\omega) = x(\theta^{M_k(\omega)}\omega)$, with the later in $\{x_0,x_1\} \stackrel{(C\ref{cond:goodtarget})}{\subset} \bigcap_{l=1}^\infty \mathcal{A}^{\theta^{M_k(\omega)}\omega}_l \subset \mathcal{A}^{\theta^{M_k(\omega)}\omega}  _{M_{q_L \wedge K(\omega)}(\omega) - M_k(\omega) }$,
    \item $T_{\theta^{M_{k'}(\omega)}\omega }^{ \scriptstyle M_{k-k'}\left(\theta^{M_{k'}(\omega)} \omega \right) }$ is continuous at $x(\theta^{M_{k'}(\omega)}\omega)$,  because, again by  (C\ref{cond:goodtarget}), one has $x(\theta^{M_{k'}(\omega)}\omega) \in \mathcal{A}^{\theta^{M_{k'}(\omega)}\omega }_{M_{k-k'}\left(\theta^{M_{k'}(\omega)} \omega \right)}$.
\end{enumerate} 

The point with condition (\ref{eq:smallrho2}) is to say that, $\rho$ is so small that, starting from any pre-intermediary time $M_{k'}(\omega)$ and going to any post-intermediary step $M_k(\omega)$, the initial $\rho$-sized ball grows under iteration up to time $M_k(\omega)$ but still fitting inside a partition domain (thus an injectivity domain) of the map evolving from time $M_k(\omega)$ until the end, $M_{q_L \wedge K(\omega)}$. In particular, the image balls won't break injectivity (or wrap around). Most importantly, it is implied that for any $z \in B_\rho(x(\omega))$:
\begin{equation*}
    \left(I_{0}^{\omega,\rho}(z), I_{M_1(\omega)}^{\omega,\rho}(z), \ldots, I_{M_{q_L \wedge K(\omega)}(\omega)}^{\omega,\rho}(z)\right)
\end{equation*}is a binary sequence starting with a batch of $1$'s followed by a (possibly degenerate) batch of $0$'s (e.g. $11100$, $1111$ or $10000$).

Then, for $\omega$, $L$ and $\rho$ as above, one has:
\begin{equation*}
    \hat\alpha^\omega_\ell(L,\rho) \mu_\omega(\Gamma_\rho(\omega)) = \mu_\omega(Z^{\omega,L-1}_{*\Gamma_\rho} \ge \ell -1 , I^{\omega,\rho}_0=1) 
\end{equation*}
\begin{equation*}
    \stackrel[]{\text{(\ref{eq:smallrho1})}}{=} \mu_\omega\left(\sum_{j \in \{M_k(\omega) : k \in [1,q_L \wedge K(\omega)]\} } I^{\omega,\rho}_j \ge \ell -1, I^{\omega,\rho}_0=1 \right)
\end{equation*}
\begin{equation*}
    \stackrel[]{\text{(\ref{eq:smallrho2})}}{=} \begin{cases}
        \mu_\omega\left(I_0^{ \omega,\rho}=1,  I^{\omega,\rho}_{M_1(\omega)} =1, \ldots, I^{\omega,\rho}_{M_{\ell-1}(\omega)}=1  \right)\text{, if }\ell-1 \le K(\omega) \\ 0\hspace{74mm}\text{, otherwise}
    \end{cases}
\end{equation*}
\begin{equation*}
    \stackrel[]{\text{(\ref{eq:smallrho2})}}{=} \begin{cases}
        \mu_\omega\left(  I^{\omega,\rho}_{M_{\ell-1}(\omega)}=1  \right)\text{, if }\ell-1 \le K(\omega) \\ 0\hspace{30mm}\text{, otherwise}
    \end{cases},
\end{equation*}so that
\begin{equation*}
\resizebox{1\hsize}{!}{
\begin{tabular}{@{}c @{}}
    $\displaystyle \alpha^\omega_\ell(L,\rho)  \frac{\mu_\omega(\Gamma_\rho(\omega))}{\hat\mu(\Gamma_\rho)} \stackrel[]{\text{(\ref{eq:deftildealphaomegainner2})}}{=} \begin{cases} \displaystyle \frac{\displaystyle \mu_\omega\left( (T_\omega^{M_{\ell-1}(\omega)})^{-1} \Gamma_\rho(\theta^{M_{\ell -1}(\omega)} \omega) \right)}{\displaystyle \hat\mu(\Gamma_\rho)}
    - \frac{\displaystyle \mu_\omega\left( (T_\omega^{M_{\ell}(\omega)})^{-1} \Gamma_\rho(\theta^{M_{\ell}(\omega) }\omega) \right)}{\displaystyle \hat\mu(\Gamma_\rho)}\text{, if } \ell \le K(\omega),\\
    \frac{\displaystyle \mu_\omega\left( (T_\omega^{M_{\ell-1}(\omega)})^{-1} \Gamma_\rho(\theta^{M_{\ell -1}(\omega)}\omega) \right)}{\displaystyle \hat\mu(\Gamma_\rho)} \hspace{55mm} \text{, if } \ell = K(\omega)+1,\\
    0 \hspace{109mm}\text{, if }\ell \ge K(\omega)+2.
    \end{cases}$ 
\end{tabular}}
\end{equation*}

Notice that
\begin{equation*}
    \frac{\displaystyle \mu_\omega\left( (T_\omega^{M_{\ell-1}(\omega)})^{-1} \Gamma_\rho(\theta^{M_{\ell -1}(\omega)} \omega) \right)}{\displaystyle \hat\mu(\Gamma_\rho)} = \frac{\operatorname{Leb}\left(h_\omega \mathbbm{1}_{(T_\omega^{M_\ell(\omega)})^{-1} \Gamma_\rho(\theta^{M_\ell(\omega)}\omega ) } \right)}{\int_\Omega {\operatorname{Leb}(h_\omega \mathbbm{1}_{\Gamma_\rho(\omega)})} d \mathbb{P}(\omega)}
\end{equation*}
\begin{equation*}
    = \frac{[h_\omega(x(\omega))+ \mathcal{O}(\epsilon)] \operatorname{Leb}\left( (T_\omega^{M_\ell(\omega)})^{-1} B_\rho(x(\theta^{M_\ell(\omega)}\omega ))  \right)}{\int_\Omega [h_\omega(x(\omega))+ \mathcal{O}(\epsilon)] {\operatorname{Leb}( B_\rho(x(\omega)))} d \mathbb{P}(\omega)}
\end{equation*}
\begin{equation*}
    = \frac{[h_\omega(x(\omega))+ \mathcal{O}(\epsilon)] \Big[\big(JT_\omega^ {M_\ell(\omega)}(x(\omega))\big)^{-1}+\mathcal{O}(\epsilon)\Big] \operatorname{Leb}\left( B_\rho(x(\theta^{M_\ell(\omega)}\omega ))  \right)}{\int_\Omega [h_\omega(x(\omega))+ \mathcal{O}(\epsilon)] {\operatorname{Leb}( B_\rho(x(\omega)))} d \mathbb{P}(\omega)}
\end{equation*}
\begin{equation}\label{eq:jacobian}
    = \frac{h_\omega(x(\omega))+ \mathcal{O}(\epsilon) }{\int_\Omega h_\omega(x(\omega))+ \mathcal{O}(\epsilon) d \mathbb{P}(\omega)} \Big[\big(JT_\omega^ {M_\ell(\omega)}(x(\omega))\big)^{-1}+\mathcal{O}(\epsilon)\Big]
\end{equation}
where, given $\epsilon>0$ (for $\omega$ and $L$ chosen as above), we've considered $\rho \le \rho_1({\omega}{,}{\epsilon})< r$ (see (C\ref{cond:leb})), with $\rho_1(\omega,\epsilon)$ small enough so that for any $\rho \le \rho_1(\omega,\epsilon)$:
\begin{equation*}
    h_\omega(z) = h_\omega(x(\omega)) + \mathcal{O}(\epsilon), \text{ } \forall z \in B_\rho(x(\omega))
\end{equation*}and
\begin{equation*}
    \big(JT_\omega^ {M_\ell(\omega)}(z)\big)^{-1} = \big(JT_\omega^ {M_\ell(\omega)}(x(\omega))\big)^{-1}+\mathcal{O}(\epsilon), \text{ } \forall z \in B_\rho(x(\omega)).
\end{equation*}
We can write 
\begin{equation*}
    \rho_1(\omega,\epsilon) = \left(\epsilon / H_\beta(h_\omega|_{B_r(x(\omega))})\right)^{\nicefrac{1}{\beta}} \wedge \left(\epsilon / H_\beta\left([JT_\omega^{M_{\ell}(\omega)}]^{-1}|_{B_r(x(\omega))}\right)\right)^{\nicefrac{1}{\beta}} \wedge 1.
\end{equation*}

We can use (C\ref{cond:finitemaps}) (finitely many maps and uniformly bounded second derivatives), (C\ref{cond:goodtarget}) (uniformly bounded finite-periods) and (C\ref{cond:leb}) (uniform Holder constants for the densities) to pass to controls that are uniform on $\omega$ and then integrate: for any $\epsilon >0$, $L \ge L_* := \ell M_\Gamma$ and $$\rho \le \rho_*(L,\epsilon) := \min_{\substack{(v_0,\ldots,v_L) \\ \in\ \{0,1\}^{L+1}}} \rho_1(v_0,\ldots,v_L)  \wedge  \essinf_\omega \rho_1(\omega,\epsilon) \in (0,1] ,$$ one has
\begin{equation*}
\resizebox{1\hsize}{!}{
\begin{tabular}{@{}l @{}}
 $\displaystyle {\Huge \alpha_\ell(L,\rho) =}$ \\$ {\mybigint}_{\hspace{-1mm}\scalebox{1.1}{$\Omega$}} \hspace{2mm}\begin{cases}
    \frac{\displaystyle h_\omega(x(\omega)) + \mathcal{O}(\epsilon)}{\displaystyle \int_\Omega h_\omega(x(\omega)) + \mathcal{O}(\epsilon) d \mathbb{P}(\omega) } \left[ \left(JT_\omega^{M_{\ell-1}(\omega)}(x(\omega))\right)^{-1} + \mathcal{O}(\epsilon) - \left(JT_\omega^{M_{\ell}(\omega)}(x(\omega))\right)^{-1} - \mathcal{O}(\epsilon)\right] \text{, if }\ell \le K(\omega) \\
        \frac{\displaystyle h_\omega(x(\omega)) + \mathcal{O}(\epsilon)}{\displaystyle \int_\Omega h_\omega(x(\omega)) + \mathcal{O}(\epsilon) d \mathbb{P}(\omega) } \left[ \left(JT_\omega^{M_{\ell-1}(\omega)}(x(\omega))\right)^{-1} + \mathcal{O}(\epsilon)\right] \hspace{53mm}  \text{, if }\ell=K(\omega)+1 \\
        0\hspace{155mm}\text{, if }\ell \ge K(\omega)+2
    \end{cases} d\mathbb{P}(\omega),$
\end{tabular}}
\end{equation*}
then taking iterated limits of the type $\lim_{\epsilon} \lim_{L} \overline{\varliminf}_{\rho}$ one finds that
\begin{equation}\label{eq:calculatedalphas}
\resizebox{1\hsize}{!}{
\begin{tabular}{@{}c @{}}
$\displaystyle \alpha_\ell = {\mybigint}_{\hspace{-1mm}\scalebox{1.1}{$\Omega$}} \hspace{2mm}\begin{cases}
    \frac{\displaystyle h_\omega(x(\omega))}{\displaystyle \int_\Omega h_\omega(x(\omega)) d \mathbb{P}(\omega) } \left[ \left(JT_\omega^{M_{\ell-1}(\omega)}(x(\omega))\right)^{-1}  - \left(JT_\omega^{M_{\ell}(\omega)}(x(\omega))\right)^{-1} \right] \text{, if }\ell \le K(\omega) \\
        \frac{\displaystyle h_\omega(x(\omega)) }{\displaystyle \int_\Omega h_\omega(x(\omega))  d \mathbb{P}(\omega) } \left[ \left(JT_\omega^{M_{\ell-1}(\omega)}(x(\omega))\right)^{-1} \right] \hspace{40mm}  \text{, if }\ell=K(\omega)+1 \\
        0\hspace{115mm}\text{, if }\ell \ge K(\omega)+2
    \end{cases} d\mathbb{P}(\omega).$
\end{tabular}}
\end{equation}

The following diagram helps one to visualize how the integrand in equation (\ref{eq:calculatedalphas}), with the factor $\frac{h_\omega(x(\omega))}{\int_\Omega h_\omega(x(\omega)) d\mathbb{P}(\omega)}$ suppressed, changes
\begin{enumerate}[label=\alph*)]
    \item for $\omega$'s with varying amounts of periodicity (read the different lines),
    \item as $\ell$ grows (read the different columns).
\end{enumerate}

\begin{equation}\label{eq:alphacasebycase}
\resizebox{1\hsize}{!}{
\begin{tabular}{r @{} c @{} c @{} c @{} c @{} c @{} c @{} l}
     &  & $\ell=1$ & & $\ell=2$ & & $\ell=3$ \\
      $K(\omega) = \infty $ & $: \Bigg($ & $1 - 1/ JT^{m_0 (\omega)}_{\omega } (x(\omega))$&${,}$&$\frac{\displaystyle 1- 1/JT^{m_1 (\omega)}_{\theta^{m_0(\omega)}\omega } (x(\omega))}{\displaystyle JT^{m_0 (\omega)}_{\omega } (x(\omega))}$&${,}$&$\frac{\displaystyle 1- 1/JT^{m_2 (\omega)}_{\theta^{m_0(\omega)+m_1(\omega)}\omega } (x(\omega))}{\displaystyle JT^{m_0 (\omega)}_{\omega } (x(\omega)) JT^{m_1 (\omega)}_{\theta^{m_0(\omega)}\omega } (x(\omega))}$&$, \ldots \Bigg)$  \\
      $K(\omega)=0$  & $: ($ & $1$&${,}$&$0$&${,}$&$0$&$,\stackrel{\bar{0}}{\ldots})$\\
      $K(\omega)=1$  & $: \Bigg($ & $1 - 1/ JT^{m_0 (\omega)}_{\omega } (x(\omega))$&${,}$&$ \frac{\displaystyle1}{\displaystyle JT^{m_0 (\omega)}_{\omega } (x(\omega))}$&${,}$&$0$&$,\stackrel{\bar{0}}{\ldots}\Bigg)$\\
      $K(\omega)=2$  & $: \Bigg($ & $1 - 1/ JT^{m_0 (\omega)}_{\omega } (x(\omega))$&${,}$&$\frac{\displaystyle 1- 1/JT^{m_1 (\omega)}_{\theta^{m_0(\omega)}\omega } (x(\omega))}{\displaystyle JT^{m_0 (\omega)}_{\omega } (x(\omega))}$&${,}$&$\frac{\displaystyle 1 }{\displaystyle JT^{m_0 (\omega)}_{\omega } (x(\omega)) JT^{m_1 (\omega)}_{\theta^{m_0(\omega)}\omega } (x(\omega)) }$&$,\stackrel{\bar{0}}{\ldots}\Bigg).$\\
    \end{tabular}}
\end{equation}

Having found that $\alpha_\ell$'s exist and have explicit representation, it remains to check that $\alpha_1 > 0$ and $\sum_{\ell=1}^\infty \ell^2 \alpha_\ell < \infty$.

It holds that $\alpha_1 > 0$ because the quantity found in the first column of diagram (\ref{eq:alphacasebycase}) is bounded below by $1- 1/d_{\min}>0$.

Moreover, \hspace{-0.25mm}considering \hspace{-0.25mm}the \hspace{-0.25mm}integrand \hspace{-0.25mm}of \hspace{-0.25mm}equation (\ref{eq:calculatedalphas}), \hspace{-0.25mm}we \hspace{-0.25mm}see \hspace{-0.25mm}that \hspace{-0.25mm}$\alpha_\ell$ \hspace{-0.25mm}is \hspace{-0.25mm}at \hspace{-0.25mm}most \hspace{-0.25mm}$(\nicefrac{1}{d_{\min}})^{\ell-1}$, therefore
\begin{equation*}
    \sum_{\ell=1}^\infty \ell^2 \hat\alpha_\ell \le \sum_{\ell=1}^\infty \ell^2 (\nicefrac{1}{{d_{\min}}})^{\ell-1}<\infty,
\end{equation*}
since $d_{\min} >1$.

This concludes that conditions (C\ref{cond:finitemaps})-(C\ref{cond:goodtarget}) imply the hypotheses of theorem \ref{thm:main} and that the associated $\alpha_\ell$'s satisfy (H\ref{hyp:return}) and the hypotheses of theorem \ref{thm:lambdaalpha}.

Let us finally notice that in this case, where $d_0,d_1,\eta,\beta =1$ and $\mathfrak{p}=
\infty$ (i.e., can be taken arbitrarily large), $q(d_0,d_1,\eta,\beta,\mathfrak{p})$, reduces to $1$. This is because the system of inequalities appearing at end of proof of lemma \ref{lem:variance}  reduces to only two ($1>\alpha$ and $w > 2$ for $(\alpha,w) \in (0,1) \times (1,\infty)$) which admit a solution that opens a margin of (at least) $1$ in both equations.
\end{proof}

\begin{remark}\label{rmk:beta}
    As it comes to $M=[0,1]$, the use of surjective branches in (C\ref{cond:finitemaps}) was to facilitate as much as possible the presentation of covers and cylinders in (H\ref{hyp:inv}) below. But these can be still presented without surjective branches. For example, one could present them for the beta maps $T_0(x) = 1/2 + 2x \text{ (mod 1)}$ and $T_1(x) = 1/2 + 3x \text{ (mod 1)}$. On the other hand, to have the type of decay against Lipschitz test functions we will be after in (H\ref{hyp:dec}), the interval maps ought to have subjective branches (otherwise the good functional space becomes bounded variation instead of Lipschitz), which is not the case of the previous beta maps. In this situation, one has to resort to seeing these beta maps as acting smoothly in $M=S^1$, and cylinders will not anymore mark regions of continuity/differentiability, but will still mark injective regions.
\end{remark} 

\begin{remark}\label{rmk:cones}
    Condition (C\ref{cond:leb}) was included to make transparent what is really used in the argument above. But one should be aware that conditions (C\ref{cond:finitemaps}-C\ref{cond:drivingmeas}) suffice to conclude that densities are a.s. bounded away from $0$ and $\infty$ and a.s. admit a uniform Holder constant (on the entire manifold $M$). See \cite{rousseau2014exponential} Example 21. This is stronger than (C\ref{cond:leb}), which then can, technically, be omitted from the list of conditions.
\end{remark}

Now we concentrate on analyzing the these conclusions of theorem \ref{thm:expanding} refine (or how $\alpha_\ell$'s in equation (\ref{eq:calculatedalphas}) simplify) when additional conditions are considered.

\begin{cor} Consider the assumptions of theorem \ref{thm:expanding} and assume further that $K(\omega)=0$ a.s.. 

Then \begin{equation}
    \alpha_\ell = \begin{cases}
        1 \text{, if }\ell =1 \\ 0 \text{, if }\ell\ge 2
    \end{cases},
\end{equation}and CPD in the limit theorem boils down to a standard Poisson.
\end{cor}

\begin{proof}
Immediate.
\end{proof}

\begin{cor}Consider the assumptions of theorem \ref{thm:expanding} and assume further that $\mathbb{P}$ is Bernoulli, $K(\omega)=\infty$ a.s. and $$h_\omega(x(\omega)) \perp \left( JT_{\theta^{M_j(\omega)}\omega}^{m_j(\omega)} (x(\theta^{M_j(\omega)}\omega ))\right)_j.\footnotemark$$ \footnotetext{This occurs when, for example, when $h_\omega \equiv 1$ a.s., or much more generally, when $h_\omega$ depends only on the past entries of $\omega$ (see, e.g., \cite{ledrappier1988entropy} prop. 1.2.3 and \cite{kifer2006random} prop. 3.3.2).}

Then $$\alpha_\ell = (D-1)D^{-\ell},\text{ with }D^{-1} := \int_\Omega [ JT_\omega^{m_0(\omega)} x(\omega)]^{-1} d \mathbb{P}(\omega),$$and the CPD in the limit theorem boils down to a Polya-Aeppli (or geometric) one.
\end{cor}

\begin{proof}
Notice that  $K(\omega)=
\infty$ a.s. and the independence of $h_\omega(x(\omega))$ from the rest implies
\begin{equation*}
    \alpha_\ell = \int_\Omega \prod_{j=0}^{\ell-2} \left[ 
JT_{\theta^{M_j(\omega)}\omega}^{m_j(\omega)} (x(\theta^{M_j(\omega)}\omega )) \right]^{-1} d \mathbb{P}(\omega) - \int_\Omega \prod_{j=0}^{\ell-1} \left[ 
JT_{\theta^{M_j(\omega)}\omega}^{m_j(\omega)} (x(\theta^{M_j(\omega)}\omega )) \right]^{-1} d \mathbb{P}(\omega),
\end{equation*}then, after we make the point in I) that $\left( \omega \mapsto JT_{\theta^{M_j(\omega)}\omega}^{m_j(\omega)} (x(\theta^{M_j(\omega)}\omega ))\right)_j$ is independent under $\mathbb{P}$, we will find that
\begin{equation*}
    \alpha_\ell = \prod_{j=0}^{\ell-2} \int_\Omega \left[ 
JT_{\theta^{M_j(\omega)}\omega}^{m_j(\omega)} (x(\theta^{M_j(\omega)}\omega )) \right]^{-1} d \mathbb{P}(\omega) - \prod_{j=0}^{\ell-1} \int_\Omega  \left[ 
JT_{\theta^{M_j(\omega)}\omega}^{m_j(\omega)} (x(\theta^{M_j(\omega)}\omega )) \right]^{-1} d \mathbb{P}(\omega),
\end{equation*}which, we will argue in II), equals
\begin{equation*}
    \alpha_\ell = \prod_{j=0}^{\ell-2} \int_\Omega \left[ JT_\omega^{m_0(\omega)} x(\omega)\right]^{-1} d \mathbb{P}(\omega) - \prod_{j=0}^{\ell-1} \int_\Omega \left[ JT_\omega^{m_0(\omega)} x(\omega)\right]^{-1} d \mathbb{P}(\omega) = (D-1) D^{-\ell},
\end{equation*}where $D^{-1} := \int_\Omega [ JT_\omega^{m_0(\omega)} x(\omega)]^{-1} d \mathbb{P}(\omega)$, as desired.

Let us make the points that are missing. 

I) Notice first that
\begin{equation*}
    \mathbb{P}(m_0(\omega)=i_0, m_1(\omega)=i_1) = \mathbb{P}(m_0(\omega)=i_0, m_0(\theta^{i_0}\omega)=i_1) = \mathbb{P}(\mathbbm{1}_{\operatorname{Per}_{i_0}(\Gamma)} \mathbbm{1}_{\theta^{-i_0} \operatorname{Per}_{i_1}(\Gamma)} )
\end{equation*}
\begin{equation*}
    = \mathbb{P}(\mathbbm{1}_{\operatorname{Per}_{i_0}(\Gamma)}) \mathbb{P}(\mathbbm{1}_{\theta^{-i_0} \operatorname{Per}_{i_1}(\Gamma)} ) = \mathbb{P}(m_0(\omega)=i_0) \mathbb{P}( m_0(\omega)=i_1),
\end{equation*}where the first equality in the second line is because $(\pi_j)$'s are independent under $\mathbb{P}$ and the indicator functions can be expressed in terms of disjoint blocks of $(\pi_j)$'s, namely $\pi_0,\ldots, \pi_{i_0-1}$ and $\pi_{i_0},\ldots, \pi_{i_0+i_1-1}$. On the other hand
\begin{equation*}
    \mathbb{P}( m_1(\omega)=i_1) = \sum_{i_0} \mathbb{P}( m_0(\omega)=i_0, m_1(\omega)=i_1) 
\end{equation*}
\begin{equation*}
    = \sum_{i_0} \mathbb{P}(m_0(\omega)=i_0) \mathbb{P}( m_0(\omega)=i_1) = \mathbb{P}( m_0(\omega)=i_1).
\end{equation*}So combining the two previous chains of equality, we find that $m_0$ and $m_1$ are independent, i.e., $m_0 \perp m_1$.

Once again, since $(\pi_j)_j$ is an independency under $\mathbb{P}$, whenever two random variables $X$ and $Y$ can be expressed as $X = \phi \circ (\pi_0,\ldots,\pi_{i_0-1})$ and $Y = \psi \circ (\pi_{i_0},\ldots,\pi_{i_0+i_1-1})$, then $X \perp Y$. Similarly for $\pi$ instead of $\pi$. This is the case for $(JT_\cdot^{i_0}(x(\cdot)), \mathbbm{1}_{m_0(\cdot)=i_0}) \perp (JT_{\theta^{i_0}\cdot}^{i_1}(x \circ \theta^{i_0}(\cdot)),\mathbbm{1}_{m_1(\cdot)=i_1})$. 

Therefore
\begin{equation*}
\resizebox{1\hsize}{!}{
\begin{tabular}{@{}c @{}}
$\displaystyle \mathbb{P}\left( \left\{\omega: \big[JT_\omega^{m_0(\omega)} (x(\omega)) \big]^{-1} = a, \big[JT_{\theta^{m_0(\omega)}\omega}^{m_1(\omega)} (x(\theta^{m_0(\omega)}\omega)) \big]^{-1} = b   \right\} \right)$\\
$\displaystyle =\sum_{i_0} \sum_{i_1} \mathbb{P}\left( \left\{\omega: \big[JT_\omega^{i_0} (x(\omega)) \big]^{-1} = a, \big[JT_{\theta^{i_0}\omega}^{i_1} (x(\theta^{i_0}\omega)) \big]^{-1} = b, m_0(\omega)=i_0, m_0(\theta^{i_0}\omega)=i_1   \right\} \right)$\\
$\displaystyle =\sum_{i_0} \sum_{i_1} \left[\mathbb{P}\left( \left\{\omega: \big[JT_\omega^{i_0} (x(\omega)) \big]^{-1} = a, m_0(\omega)=i_0 \right\}\right) \mathbb{P}\left(\left\{\omega:\big[JT_{\theta^{i_0}\omega}^{i_1} (x(\theta^{i_0}\omega)) \big]^{-1} = b, m_0(\theta^{i_0}\omega)=i_1   \right\} \right)\right]$\\
$\displaystyle =\left[\sum_{i_0} \mathbb{P}\left( \left\{\omega: \big[JT_\omega^{i_0} (x(\omega)) \big]^{-1} = a, m_0(\omega)=i_0 \right\}\right)\right] \left[\sum_{i_1} \mathbb{P}\left(\left\{\omega:\big[JT_{\omega}^{i_1} (x(\omega)) \big]^{-1} = b, m_0(\omega)=i_1   \right\} \right)\right]$\\
$\displaystyle =\mathbb{P}\left( \left\{\omega: \big[JT_\omega^{m_0(\omega)} (x(\omega)) \big]^{-1} = a \right\}\right) \mathbb{P}\left(\left\{ \omega: \big[JT_{\omega}^{m_0(\omega)} (x(\omega)) \big]^{-1} = b   \right\} \right).$
\end{tabular}}
\end{equation*}
On the other hand 
\begin{equation*}
    \mathbb{P}\left(\left\{\omega: \big[ JT_{\theta^{m_0(\omega)}\omega}^{m_1(\omega)}(x(\theta^{m_0(\omega)} \omega)) \big]^{-1} =b \right\} \right)
\end{equation*}
\begin{equation*}
    = \sum_a \mathbb{P}\left(\left\{\omega: \big[ JT_{\omega}^{m_0(\omega)}(x(\omega)) \big]^{-1} =a, \big[ JT_{\theta^{m_0(\omega)}\omega}^{m_1(\omega)}(x(\theta^{m_0(\omega)} \omega)) \big]^{-1} =b \right\} \right)
\end{equation*}
\begin{equation*}
    = \sum_a \mathbb{P}\left( \left\{\omega: \big[JT_\omega^{m_0(\omega)} (x(\omega)) \big]^{-1} = a \right\}\right) \mathbb{P}\left(\left\{ \omega: \big[JT_{\omega}^{m_0(\omega)} (x(\omega)) \big]^{-1} = b   \right\} \right)
\end{equation*}
\begin{equation*}
    = \mathbb{P}\left(\left\{ \omega: \big[JT_{\omega}^{m_0(\omega)} (x(\omega)) \big]^{-1} = b   \right\} \right).
\end{equation*}So combining the two previous chains of equality, we find that $$JT_\cdot^{m_0(\cdot)}(x(\cdot)) \perp JT_{\theta^{m_0(\cdot)}\cdot}^{m_1(\cdot)}(x(\theta^{m_0(\cdot)}\cdot)),$$as desired.

II) Notice that 
\begin{equation*}
    \int_\Omega \big[ JT_{\theta^{m_0(\omega)}\omega}^{m_1(\omega)}x(\theta^{m_0(\omega)}\omega) \big]^{-1} d \mathbb{P} (\omega) = \sum_b b \mathbb{P}\left( \left\{\omega : JT_{\theta^{m_0(\omega)}\omega}^{m_1(\omega)}x(\theta^{m_0(\omega)}\omega) \big]^{-1} = b \right\}\right)
\end{equation*}
\begin{equation*}
    = \sum_b b  \mathbb{P}\left(\left\{ \omega: \big[JT_{\omega}^{m_0(\omega)} (x(\omega)) \big]^{-1} = b   \right\} \right) = \int_\Omega \big[ JT_{\omega}^{m_0(\omega)}x(\omega) \big]^{-1} d \mathbb{P} (\omega),
\end{equation*}where we have used the last equality in I).    
\end{proof}

\section{Acknowledgements}
LA \hspace{-0.25mm}was \hspace{-0.25mm}supported \hspace{-0.25mm}by \hspace{-0.25mm}Funda\c{c}\~ao \hspace{-0.25mm}para \hspace{-0.25mm}a \hspace{-0.25mm}Ci\^encia \hspace{-0.25mm}e \hspace{-0.25mm}a \hspace{-0.25mm}Tecnologia \hspace{-0.25mm}grant \hspace{-0.25mm}PD/BD/150458/2019
and also by the Faculty of Sciences of the University of Porto, the Center of Mathematics of the University of Porto, the Universit\'e de Toulon, France, and the Centre de Physique
Th\'eorique in Luminy, Marseille, France
NH was supported by the Simons Foundation grant
{\it Collaboration Grants for Mathematicians} and also by the
Universit\'e de Toulon, France, and the Centre de Physique
Th\'eorique in Luminy, Marseille, France. The research of SV was supported by the project Dynamics and Information Research Institute within the agreement between UniCredit Bank and Scuola Normale Superiore di Pisa and by the Laboratoire International Associé LIA LYSM, of the French CNRS and INdAM (Italy). SV was also supported by the project MATHAmSud TOMCAT 22-Math-10, N. 49958WH, du french CNRS and MEAE and by the Mathematical Research Institute MATRIX for hosting a workshop during which part of this work was conceived. The authors thank P. Varandas for useful comments on applications and J. Freitas for discussions on alternative interpretations of relationships between return and hitting statistical quantities.

\newpage

\end{document}